\documentclass[11pt,a4paper]{amsart}
\usepackage{fullpage}
\usepackage[foot]{amsaddr}

\usepackage{amsmath,amsthm,amsfonts,amssymb,array,graphicx,epsfig,fancyhdr,stmaryrd,algpseudocode,mathtools}
\usepackage{bm}

\usepackage[scr=euler,scrscaled=1.00]{mathalfa}

\usepackage{enumerate}
\usepackage{cite}

\usepackage[english]{babel}
\usepackage[T1]{fontenc}
\usepackage[utf8]{inputenc}

\usepackage[usenames,dvipsnames,svgnames]{xcolor}

\usepackage[pdftitle={Stability of Low-Rank Tensor Representations and Structured Multilevel Preconditioning for Elliptic PDEs},
pdfauthor={Markus Bachmayr and Vladimir Kazeev}, 			%
colorlinks=true,                               %
linkcolor=Red,                                  %
citecolor=MidnightBlue,                                %
filecolor=Purple,                              %
urlcolor=VioletRed,                                  %
pdftex]{hyperref}

\usepackage{diagbox}

\usepackage{multicol}
\usepackage[autolanguage,addmissingzero]{numprint}

\newtheorem{theorem}{Theorem}[section]
\newtheorem{assumption}[theorem]{Assumption}
\newtheorem{corollary}[theorem]{Corollary}
\newtheorem{proposition}[theorem]{Proposition}
\newtheorem{lemma}[theorem]{Lemma}
\theoremstyle{definition}
\newtheorem{definition}[theorem]{Definition}
\newtheorem{example}[theorem]{Example}
\newtheorem{remark}[theorem]{Remark}

\numberwithin{equation}{section}

\newcommand{\N}{\ensuremath{{\mathbb N}}}
\newcommand{\Nz}{\N_0}

\newcommand{\R}{\ensuremath{{\mathbb R}}}

\newcommand{\cond}{\operatorname{cond}}

\DeclareMathOperator{\supp}{supp}

\DeclareMathOperator{\ran}{ran}
\DeclareMathOperator{\rank}{rank}
\DeclareMathOperator{\ranks}{ranks}

\DeclareMathOperator{\diag}{diag}

\DeclareMathOperator{\asm}{\tau}
\DeclareMathOperator{\ramp}{ramp}
\DeclareMathOperator{\rcond}{rcond}

\DeclareMathOperator{\opramp}{mramp}
\DeclareMathOperator{\oprcond}{mrcond}

\newcommand{\lorth}{\ensuremath{\operatorname{orth}^-}}
\newcommand{\rorth}{\ensuremath{\operatorname{orth}^+}}

\def\Chi{\raise .3ex\hbox{\large $\chi$}}

\providecommand{\abs}[1]{\lvert#1\rvert}

\providecommand{\norm}[1]{\lVert#1\rVert}
\providecommand{\bignorm}[1]{\bigl\lVert#1\bigr\rVert}
\providecommand{\Bignorm}[1]{\Bigl\lVert#1\Bigr\rVert}

\allowdisplaybreaks

\usepackage{tabularx}
\DeclareMathOperator*{\DProd}{\times}

\DeclareMathOperator*{\Comp}{\circ}

\DeclareMathOperator*{\TProd}{\otimes}
\DeclareMathOperator*{\KProd}{\otimes}
\DeclareMathOperator*{\KProdBig}{\bigotimes}
\DeclareMathOperator*{\MP}{\bullet}

\usepackage{tikz,tkz-linknodes}
\usetikzlibrary{arrows,shapes,fit}
\usetikzlibrary{positioning,decorations.pathreplacing,shapes,patterns}
\usetikzlibrary{calc}
\usetikzlibrary{intersections}

\DeclareMathOperator*{\RP}{\Join}

\DeclareMathOperator*{\essinf}{ess\,inf}

\newcommand{\SetMinus}{\mathbin{\mathpalette\rsetminusaux\relax}}
\newcommand{\rsetminusaux}[2]{\mspace{-4mu}
    \raisebox{\rsmraise{#1}\depth}{\rotatebox[origin=c]{-20}{$#1\smallsetminus$}}
  \mspace{-4mu}
}
\newcommand{\rsmraise}[1]{%
  \ifx#1\displaystyle .8\else
    \ifx#1\textstyle .8\else
      \ifx#1\scriptstyle .6\else
        .45%
      \fi
    \fi
  \fi}

\newcommand{\Restr}[1]{\raise-0.5ex\hbox{$\mid$}_{#1}}

\newcommand{\partialup}{\rotatebox[origin=c]{15}{\ensuremath{\partial}}\mkern-2mu}
\DeclareMathOperator{\Boundary}{\partialup\!}

\DeclareMathOperator{\Span}{span}

\newcommand{\DelimiterGroup}[4]{
	\ifcase#1\relax
		#2 #4 #3	%
	\or
		\bigl#2 #4 \bigr#3	%
	\or
		\Bigl#2 #4 \Bigr#3	%
	\or
		\biggl#2 #4 \biggr#3	%
	\or
		\Biggl#2 #4 \Biggr#3	%
	\or
	\or
	\or
		\mleft#2 #4 \mright#3	%
	\or
		\left#2 #4 \right#3	%
	\else
	\fi
}

\newcommand{\Par}[2][0]{\DelimiterGroup{#1}{(}{)}{#2}}
\newcommand{\SqBr}[2][0]{\DelimiterGroup{#1}{[}{]}{#2}}
\newcommand{\CuBr}[2][0]{\DelimiterGroup{#1}{\{}{\}}{#2}}

\newcommand{\Abs}[2][0]{\DelimiterGroup{#1}{\lvert}{\rvert}{#2}}
\newcommand{\IndNorm}[2][0]{\DelimiterGroup{#1}{\lvert}{\rvert}{#2}}
\newcommand{\Norm}[2][0]{\DelimiterGroup{#1}{\lVert}{\rVert}{#2}}

\newcommand{\Set}[2][0]{\DelimiterGroup{#1}{\{}{\}}{#2}}
\newcommand{\Tuple}[2][0]{\DelimiterGroup{#1}{(}{)}{#2}}

\newcommand{\IProd}[3][0]{\DelimiterGroup{#1}{\langle}{\rangle}{#2,#3}}

\newcommand{\IntOO}[3][0]{\DelimiterGroup{#1}{(}{)}{#2,#3}}

\newcommand{\Card}[2][0]{\DelimiterGroup{#1}{\lvert}{\rvert}{#2}}

\newcommand{\SSp}[2][]{\mathrm{H}^{#2}_{#1}}
\newcommand{\LSp}[2][]{\mathrm{L}^{#2}_{#1}}
\newcommand{\lSp}[2][]{\ell^{#2}_{#1}}

\newcommand{\MT}{\mathsf{T}\!}

\newcommand{\CQ}{,{\:}\!}
\newcommand{\QQ}{{\:}\!}

\newcommand{\PartNode}[2]{\hat{\tau}_{#1 \CQ #2}}
\newcommand{\PartInt}[2]{\hat{\varOmega}_{#1 \CQ #2}}
\newcommand{\PartMap}[2]{\hat{\phi}_{#1 \CQ #2}}
\newcommand{\PartMapD}[2]{\phi_{#1 \CQ #2}}

\newcommand{\vphi}[2]{\hat{\varphi}_{#1 \CQ #2}}
\newcommand{\vphiD}[2]{\varphi_{#1 \CQ #2}}
\newcommand{\cJ}{\mathcal{J}}

\newcommand{\RefEl}[2]{\varOmega_{#1 \CQ #2}}

\newcommand{\BR}[3][0]{#2^{\DelimiterGroup{#1}{[}{]}{#3}}}
\newcommand{\BM}[3][0]{#2^{\DelimiterGroup{#1}{\{}{\}}{#3}}}

\newcommand{\TTset}{\mathrm{TT}}

\newcommand{\Mat}[1]{\mathrm{U}_{#1}}

\renewcommand{\Vec}[1]{
	\boldsymbol{#1}
}
\newcommand{\Ten}[1]{
	\boldsymbol{#1}
}
\newcommand{\Core}[1]{#1}

\newcommand{\Rep}[1]{
	\mathsf{#1}
}

\DeclareSymbolFont{stmry}{U}{stmry}{m}{n}
\DeclareMathSymbol\llparenthesis\mathop{stmry}{"4C}
\DeclareMathSymbol\rrparenthesis\mathop{stmry}{"4D}

\newcommand{\cI}{\mathcal{I}}
\newcommand{\cV}{\mathcal{V}}

\newcommand{\bB}{\Ten{B}}
\newcommand{\bA}{\Ten{A}}
\newcommand{\bC}{\Ten{C}}

\newcommand{\bu}{\Vec{u}}
\newcommand{\bv}{\Vec{v}}
\newcommand{\bw}{\Vec{w}}
\newcommand{\bbf}{\Vec{f}}
\newcommand{\bg}{\Vec{g}}

\newcommand{\cS}{\mathcal{S}}
\newcommand{\cO}{\mathcal{O}}

\usepackage[section]{algorithm}
\usepackage{algpseudocode}
\algrenewcommand\algorithmicrequire{\textbf{input:}}
\algrenewcommand\algorithmicensure{\textbf{output:}}
\algrenewcommand\algorithmicfunction{\textbf{function}}
\algrenewcommand\algorithmicwhile{\textbf{while}}
\algrenewcommand\algorithmicdo{}
\algrenewcommand\algorithmicend{\textbf{end}}
\algrenewcommand\algorithmicforall{\textbf{for all}}
\algrenewcommand\algorithmicfor{\textbf{for}}
\algrenewcommand\algorithmicrepeat{\textbf{repeat}}
\algrenewcommand\algorithmicuntil{\textbf{until}}
\algrenewcommand\algorithmicif{\textbf{if}}
\algrenewcommand\algorithmicelse{\textbf{else}}
\algrenewcommand\algorithmicthen{\textbf{then}}

\newcommand{\stsolve}{{\textsc{STSolve}}}
\newcommand{\amen}{{\textsc{AMEn}}}

\usepackage[shortlabels]{enumitem}
\setlist{nolistsep}

\title[Running Title]{Stability of Low-Rank Tensor Representations\\ and Structured Multilevel Preconditioning\\ for Elliptic PDEs}
\date{\today}
\author{Markus Bachmayr$^1$}
\address{\rm $^1$ Institut f\"ur Mathematik, Johannes Gutenberg-Universit\"at Mainz, Staudingerweg 9, 55128 Mainz, Germany}
\email[Markus Bachmayr]{bachmayr@uni-mainz.de}
\thanks{M.B.\ acknowledges support by the Hausdorff Center of Mathematics, University of Bonn}

\author{Vladimir Kazeev$^2$}
\email[Vladimir Kazeev]{kazeev@stanford.edu}
\address{\rm $^2$ Department of Mathematics, Stanford University, 450 Serra Mall, 94305 Stanford, USA}

\begin{document}
\maketitle
\vspace{-6pt}
\begin{abstract}
Folding grid value vectors of size $2^L$ into $L$th order tensors of mode size $2\times \cdots\times 2$, combined with low-rank representation in the tensor train format, has been shown to result in highly efficient approximations for various classes of functions. These include solutions of elliptic PDEs on nonsmooth domains or with oscillatory data.
This tensor-structured approach is attractive because it leads to highly compressed, adaptive approximations based on simple discretizations.
Standard choices of the underlying bases, such as piecewise multilinear finite elements on uniform tensor product grids, entail the well-known \emph{matrix ill-conditioning} of discrete operators.
We demonstrate that, for low-rank representations, the use of tensor structure itself additionally introduces \emph{representation ill-conditioning}, a new effect specific to computations in tensor networks.
We analyze the tensor structure of a BPX preconditioner for a second-order linear elliptic operator and construct an explicit tensor-structured representation of the preconditioner, with ranks independent of the number $L$ of discretization levels.
The straightforward application of the preconditioner yields discrete operators whose matrix conditioning is uniform with respect to
the discretization parameter, but in decompositions that suffer from representation ill-conditioning.
By additionally eliminating certain redundancies in the representations of the preconditioned discrete operators,
we obtain reduced-rank decompositions that are free of both matrix and representation ill-conditioning.
For an iterative solver based on soft thresholding of low-rank tensors, we obtain convergence and complexity estimates and demonstrate its reliability and efficiency for discretizations with up to $2^{50}$ nodes in each dimension.


\end{abstract}

\section{Introduction}\label{Sc:Intro}

The direct textbook treatment of elliptic PDEs by low-order discretizations on uniform grids becomes unaffordable for many important problem classes. The high computational costs are due to the prohibitively large number of degrees of freedom required to resolve specific features of solutions, such as singularities and high-frequency oscillations, that arise in problems with nonsmooth or oscillatory data. More efficient discretizations can be obtained with basis functions that are adapted to the given problem and require fewer degrees of freedom.
However, the construction and analysis of such methods (for instance, of $hp$-adaptive solvers) generally depends on specific features of the considered problem classes and accordingly specialized analytical tools.

By the approach considered in this work, efficiency is achieved in a different way:
extremely large arrays of coefficients parametrizing simple, uniformly refined low-order discretizations are themselves
parametrized as nonlinear functions of relatively few effective degrees of freedom.
The latter parametrization is based on representing the coefficient arrays, reshaped into high-order tensors, in the \emph{tensor train} decomposition with low ranks. This representation exploits low-rank structure with respect to a hierarchy of dyadic scales,
providing, at each scale, a problem-adapted basis that can be computed using standard techniques of numerical linear algebra.
In other words, for the identification of suitable degrees of freedom,
this approach avoids relying on problem-specific \emph{a priori} information;
instead, suitable degrees of freedom are found by the low-rank tensor compression of generic,
conceptually straightforward discretizations.

In numerical solvers for PDE problems that operate on such highly compressed, nonlinear representations of basis coefficients, new difficulties arise compared to a standard entry-wise representation.
 As we demonstrate in this contribution, specific types of ill-conditioning in such tensor representations can dramatically affect the numerical stability of solvers. We show how a special low-rank representation of a BPX preconditioner allows to overcome these difficulties and obtain estimates for the total computational complexity of computing solutions with low-rank tensor-train structure.

\subsection{Low-rank tensor approximations}
The development of low-rank tensor
representations~\cite{Hackbusch:2009:HTF,Oseledets:2009:TT,Oseledets:2011:TT,Grasedyck:2010:HierarchicalSVD,Uschmajew:2013:Geometry},
such as the \emph{tensor train} format,
has originally been motivated by applications to high-dimensional PDEs.
As observed in \cite{Oseledets:2009:QTT:Dokl,Oseledets:2010:QTT,Khoromskij:2011:QuanticsApprox,Grasedyck:10:tensorization}, the artificial treatment of coefficient vectors in lower-dimensional problems as high-dimensional quantities, known in the literature as \emph{quantized tensor train} (QTT) decomposition or \emph{tensorization}, leads to highly efficient approximations in many problems of interest. See \cite{Kh18} for a general overview and, for instance, \cite{KhKh:14,KKNS:14} for further applications.

To briefly illustrate this concept, let us suppose that a function $u$ has an accurate approximation
$u \approx \sum_{j = 1}^N \bu_j \phi_j$
in terms of the basis functions
$\{ \phi_j \}_{j = 1, \ldots, N}$ with the coefficient vector $\bu = (\bu_{j} )_{j=1,\ldots,N} \in \R^N$.
The basic idea is to re-interpret $\bu$ as a higher-order tensor of mode sizes $n_1\times \cdots\times n_L$ with $\prod_{\ell=1}^L n_\ell = N$ via the identification
\[
  j\quad \leftrightarrow \quad (i_1, \ldots, i_L) \in \{ 0,\ldots,n_1-1\} \times\cdots\times \{ 0,\ldots, n_L-1\}
\]
provided by the unique decomposition
\[
    j-1
    =
    \sum_{\ell=1}^L i_\ell \prod_{k=\ell+1}^L n_k
    \quad\text{with}\quad
    i_\ell \in \{0,\ldots,n_\ell-1\}
    \quad\text{for all}\quad
    \ell = 1,\ldots,L
    \, .
\]

We assume a simple choice of basis functions, such as low-order splines, combined with a compressed, nonlinearly parametrized approximation of the corresponding coefficients $\bu$ in the tensor train format,
\begin{equation}\label{eq:ttapprox}
 \Vec{u}_{i_1, \ldots, i_L}
		\approx
		\sum_{\alpha_1=1}^{r_1}
		\cdots
		\sum_{\alpha_{L-1}=1}^{r_{L-1}}
		U_1(1, i_1, \alpha_1)\, U_2(\alpha_1, i_2, \alpha_2)
		\, \cdots \,
		U_L(\alpha_{L-1}, i_L, 1).
\end{equation}
The actual degrees of freedom are now the entries of the third-order tensors
$U_\ell \in \R^{r_{\ell-1} \times n_\ell \times r_\ell}$ with $\ell\in\{1,\ldots,L\}$, which are referred to as \emph{cores} (where $r_0 = r_L=1$ for notational convenience). In the case of $n_\ell = n\in\N$ for all $\ell$, which we consider in this work, the total number of parameters defining this approximation equals
$\sum_{\ell = 1}^L n_\ell \, r_{\ell-1} \, r_\ell \lesssim (\log N) \max \{ r_1^2,\ldots,r_{L-1}^2\}$.

For certain representative approximation problems (such as functions with isolated singularities or high-frequency oscillations), as shown in \cite{Grasedyck:10:tensorization,KazeevSchwab,KORS:2017:Multiscale1d,Khoromskij:2011:QuanticsApprox}, one obtains approximations where the rank parameters $r_1,\ldots,r_{L-1}$ grow at most polylogarithmically in the corresponding error.
This suggests the possibility of constructing numerical methods with complexity scaling as $(\log N)^\alpha$ for a fixed $\alpha$.

\subsection{Multilevel low-rank approximations for elliptic boundary value problems}

In this work, we focus on the application of low-rank tensor techniques for solving second-order elliptic boundary value problems on domains $\varOmega\subset \R^D$, where we are mainly interested in the cases of $D\in \{1,2,3\}$.
First, consider the exact solution $u$ and finite element solutions $u_h$, where $h>0$ is a mesh-size parameter,
that are simple, low-order finite element functions with coefficient vectors $\bu_h$.
These are given by suitable linear systems of the form $\bA_h \bu_h = \bbf_h$.
For each mesh size $h$, one can seek instead $u^{\mathrm{LR}}_h$ from the same finite element space
whose coefficient vector $\bu_h^{\mathrm{LR}}$ is a low-rank approximation in the form~\eqref{eq:ttapprox} of $\bu_h$.
In order to benefit from the complexity reduction afforded by the representation~\eqref{eq:ttapprox},
the vector $\bu_h^\mathrm{LR}$ needs to be computed directly in this low-rank representation.
Using corresponding representations of $\bA_h$ and $\bbf_h$,
this can be achieved by iteratively solving the nonlinear problem in terms of the cores $U_1,\ldots,U_L$ of $\bu_h^\mathrm{LR}$ in~\eqref{eq:ttapprox}.
In our setting, the binary indexing $(i_1,\ldots,i_L)$ used in the interpretation of $\bu_h$ as a tensor
of order $L$ corresponds to uniform grid refinement with $L$ levels, and thus $h\sim 2^{-L}$.
The separation of variables expressed by~\eqref{eq:ttapprox} therefore applies not to the spatial dimensions
but rather to the dyadic scales of $u^{\mathrm{LR}}_h$.

In our model problem, the underlying discretization uses piecewise $D$-linear finite elements.
Using the triangle inequality, we can decompose the error $u - u^\mathrm{LR}_h$
into a discretization error $u-u_h$, for which on uniform meshes one obtains bounds of the form
\begin{equation}\label{eq:discrerr}
  \| u - u_h \|_{\SSp{1}} \leq C_u h^s,
\end{equation}
with $C_u>0$ depending only on $u$ and $0<s\leq 1$, and the computation error $u_h - u^\mathrm{LR}_h$
including the error of low-rank approximation.
In problems where $u$ exhibits, for instance, singularities or high-frequency oscillations, one may be dealing with $C_u$ extremely large or with $s \ll 1$.
Thus, achieving reasonable total errors may require values of $h$ that are so small that the entry-wise representations of coefficient vectors and matrices is computationally infeasible.

Under natural assumptions on the data and on the underlying mesh, the problem of finding $u_h$ remains well-conditioned with respect to the problem data independently of $h$.
However, for very small $h$ as considered here, it becomes a nontrivial issue to ensure numerical stability of algorithms, since these are affected by the  condition numbers $\mathcal{O}(h^{-2})$ of $\bA_h$. Regardless of the type of solver that is employed, preconditioning $\bA_h$ becomes a necessity for avoiding numerical instabilities even for moderately small $h$.
As a first step, we therefore construct a preconditioner for $\bA_h$ that can be applied directly in low-rank form, where both the resulting matrix condition numbers after preconditioning and the tensor representation ranks are uniformly bounded with respect to the discretization level $L$.

However, we also find that when such a preconditioner is applied as usual by the standard matrix-vector multiplication in the tensor format, numerical solvers \emph{still} stagnate at an error $\norm{ u_h - u_h^\mathrm{LR} }_{\SSp{1}}$ of order $\mathcal{O}(h^{-2} \epsilon)$, where $\epsilon$ is the machine precision.
This shows that ensuring uniformly bounded matrix condition numbers by preconditioning is not sufficient for low-rank tensor methods to remain numerically stable for very small $h$. It turns out that tensor representations of vectors in the form of~\eqref{eq:ttapprox} generated by the action of $\bA_h$ can be extremely sensitive to perturbations of each single core. This new type of ill-conditioning
cannot be eradicated by simply multiplying by the preconditioner, and any further numerical manipulations
of the resulting tensor representations are prone to large round-off errors.
To quantify this effect, we introduce the notion of \emph{representation condition numbers}.

Without addressing the issue of representation ill-conditioning, one can therefore only expect $\| u - u_h \|_{\SSp{1}} = \mathcal{O}( C_u h^s + h^{-2} \epsilon)$. With the optimal choice of $h$, this yields a total error of order $\mathcal{O}(C_u^{2/(2+s)} \epsilon^{s/(2+s)})$; even in the ideal case $s=1$, one thus has a limitation to $\mathcal{O}(\epsilon^{1/3})$.
In the present paper, by analytically combining the low-rank representations of the preconditioner and of the stiffness matrix, we obtain a tensor representation that retains favorable representation condition numbers also for large $L$ and leads to solvers that remain numerically stable even for $h$ on the order of the machine precision $\epsilon$.
For the problems preconditioned in this manner, we can apply results from \cite{BD:15,BS16} to obtain bounds for the number of operations required for computing $\bu_h^\mathrm{LR}$, in terms of the ranks of low-rank \emph{best} approximations of $\bu_h$ with the same error. Since the costs depend only weakly on the discretization level $L$, one may then in fact simply choose $L$ so large that $h \approx \epsilon$. This ensures that the discretization error $\norm{u - u_h}_{\SSp{1}}$ is negligible in all practical situations and
only the explicitly controllable low-rank approximation error $\norm{u_h - u_h^\mathrm{LR}}_{\SSp{1}}$ remains.

\subsection{Conditioning of tensor train representations}

Let us now briefly outline the source of numerical instability that we need to mitigate here.
Subspace-based tensor decompositions such as the Tucker format, hierarchical tensors \cite{Hackbusch:2009:HTF}, or the presently considered tensor train format \cite{Oseledets:2011:TT} share the basic stability property that the existence of low-rank best approximations with fixed rank parameters is guaranteed. In contrast, such best approximation problems for canonical tensors are in general ill-posed \cite{Silva:08}, and one has the well-known border rank phenomena where given tensors can be approximated arbitrarily well by tensors of lower canonical ranks. In subspace-based formats, such pathologies of the canonical rank are avoided by working only with matrix ranks of certain tensor matricizations. This leads to natural higher-order generalizations of the singular value decomposition (SVD), in particular the TT-SVD algorithm for tensor trains.

However, when performing computations in such tensor formats, tensors in general do not remain in orthogonalized standard representations, such as those given by the TT-SVD. For instance, the action of low-rank representations of finite element stiffness matrices in iterative solvers may create tensor train representations with substantial redundancies that are far from their respective SVD forms. A return to the rank-reduced SVD form can then in principle be accomplished by applying standard linear algebra operations (such as QR decomposition and SVD) to the representation components.

As we demonstrate in what follows, in relevant cases, tensor train representations can become so ill-conditioned that performing this rank reduction with machine precision no longer produces useful results. To our knowledge, this particular point has not received attention in the literature so far. As we consider in further detail in Section \ref{sec:repcond}, a particular instance where this effect occurs are multilevel low-rank representations of discretization matrices of differential operators.

In order to illustrate these issues, let us consider a low-rank matrix $M=A B^\MT$ with $A \in \R^{m\times r}$ and $B \in \R^{n\times r}$. Performing numerical manipulations of $A$, for instance a QR factorization with machine precision, amounts to replacing $M$ by $\tilde M = \tilde A B^\MT$ with $\norm{A - \tilde A}_F \leq \delta \norm{A}_F$, where $\delta$ will ideally be close to the relative machine precision. Similarly to standard perturbation estimates for matrix products (see, e.g., \cite[Sec.\ 3]{MR1927606}), one obtains the generally sharp worst-case bound
\[
  \norm{M - \tilde M}_F \leq \delta \norm{A}_F \norm{B}_{2\to 2}.
\]
In the case of high-order tensor train representations, one may think of $B$ as composed of many individual cores. Even when each of these cores looks completely innocent, their cumulative effect can lead to very large $\norm{B}_{2\to 2}$.
In cases where cancellations occur in the product with $A$, the size of $\norm{M}_F$, however, can be small compared to $\norm{B}_{2\to 2}$, and perturbations to $A$ are strongly amplified.
This means that any numerical manipulation of such representations (such as orthogonalization, which is also the first step in performing a TT-SVD, see Section \ref{sec:opcores}) can introduce extremely large errors in the represented tensor.

We define the representation condition number of an operator in low-rank representation as the factor by which its action may deteriorate the conditioning of tensor train representations. In the case of the finite element stiffness matrices $\Ten{A}_h$, we find that this condition number scales (matching the standard matrix condition number) as $\mathcal{O}(h^{-2})$, which agrees with the numerically observed loss of precision. One may regard this as a tensor-decomposition analogue of the classical amplification of relative errors by ill-conditioned matrices. However, this error amplification manifests itself not in the action of the tensor representation of $\Ten{A}_h$ on any single tensor core, which by itself is harmless, but rather in the \emph{cumulative} effect that emerges when further operations are performed on the resulting output cores.

\subsection{Novelty and relation to previous work}

As a main contribution of this work, we introduce basic notions and auxiliary results for studying the representation conditioning of tensor train representations.
In particular, our finding that the stiffness matrix represented in low-rank format has a representation condition number of order $2^{2 L}$ explains numerical instabilities in its direct application for large $L$ as observed in tests in \cite{Chertkov:16}.
We prove a new result on a BPX preconditioner for second-order elliptic problems that is tailored to our purposes, and we construct a low-rank decomposition of the preconditioned stiffness matrix with the following properties:  it is well-conditioned uniformly in discretization level $L$ as a matrix; its ranks are independent of $L$; and its representation condition numbers remain moderate for large $L$.
Based on these properties, we establish an estimate for the total computational complexity of finding approximate solutions in low-rank form. These complexity bounds are shown for an iterative solver based on the soft thresholding of tensors \cite{BS16}, for which the ranks of approximate solutions can be estimated in terms of the ranks of the exact Galerkin solution. We identify appropriate approximability assumptions on solutions in the present context, which are slightly different from those proved in \cite{KazeevSchwab}.

Difficulties with the numerical stability of solvers for large $L$ have also been noted previously in \cite{KazeevSchwab}. In \cite{Chertkov:16,OseledetsRakhubaChertkov:16}, a reformulation as a constrained minimization problem with Volterra integral operators is proposed. It is demonstrated numerically in \cite{Chertkov:16} up to $L \approx 20$ to lead to improved numerical stability, compared to a direct finite difference discretization, for Poisson-type problems with $D=2$ dimensions.
However, in this reformulation, which so far has been studied only experimentally, the matrix condition number still grows exponentially with respect to $L$, and numerical stability is still observed to be lacking for larger values of $L$.

A different class of preconditioners based on approximate matrix exponentials has been proposed for QTT decompositions in \cite{KhO:QTTPrec:11}.
In the different context of separation of spatial coordinates in high-dimensional problems, tensor representations have been combined with multilevel preconditioners based on multigrid methods \cite{BallaniGrasedyck:2013,Hackbusch:2015}, BPX preconditioners \cite{AndreevTobler:2015}, and wavelet Riesz bases \cite{MR3474850}.
 There the required representation ranks of preconditioners have been observed to increase with discretization levels, in contrast to the uniformly bounded ranks that we obtain in our present setting of tensor separation between scales.

\subsection{Outline} In Section \ref{Sc:DiscrPrec}, we consider the structure of discretization matrices in detail and establish a result on symmetric BPX preconditioning. In Section \ref{Sc:TT}, we recapitulate basic notation and operations for the tensor train format. In Section \ref{sec:repcond}, we introduce notions of representation condition numbers of tensor decompositions and investigate some of their basic properties. Building on these concepts, in Section \ref{sec:tensorstructure} we construct well-conditioned multilevel low-rank representations of preconditioned discretization matrices. In Section \ref{sec:solvers}, we discuss the implications of our findings on the complexity of finding approximate solutions, and illustrate the performance of numerical solvers in Section \ref{sec:numexp}.

We use the following general notational conventions: $A\lesssim B$ denotes $A \leq C B$ with $C$ independent of any parameters explicitly appearing in the expressions $A$ and $B$, and $A\sim B$ denotes $A\lesssim B \wedge A \gtrsim B$.
We use $\norm{\cdot}_2$ to denote the $\ell^2$-norm both of vectors and of higher-order tensors, and $\norm{\cdot}_{2\to 2}$ to denote the associated operator norm. In addition, $\norm{\cdot}_F$ denotes the Frobenius norm of matrices.
By $\langle \cdot,\cdot\rangle$, we denote the $\ell^2$-inner product of vectors and tensors or the $\LSp{2}$-inner product of functions, as well as the corresponding duality product.

\section{Discretization and Preconditioning}\label{Sc:DiscrPrec}

The model problem that we focus on in what follows is posed on the product domain
$\varOmega=\hat{\varOmega}^D \subset \R^D$ with $\hat{\varOmega}=\IntOO{0}{1}$.
With $\varGamma = \Set{ x\in \Boundary \varOmega\colon \; x_1 \cdots x_D = 0 }$,
we consider the corresponding Sobolev space of functions
defined on $\varOmega$ and vanishing on $\varGamma$,
\begin{equation}\label{Eq:Space}
	V = \Set{ v \in \SSp{1}\Par{\varOmega}\colon v\Restr{\varGamma} = 0 },
\end{equation}
with norm $\norm{v}_V = \norm{v}_{\SSp[0]{1}\Par{\varOmega}} \sim \norm{v}_{\SSp{1}\Par{\varOmega}}$.
On this space, we consider the variational problem
\begin{equation}\label{varform}
	\text{find $u\in V$ such that}
	\quad
	a(u,v) = f(v)
	\quad\text{for all}\quad
	v \in V,
\end{equation}
where $a: V \times V \rightarrow \R$ is the bilinear form given by
\begin{equation}\label{Eq:DefBLF}
	a(w,v)
	=
	\int_\varOmega
	\Par{ \nabla v }^{\MT} A \QQ \nabla w
	+
	\int_\varOmega
	c \QQ v \QQ w
	\quad\text{for all}\quad
	w,v \in V,
\end{equation}
and $f \in V'$ is a given linear form.
We assume the diffusion and reaction coefficients
$A\in\LSp{\infty}\Par[1]{\varOmega, \R^{D \times D}}$
and
$c\in\LSp{\infty}\Par{\varOmega}$ to be
strongly elliptic and nonnegative, respectively:
\begin{equation}\nonumber
	\underline{A}
	=
	\essinf_{\varOmega}
	\inf_{\xi\in\R^D \SetMinus\Set{0}}
	\frac{\xi^{\MT} \! A \QQ \xi}{\xi^{\MT} \xi} > 0
	\quad\text{and}\quad
	c \geq 0
	\quad\text{a.e. on }
	\varOmega
	\, .
\end{equation}
The problem~\eqref{varform}
is a variational formulation
of a boundary value problem
for a reaction-diffusion equation with
homogeneous mixed boundary conditions:
of Dirichlet type on $\varGamma$,
and of Neumann type on $\Boundary \varOmega \SetMinus \varGamma$.

Under the assumptions on the data made so far,
the bilinear form $a$ is continuous and coercive
and the linear form $f$ is continuous.
By the Lax--Milgram theorem, \eqref{varform}
has a unique solution satisfying
\begin{equation}\label{laxmilgrambound}
	\Norm{u}_{V}
	\leq
	\underline{A}^{-1}
	\,
	\Norm{f}_{V'}
	\, .
\end{equation}
Additional assumptions on the data of the problem~\eqref{varform},
essential for its tensor-structured
preconditioning and solution,
are stated in Sections \ref{Sc:DiscrPrec} and \ref{sec:tensorstructure}.

In what follows, we consider a hierarchy of discretizations based on piecewise $D$-linear nodal basis functions on a sequence of uniform grids with cell sizes $2^{-\ell}\times\cdots \times 2^{-\ell}$, $\ell = 0,1,2,\ldots$; the basis functions can be written as tensor products of standard univariate hat functions.

In this section,
we describe
$V_\ell$ with $\ell\in\Nz$,
nested finite-dimensional subspaces of $V$ introduced in~\eqref{Eq:Space}.
We will use these subspaces to approximate the solution
of the variational problem stated in~\eqref{varform}.

\subsection{Finite element spaces for $\hat{\varOmega}=\IntOO{0}{1}$}\label{Sc:BasisFunctions}

Throughout this section, we assume that an arbitrary number $\ell\in\Nz$
of refinement levels
is fixed. We consider a uniform partition of $\hat{\varOmega}$ into $2^\ell$ subintervals
and corresponding $2^\ell$ continuous piecewise linear functions defined on $\hat{\varOmega}$.
Then, by tensorization, we introduce basis functions defined on $\varOmega$.

First, we
consider
the uniform partition
of
$\hat{\varOmega}$
that consists of
the $2^{\ell}$ intervals
\begin{equation}\label{Eq:PartInt}
	\PartInt{\ell}{i}
	=
	\IntOO{ \PartNode{\ell}{i-1} }{ \PartNode{\ell}{i} \QQ }
	\quad\text{with}\quad
	i\in\hat{\cJ}_\ell
	=
	\Set{ 1, \ldots, 2^\ell}
\end{equation}
given by the $2^{\ell} + 1$ nodes
\begin{equation}\label{Eq:PartNode}
	\PartNode{\ell}{j} = 2^{-\ell} j
	\quad\text{with}\quad
	j=0,\ldots,2^{\ell}
	\, .
\end{equation}
For each $i\in\hat{\cJ}_\ell$, we introduce
an affine mapping $\PartMap{\ell}{i}$ from $\IntOO{-1}{1}$ onto $\PartInt{\ell}{i}$:
\begin{equation}\label{Eq:DefPartMap}
	\PartMap{\ell}{i} \Par{t}
	=
	\frac12 \Par{ \PartNode{\ell}{i} + \PartNode{\ell}{i-1} }
	+
	\frac{t}{2}
	\Par{ \PartNode{\ell}{i} - \PartNode{\ell}{i-1} }
	=
	2^{-\ell} \QQ i
	+ 2^{-\ell-1} (t-1)
\end{equation}
for all $t\in\IntOO{-1}{1}$.

Further, we consider
nodal functions defined on $\hat{\varOmega}$ and
associated with these nodes:
for each $j\in\hat{\cJ}_\ell$, by
$\vphi{\ell}{j}$ we denote the function
that
is
linear on
each $\PartInt{\ell}{i}$ with $i\in\hat{\cJ}_\ell$,
continuous on $\hat{\varOmega}$
and
such that
\begin{equation}\label{Eq:PartFct}
	\vphi{\ell}{j} \Par{ \PartNode{\ell}{j'} }
	=
	2^{\frac{\ell}{2}}
	\QQ
	\delta_{j j'}
	\quad\text{for all}\quad
	j'=0,\ldots,2^\ell
	\, .
\end{equation}
The $\ell$-dependent normalization factor
in the right-hand side of~\eqref{Eq:PartFct}
results in the uniform normalization
\begin{equation}\label{Eq:vphiNorm}
	\Norm{\vphi{\ell}{j}}_{\LSp{2}\Par{\hat{\varOmega}}}
	\sim
	1
	\, .
\end{equation}
\begin{subequations}
By the above construction of basis functions,
each $\vphi{\ell}{j}$ with $j\in\hat{\cJ}_\ell$
is a degree-one polynomial
on every $\PartInt{\ell}{i}$ with $i\in\hat{\cJ}_\ell$.
This implies that, for $\alpha=0,1$,
there exist matrices $\hat{\Ten{M}}_{\ell \CQ \alpha}$
with rows and columns indexed by
$\hat{\cJ}_\ell \times \Set{\alpha,1}$
and $\hat{\cJ}_\ell$, respectively,
such that
\begin{equation}\label{Eq:PartFctDec}
	\partial^{\alpha} \vphi{\ell}{j} \Comp \PartMap{\ell}{i}
	=
	\sum_{\beta=\alpha,1}
	\Par{\hat{\Ten{M}}_{\ell \CQ \alpha}}_{i \beta \; j}
	\;\,
	\hat{\psi}_\beta
	\quad\text{on}\quad
	\IntOO{-1}{1}
\end{equation}
for all $i,j\in\hat{\cJ}_\ell$,
where $\hat\psi_0$ and $\hat\psi_1$
are the standard monomials of degree zero and one,
\begin{equation}\label{Eq:BasFct}
	\hat{\psi}_0\Par{t}
	=
	1
	\quad\text{and}\quad
	\hat{\psi}_1\Par{t}
	=
	t
	\quad\text{for all}\quad
	t\in\IntOO{-1}{1}
	\, .
\end{equation}
We note that the matrix $\hat{\Ten{M}}_{\ell \CQ 0}$ is rectangular of size
$2^{\ell+1} \times 2^\ell$
and
the matrix $\hat{\Ten{M}}_{\ell \CQ 1}$ is a square matrix of order $2^{\ell}$.

For the basis functions defined in~\eqref{Eq:BasFct},
since $\hat{\psi}_1' = \hat{\psi}_0$,
the odd rows of $\hat{\Ten{M}}_{\ell \CQ 0}$
form
a multiple of $\hat{\Ten{M}}_{\ell \CQ 1}$:
for $\beta = 1$ and all $i,j\in\hat{\cJ}_\ell$,
we have
\begin{equation}\label{Eq:M01Emb}
	\Par{\hat{\Ten{M}}_{\ell \CQ 1}}_{i \beta \; j}
	=
	2^{\ell+1}
	\Par{\hat{\Ten{M}}_{\ell \CQ 0}}_{i \beta \; j}
	\, .
\end{equation}
Furthermore,
the matrices $\hat{\Ten{M}}_{\ell \CQ 0}$ and $\hat{\Ten{M}}_{\ell \CQ 1}$
have the following explicit form, which will be used below:
\begin{equation}\label{Eq:DefM01}
	\begin{array}{rclcl}
		\hat{\Ten{M}}_{\ell \CQ 0}
		&=&
		2^{\frac{1}{2}\ell-1}
		\QQ
		\CuBr[3]{
		\Par{ \hat{\Ten{I}}_{\ell} + \hat{\Ten{S}}_{\ell} \QQ }
		\otimes
		\begin{pmatrix}
			1
			\\
			0
			\\
		\end{pmatrix}
		+
		\Par{ \hat{\Ten{I}}_{\ell} - \hat{\Ten{S}}_{\ell} \QQ }
		\otimes
		\begin{pmatrix}
			0
			\\
			1
			\\
		\end{pmatrix}
		}
		\, ,
		\\
		\hat{\Ten{M}}_{\ell \CQ 1}
		&=&
		2^{\frac{1}{2}\ell-1 + \Par{\ell+1}}
		\QQ
		\Par{ \hat{\Ten{I}}_{\ell} - \hat{\Ten{S}}_{\ell} \QQ }
		\, ,
	\end{array}
\end{equation}
where
\begin{equation}\label{Eq:DefIS}
	\hat{\Ten{I}}_{\ell}
	=
	\begin{pmatrix}
		1	&																			\\
		0	& \ddots	& 										\\
			&	\ddots	&	\ddots			&								\\
			&					&	0			& 1
	\end{pmatrix}
	\quad\text{and}\quad
	\hat{\Ten{S}}_{\ell}
	=
	\begin{pmatrix}
		0	&																			\\
		1	& \ddots	& 										\\
			&	\ddots	&	\ddots			&								\\
			&					&	1			& 0
	\end{pmatrix}
\end{equation}
are square matrices of order $2^\ell$.
\end{subequations}

The finite element spaces
$\Span\Set{\vphi{\ell}{j}}_{j\in\hat{\cJ}_\ell}$ with $\ell\in\Nz$
are nested:
for all $L,\ell\in\Nz$ such that $\ell \leq L$,
we have
\begin{equation}\label{Eq:PartFctNst}
	\vphi{\ell}{j}
	=
	\sum_{j'\in\hat{\cJ}_L}
	\Par{ \hat{\Ten{P}}_{\ell,L} }_{j' \: j}
	\;
	\vphi{L}{j'}
	\quad\text{for all}\quad
	j\in\hat{\cJ}_\ell
	\, ,
\end{equation}
where $\hat{\Ten{P}}_{\ell,L}$ is the matrix
of the identity operator
from
$\Span\Set{\vphi{\ell}{j}}_{j\in\hat{\cJ}_\ell}$
to
$\Span\Set{\vphi{L}{j'}}_{j'\in\hat{\cJ}_L}$
with respect to the bases defined in~\eqref{Eq:PartFct}:
\begin{equation}\label{Eq:DefP}
	\hat{\Ten{P}}_{\ell,L}
	=
	2^{\Par{\ell-L}/2}
	\QQ
	\Par[1]{
		\hat{\Ten{I}}_{\ell}
		\KProd
		\hat{\Vec{\eta}}_{L-\ell}
		+
		\hat{\Ten{S}}_{\ell}
		\KProd
		\,
		\Par{\hat{\Vec{\xi}}_{L-\ell} - \hat{\Vec{\eta}}_{L-\ell} \QQ }
	}
	\,
\end{equation}
where
\begin{equation}\label{Eq:DefXiEta}
	\hat{\Vec{\xi}}_{k}
	=
	\begin{pmatrix}
		1									\\
		1									\\
		\vdots						\\
		1									\\
		1									\\
	\end{pmatrix}
	\quad\text{and}\quad
	\hat{\Vec{\eta}}_{k}
	=
	2^{-k}
	\begin{pmatrix}
		1									\\
		2									\\
		\vdots						\\
		2^{k}-1	\\
		2^{k}		\\
	\end{pmatrix}
\end{equation}
are $2^k$-component vectors for each $k\in\Nz$.

\subsection{Finite element spaces for $\varOmega=\IntOO{0}{1}^D$}\label{Sc:BasisFunctionsD}

The partition~\eqref{Eq:PartInt} induces
a uniform tensor product partition
of $\varOmega$
that consists of
the $2^{D \ell}$ elements
\begin{equation}\label{Eq:RefElD}
	\RefEl{\ell}{i}
	=
	\bigtimes_{d=1}^D
	\PartInt{\ell}{i_d}
	\quad\text{with}\quad
	i = \Tuple{i_1,\ldots,i_D} \in \cJ_\ell
	 = \hat{\cJ}^D_\ell = \Set{ 1, \ldots, 2^\ell}^D
	\, .
\end{equation}
Tensorizing~\eqref{Eq:PartFct}, we obtain
the $2^{D \ell}$ functions
\begin{equation}\label{Eq:RefFctD}
	\vphiD{\ell}{j}
	=
	\bigotimes_{d=1}^D
	\vphi{\ell}{j_d}
	\quad\text{with}\quad
	j = \Tuple{j_1,\ldots,j_D} \in\cJ_\ell
	\, ,
\end{equation}
which are continuous on $\overline{\varOmega}$
and
$D$-linear on each of the partition elements given by~\eqref{Eq:RefElD}.
We will use these functions as a basis of a finite-dimensional subspace of $V$,
\begin{equation}\label{Eq:DefVlD}
	V_{\ell}
	=
	\Span\Set{
		\vphiD{\ell}{j}
	}_{j\in\cJ_\ell}
	\subset V
	\, .
\end{equation}
The normalization of univariate factors in~\eqref{Eq:vphiNorm} implies
\begin{equation}\label{Eq:vphiNormD}
	\Norm{\vphiD{\ell}{j}}_{\LSp{2}\Par{\varOmega}}
	\sim
	1
	\, ,
\end{equation}
and hence
\[
	\Norm[3]{
		\,
		\sum_{i \in \cJ_\ell}
		\Vec{v}_{i}\, \vphiD{\ell}{i }
	}_{\LSp{2}\Par{\varOmega}}
	\sim
	\,
	\Norm{\Vec{v}}_{\lSp{2}}
	 \quad\text{for all}\quad
	 \Vec{v} \in \R^{\cJ_\ell}
	 \, ,
\]
with equivalence constants independent of $\ell\in\N$.
\begin{subequations}
Also, the relationship~\eqref{Eq:PartFctDec}
results in
\begin{equation}\label{Eq:PartFctDecD}
	\partial^{\alpha} \vphiD{\ell}{j}
	\Comp
	\PartMapD{\ell}{i}
	\\
	=
	\!\!\!\!\!\!
	\sum_{\beta\in\Set{\alpha_1,1} \times\cdots\times \Set{\alpha_D,1} }
	\!\!\!\!\!\!\!\!\!\!
	\Par{\Ten{M}_{\ell \CQ \alpha}}_{
		i \beta \;j
	}
	\;
	\psi_\beta
	\quad\text{on}\quad
	\IntOO{-1}{1}^D
\end{equation}
for all $\alpha=\Tuple{\alpha_1,\ldots,\alpha_D}\in\Set{0,1}^D$
and $i,j\in\cJ_\ell$
with
\begin{equation}\label{Eq:DefPartMapChM01D}
	\PartMapD{\ell}{i}
	=
	\bigotimes_{d=1}^D \PartMap{\ell}{i_d}
	\quad\text{and}\quad
	\psi_{\beta}
	=
	\bigotimes_{d=1}^D \hat{\psi}_{\beta_d}
\end{equation}
for all
$i=\Tuple{i_1,\ldots,i_D}\in\cJ_\ell$
and $\beta=\Tuple{\beta_1,\ldots,\beta_D}\in\Set{0,1}^D$
and with $\Ten{M}_{\ell \CQ \alpha}$ given by
\begin{equation}\label{Eq:DefMD}
	\Par{ \Ten{M}_{\ell \CQ \alpha} }_{i \beta \; j}
	=
	\prod_{k=1}^D \Par{ \hat{\Ten{M}}_{\ell \CQ \alpha_k} }_{i_k \beta_k \; j_k}
\end{equation}
for all
$i=\Tuple{i_1,\ldots,i_D}\in\cJ_\ell$, $j=\Tuple{j_1,\ldots,j_D}\in\cJ_\ell$
and $\beta=\Tuple{\beta_1,\ldots,\beta_D}\in\Set{0,1}^D$.
Note that, for each $\alpha \in \Set{0,1}^D$,
the rows and columns of
$\Ten{M}_{L \CQ \alpha}$
are indexed by
$\cJ_L \times \Set{\alpha_1,1} \times\cdots\times \Set{\alpha_D,1}$
and
$\cJ_L$,
respectively.
The embedding~\eqref{Eq:M01Emb} implies
\begin{equation}\label{Eq:M01EmbD}
	\Par{\Ten{M}_{\ell \CQ \alpha'}}_{i \beta \; j}
	=
	2^{ \IndNorm{\alpha'-\alpha} \Par{\ell+1} }
	\Par{\Ten{M}_{\ell \CQ \alpha}}_{i \beta \; j}
\end{equation}
for all $i,j\in \cJ_\ell$
and
$\alpha,\alpha',\beta\in \Set{0,1}^D$
such that
$\alpha_k \leq \alpha'_k \leq \beta_k$
for each $k=1,\ldots,D$.
\end{subequations}

The finite element spaces $V_\ell$ with $\ell\in\Nz$
are also nested:
for all $L,\ell\in\Nz$ such that $\ell \leq L$,
we have $V_\ell \subset V_L$.
In particular,
the basis functions of $V_\ell$ and $V_L$ introduced in~\eqref{Eq:RefFctD}
satisfy the refinement relation
\begin{equation}\label{Eq:PartFctNstD}
	\vphiD{\ell}{j}
	=
	\sum_{j'\in\cJ_L}
	\Par{
		\Ten{P}_{\ell,L}
	}_{j' \; j}
	\;\,
	\vphiD{L}{j'}
	\quad\text{for all}\quad
	j\in\cJ_\ell
	\, ,
\end{equation}
where
\begin{equation}\label{Eq:DefPD}
	\Ten{P}_{\ell,L}
	=
	\bigotimes_{k=1}^D
	\hat{\Ten{P}}_{\ell,L}
\end{equation}
with $\hat{\Ten{P}}_{\ell,L}$ given by~\eqref{Eq:DefP}.

The stiffness matrix for the bilinear form $a$ and discretization level $\ell$ is  given by
\begin{equation}\label{Adef}
  \Ten{A}_\ell =	\bigl( a(\vphiD{\ell}{i}  ,\vphiD{\ell}{j})\bigr)_{j,i\in \cJ_\ell} .
\end{equation}
Note that due to \eqref{Eq:vphiNormD},
\begin{equation*}
	\langle \Ten{A}_\ell \QQ \Vec{v}, \Vec{v} \rangle
	 \sim  \Bignorm{ \sum_{i\in \cJ_\ell} \Vec{v}_i \,\vphiD{\ell}{i}}_{V}^2
	 \quad\text{for all}\quad
	 \Vec{v} \in \R^{\cJ_\ell}
	 \, .
\end{equation*}
For the right-hand side, we set
$
   \Vec{f}_\ell = \bigl( f(\vphiD{\ell}{i}) \bigr)_{i \in \cJ_\ell}
$.

\subsection{Representation of differential operators}\label{Sc:FactDiffOp}

\begin{subequations}
	The bilinear form $a\!:\, V \times V \rightarrow \R$ in~\eqref{Eq:DefBLF} can be rewritten in the form
	\begin{equation}\label{Eq:DefBLF2}
		a(u,v)
		=
		\sum_{\Tuple{\alpha,\alpha'} \in \mathcal{D}}
		\;
		\int_\varOmega
		c_{\alpha \alpha'}
		\QQ
		\Par{\partial^{\QQ \alpha} v}
		\QQ
		\Par{\partial^{\QQ \alpha'} u}
		\quad\text{for all}\quad
		u,v \in V
	\end{equation}
	with a $\mathcal{D} \subset \Set{0,1}^D \times \Set{0,1}^D$.
	We assume that
	each coefficient function
	$c_{\alpha \alpha'}\in\LSp{\infty}\Par{\varOmega}$
	with $\Tuple{\alpha,\alpha'} \in \mathcal{D}$
	is given by
	\begin{equation}\label{Eq:DefBLFc}
		c_{\alpha \alpha'}
		\QQ \Comp \QQ
		\PartMapD{L}{i}
		=
		\sum_{\gamma \in \varGamma_{\! \alpha \alpha'}}
		\Par{ \Vec{c}_{L \CQ \alpha \CQ \alpha'} }_{ i \QQ \gamma }
		\,
		\chi_{\alpha \alpha' \gamma}
		\quad\text{on}\quad
		\IntOO{-1}{1}^D
		\quad\text{for all}\quad
		i \in \cJ_L
	\end{equation}
	in terms of
	the affine transformations
	$\PartMapD{L}{i}$ with $i \in \cJ_L$
	defined
	by~\eqref{Eq:DefPartMap} and~\eqref{Eq:DefPartMapChM01D},
	a finite index set $\varGamma_{\! \alpha \alpha'}$
	of cardinality $R_{\alpha \alpha'} = \Card{ \varGamma_{\! \alpha \alpha'} }$,
	functions $\chi_{\alpha \alpha' \gamma}\in\LSp{\infty}\Par{\IntOO{-1}{1}^D}$
	with $\gamma\in\varGamma_{\! \alpha \alpha'}$
	and
	a coefficient vector
	$
		\Vec{c}_{L \CQ \alpha \CQ \alpha'}
		\in
		\R^{\cJ_L \times \varGamma_{\! \alpha \alpha'}}
		\simeq
		\R^{2^{DL} R_{\alpha \alpha'}}
	$.
\end{subequations}
\begin{subequations}
	In this section, we analyze the dimension structure of the matrix
	$\Ten{A}_L$ of $a$ restricted to $V_L \times V_L$
	with respect to the basis of $\vphiD{L}{j}$ with $j\in\cJ_L$,
	whose entries are
	\begin{equation}\label{Eq:DefBLFA}
		\Par{\Ten{A}_L}_{j \; j'}
		=
		a \Par{\vphiD{L}{j}, \vphiD{L}{j'}}
		=
		\sum_{\Tuple{\alpha,\alpha'} \in \mathcal{D}}
		\;
		\int_\varOmega
		c_{\alpha \alpha'}
		\QQ
		\Par{\partial^{\QQ \alpha} \vphiD{L}{j}}
		\QQ
		\Par{\partial^{\QQ \alpha'} \! \vphiD{L}{j'}}
		\quad\text{with}\quad
		j,j' \in \cJ_L
		,
	\end{equation}
	induced by the tensor product dimension structure
	of the basis.
	Splitting integration over the
	elements
	$\RefEl{\ell}{i}$
	with $i\in\cJ_L$,
	given by~\eqref{Eq:RefElD},
	and applying~\eqref{Eq:PartFctDecD},
	we obtain
	\begin{multline}\label{Eq:OpDecElementwise}
		\Par{\Ten{A}_{L}}_{j \; j'}
		=
		\sum_{\Tuple{\alpha,\alpha'} \in \mathcal{D}}
		\,
		\sum_{ i \in \cJ_L }
		\;
		\int_{\RefEl{L}{i}}
		\!\!\!\!
		c_{\alpha \alpha'}
		\QQ
		\Par{\partial^{\QQ \alpha} \vphiD{L}{j}}
		\QQ
		\Par{\partial^{\QQ \alpha'} \! \vphiD{L}{j'}}
		\\
		=
		\sum_{\Tuple{\alpha,\alpha'} \in \mathcal{D}}
		\,
		\sum_{ i \in \cJ_L }
		\,
		\sum_{\gamma\in\varGamma_{\! \alpha \alpha'}}
		2^{-D \QQ \Par{L+1}}
		\,
		\Par{\Vec{c}_{L \CQ \alpha \alpha'}}_{ i \QQ \gamma }
		\!\!\!\!
		\int\displaylimits_{\IntOO{-1}{1}^D}
		\!\!\!\!\!\!
		\chi_{\alpha \alpha' \gamma}
		\\
		\sum_{\beta\in\Set{\alpha_1,1} \times\cdots\times \Set{\alpha_D,1} }
		\!\!\!\!\!\!\!\!\!\!\!\!\!\!\!
		\Par{\Ten{M}_{L \CQ \alpha}}_{i\beta \; j}
		\;\,
		\partial^{\QQ \alpha} \QQ \psi_{\beta}
		\sum_{\beta'\in\Set{\alpha'_1,1} \times\cdots\times \Set{\alpha'_D,1} }
		\!\!\!\!\!\!\!\!\!\!\!\!\!\!\!
		\Par{\Ten{M}_{L \CQ \alpha'}}_{i\beta' \; j'}
		\;\,
		\partial^{\QQ \alpha' \!} \psi_{\beta'}
			\, .
	\end{multline}
\end{subequations}
\begin{subequations}
Let us now, for all $\alpha,\alpha'\in\mathcal{D}$, introduce
a matrix $\Ten{\varLambda}_{L \CQ \alpha \CQ \alpha'}$
of size $2^{D(L+1)-\IndNorm{\alpha}} \times 2^{D(L+1)-\IndNorm{\alpha'}}$:
\begin{equation}\label{Eq:OpDecTermwiseLambda}
	\Par{\Ten{\varLambda}_{L \CQ \alpha \CQ \alpha'}}_{i \QQ \beta \; i' \! \beta'}
	=
	\delta_{i i'}
	\;
	2^{-D\Par{L+1}}
	\!\!
	\sum_{\gamma\in\varGamma_{\! \alpha \alpha'}}
	\!\!
	\Par{\Vec{c}_{L \CQ \alpha \CQ \alpha'}}_{ i \QQ \gamma }
	\!\!\!\!
	\int\displaylimits_{\IntOO{-1}{1}^D}
	\!\!\!\!\!\!
	\chi_{\alpha \alpha' \gamma}
	\,
	\Par{\partial^{\QQ \alpha} \psi_{\beta}}
	\,
	\Par{\partial^{\QQ \alpha' \!} \psi_{\beta' \!}}
\end{equation}
for all
$i,i'\in\cJ_L$,
$\beta\in\Set{\alpha_1,1} \times\cdots\times \Set{\alpha_D,1}$
and
$\beta'\in\Set{\alpha'_1,1} \times\cdots\times \Set{\alpha'_D,1}$.
Using these matrices, we can rewrite~\eqref{Eq:OpDecElementwise} as
\begin{equation}\label{Eq:OpDecTermwise}
	\Ten{A}_{L}
	=
	\sum_{\Tuple{\alpha,\alpha'} \in \mathcal{D}}
	\,
	\Ten{M}_{L \CQ \alpha}^\MT
	\;
	\Ten{\varLambda}_{L \CQ \alpha \CQ \alpha'}
	\,
	\Ten{M}_{L \CQ \alpha'}
	\, .
\end{equation}
\end{subequations}
\begin{example}\label{Ex:LambdaLaplace}
	\begin{subequations}
		In the case of the negative Laplacian, we deal with a bilinear form
		given by~\eqref{Eq:DefBLF2}
		with
		$
			\mathcal{D}
			=
			\Tuple[1]{
				\Tuple{\delta_{k 1},\ldots,\delta_{k D}}
				\CQ
				\Tuple{\delta_{k 1},\ldots,\delta_{k D}}
			}_{k=1}^D
		$
		and $c_{\alpha \alpha'} = 1$ for all
		$\Tuple{\alpha,\alpha'} \in \mathcal{D}$.
		For each $\Tuple{\alpha,\alpha}$,
		the corresponding coefficient is
		of the form~\eqref{Eq:DefBLFc}
		with $\varGamma_{\! \alpha \alpha'}=\Set{0}$,
		$\chi_{\alpha \alpha' \QQ 0}=1$
		and
		$\Par{\Vec{c}_{L \CQ \alpha \alpha'}}_{ i \QQ 0 }=1$
		for all
		$i\in\cJ_L$.
		The corresponding matrix $\Ten{\varLambda}_{L \CQ \alpha \CQ \alpha}$
		given by~\eqref{Eq:OpDecTermwiseLambda} takes the Kronecker product form
		\begin{equation}\label{Eq:LambdaLaplaceKronProd}
			\Ten{\varLambda}_{L \CQ \alpha \CQ \alpha}
			=
			\KProdBig_{k=1}^D
			\hat{\Ten{\varLambda}}_{L \CQ \alpha_k \CQ \alpha_k}
			\, ,
		\end{equation}
		where
		the factors
		$\hat{\Ten{\varLambda}}_{L \CQ 0 \CQ 0}$
		and
		$\hat{\Ten{\varLambda}}_{L \CQ 1 \CQ 1}$
		are diagonal matrices
		independent of $\Tuple{\alpha,\alpha'} \in \mathcal{D}$
		whose rows and columns
		are indexed by
		$\cJ_L \otimes \Set{0,1}$
		and
		$\cJ_L \otimes \Set{1}$
		respectively. Specifically, their nonzero entries are
		\begin{equation}\label{Eq:LambdaLaplaceEntries}
			\Par{
				\hat{\Ten{\varLambda}}_{L \CQ 0 \CQ 0}
			}_{
				i \CQ 0 \; i \CQ 0
			}
			=
			\Par{
				\hat{\Ten{\varLambda}}_{L \CQ 1 \CQ 1}
			}_{
				i \CQ 1 \; i \CQ 1
			}
			=
			2^{-L}
			\quad\text{and}\quad
			\Par{
				\hat{\Ten{\varLambda}}_{L \CQ 0 \CQ 0}
			}_{
				i \CQ 1 \; i \CQ 1
			}
			=
			\frac13 \,
			2^{-L}, \qquad
			 i\in\cJ_L\,.
		\end{equation}
	\end{subequations}
\end{example}

The multilevel tensor structure of the factorization~\eqref{Eq:OpDecTermwise}
and, in particular, of
$\Ten{\varLambda}_{L \CQ \alpha \CQ \alpha'}$
with $\Tuple{\alpha,\alpha'} \in \mathcal{D}$
is investigated in
in Section~\ref{sec:tensorstructure}.
This analysis applies to
the case of general nonconstant coefficients $c_{\alpha \alpha'}$
with $\Tuple{\alpha,\alpha'} \in \mathcal{D}$
under the assumption that each of them
exhibits the multilevel low-rank structure
in the sense of the following Section~\ref{Sc:TT}.
Specifically, in Section~\ref{sec:tensorstructure},
we analyze the
low-rank structure of every factor matrix $\Ten{M}_{L \CQ \alpha}$
with $\alpha\in\Set{0,1}^D$
and also show how the low-rank structure of
$c_{\alpha \alpha'}$
with $\Tuple{\alpha,\alpha'} \in \mathcal{D}$
translates into that of
$\Ten{\varLambda}_{L \CQ \alpha \CQ \alpha'}$.
First, however, in the remainder of Section~\ref{Sc:DiscrPrec} we turn to the multilevel preconditioning
of $\Ten{A}_{L}$.
This gives rise to the preconditioned operator $\Ten{B}_L$
and matrices
$\Ten{Q}_{L \CQ \alpha}$
with
$\alpha \in \Set{0,1}^D$, defined in~\eqref{Eq:DefQ} below,
which relate to $\Ten{B}_L$
as
$\Ten{M}_{L \CQ \alpha}$
with $\alpha\in\Set{0,1}^D$
to $\Ten{A}_L$.
The low-rank multilevel structure of
$\Ten{B}_L$ and $\Ten{Q}_{L \CQ \alpha}$
with
$\alpha \in \Set{0,1}^D$
is the main topic of Section~\ref{sec:tensorstructure}.

\begin{remark}\label{Rm:LambdaLaplace1d}
	In the case of one dimension ($D=1$),
	let us consider a diffusion operator
	with a coefficient $c$ that is piecewise constant:
	$
		c
		\QQ \Comp \QQ
		\PartMap{L}{i}
		=
		\Par{ \hat{\Vec{c}}_{L} }_{ i }
	$
	on $\IntOO{-1}{1}$ for all $i \in \hat{\cJ}_L$,
	cf.~\eqref{Eq:DefBLFc}.
	Such coefficients appear, for example, as approximations
	in the midpoint quadrature rule.
	Then the representation~\eqref{Eq:OpDecElementwise}
	takes the form
	\begin{equation}\label{Eq:Dec1d}
		\Ten{A}_{L}
		=
		2^{-L} \,
		\hat{\Ten{M}}_{L \CQ 1}^\MT
		\;
		\Par{\diag \hat{\Vec{c}}_{L}}
		\,
		\hat{\Ten{M}}_{L \CQ 1}
		=
		2^{2L}
		\,
		\SqBr[1]{
			\diag \, \Par[1]{
				\Par{\hat{\Ten{I}}_{L} + \hat{\Ten{S}}_{L}^\MT \;} \, \hat{\Vec{c}}_{L}
			}
			-
			\hat{\Ten{S}}_{L}^\MT \, \Par{\diag \hat{\Vec{c}}_{L}}
			-
			\Par{\diag \hat{\Vec{c}}} \QQ \hat{\Ten{S}}_{L}
			\QQ
		}
		\, ,
	\end{equation}
	where
	$
		\hat{\Ten{M}}_{L \CQ 1} = 2^{\frac32 L}
		\Par{ \hat{\Ten{I}}_{\ell} - \hat{\Ten{S}}_{\ell} \QQ }
	$
	is
	defined by~\eqref{Eq:PartFctDec} and is given explicitly by~\eqref{Eq:DefM01}.
	The representation \eqref{Eq:Dec1d} has been used for this one-dimensional case in \cite{Dolgov:2010:55,DKK:2017:Parametric1d,KORS:2017:Multiscale1d}; the representation \eqref{Eq:OpDecTermwise} provides a generalization to higher dimensions and general coefficients.

\end{remark}

\subsection{Multilevel preconditioning}

Among the various existing methods for preconditioning discretization matrices of second-order elliptic problems, we are especially interested in approaches that provide optimal preconditioning and at the same time lead to favorable multilevel low-rank structures. A choice that meets these criteria is based on the classical BPX preconditioner \cite{MR1023042}. For our particular purposes, in what follows we also obtain a new result on symmetric preconditioning by this method.

The BPX preconditioner requires a hierarchy of nested finite element spaces $V_0 \subset V_1 \subset \cdots \subset V_L \subset V$, which in the present case are the uniformly refined spaces defined in \eqref{Eq:DefVlD}.
The standard implementable form of the preconditioner (cf.\ \cite{MR1023042,MR1069277}) is then given by
\[
   C_{2,L} \, v = \sum_{\ell=0}^L 2^{-2\ell} \sum_{j \in \cJ_L} {\langle v, \vphiD{\ell}{j} \rangle} \vphiD{\ell}{j}, \quad v \in V_L.
\]
Interpreting $C_{2,L}$ as a mapping of coefficient sequences $(\langle v, \vphiD{L}{j}\rangle)_{j\in \cJ_L}$ to nodal values of finite element functions, one obtains the corresponding matrix representation
\begin{equation}\label{bpx2def}
	\Ten{C}_{2,L} = \sum_{\ell=0}^L 2^{-2\ell} \Ten{ P}_{\ell,L} \QQ \Ten{ P}_{\ell,L}^\MT
	\, ,
\end{equation}
where $\Ten{P}_{\ell,L}$ is as in \eqref{Eq:PartFctNstD}, \eqref{Eq:DefPD}.
The following result on the BPX preconditioner \eqref{bpx2def} was established in~\cite{Oswald:92,MR1186345},
see also \cite{MR1189535,MR1213412}.

\begin{theorem}\label{thm:bpx}
Let $\Ten{A}_{L}$ and $\Ten{C}_{2,L}$ be as in \eqref{Adef} and \eqref{bpx2def}.
Then there exist $c,C>0$ independent of $L$ such that
\[
 c \, \langle \Ten{C}_{2,L}^{-1} \Vec{v},\Vec{v}\rangle \leq  \langle \Ten{A}_L \Vec{v},\Vec{v}\rangle \leq  C \, \langle \Ten{C}_{2,L}^{-1} \Vec{v},\Vec{v}\rangle , \quad \Vec{v}\in \R^{\cJ_L}.
 \]
\end{theorem}

This preconditioner is therefore optimal, that is, the condition numbers of preconditioned systems remain bounded uniformly in the discretization level.
It is usually applied in the form of a left-sided preconditioning: it implies in particular that $\cond(\Ten{C}_{2,L}^{1/2} \Ten{A}_L \, \Ten{C}_{2,L}^{1/2})$ is uniformly bounded with respect to $L$ and that there exists $\omega>0$ such that the iteration $   \Ten{u}^{k+1} = \Ten{u}^k - \omega \, \Ten{C}_{2,L} \bigl( \Ten{A}_L \Ten{u}^k - \Ten{f}_L  \bigr) $
converges at an $L$-independent rate. Also standard implementations of the preconditioned conjugate gradient method use only the action of $\Ten{C}_{2,L}$.

For our purposes, for several reasons explained in further detail in what follows, we require \emph{symmetric} preconditioning, that is, an implementable operator $\Ten{C}_L$ such that
$\Ten{C}_L \Ten{A}_{L} \Ten{C}_L$
is well-conditioned.
Although $\Ten{C}_{2,L}^{1/2}$ provides optimal symmetric preconditioning
by Theorem \ref{thm:bpx}, this is not directly numerically realizable.

We thus instead consider two-sided preconditioning by the implementable operator
\begin{equation}\label{Cdef}
 \Ten{C}_{L} = \sum_{\ell=0}^L 2^{-\ell} \, \Ten{ P}_{\ell,L} \Ten{ P}_{\ell,L}^\MT
 \; .
\end{equation}
For bounding the condition number of the symmetrically preconditioned operator $\Ten{C}_{L} \Ten{A}_{L} \Ten{C}_{L}$, we need to establish spectral equivalence of $\Ten{A}_L$ and $\Ten{C}_{L}^{-2}$. This is not a direct consequence of Theorem \ref{thm:bpx}. Although relying mainly on adaptations of established techniques as in \cite{MR1443598,MR1189535,MR2425155}, the following result appears to be new.
The proof is given in Appendix \ref{App:PO}.

\begin{theorem}\label{Th:Prec}
\label{thm:bpxsymm}
With $\Ten{A}_{L}$ as in \eqref{Adef} and $\Ten{C}_{L}$ as in \eqref{Cdef}, there exist $c,C>0$ independent of $L$ such that
\begin{equation}\label{bpxtobeshown}
   c \norm{\Vec{v}}_2^2 \leq  \langle \Ten{C}_{L} \Ten{A}_L \Ten{C}_{L} \Vec{v},\Vec{v}\rangle \leq  C \norm{\Vec{v}}_2^2 , \quad \Vec{v}\in \R^{\cJ_L}.
\end{equation}
\end{theorem}

\begin{remark}
As an immediate consequence of Theorem \ref{thm:bpxsymm},
\begin{equation}\label{normequiv}
  \norm{v}_{\SSp{1}} \sim   \norm{\Ten{C}_{L}^{-1} \Vec{v}}_{2} \qquad \text{for } v = \sum_{j\in \cJ_L} \Vec{v}_j \, \vphiD{L}{j}, \quad \Vec{v}\in \R^{\cJ_L},
\end{equation}
which means that the functions $\sum_{i\in\cJ_L} (\Ten{C}_L)_{ij} \vphiD{L}{i}$, $j\in \cJ_L$, form a Riesz basis of the subspace $V_L \subset \SSp{1}(\varOmega)$ with bounds independent of $L$.
\end{remark}

In what follows, we consider the symmetrically preconditioned problem of finding $\Vec{u}_{L}$ such that
\begin{subequations}
	\begin{equation}\label{precondsystem}
	\Ten{B}_{L} \Vec{u}_{L} = \bg_L\quad\text{where}\quad \Ten{B}_{L} = \Ten{C}_{L} \Ten{A}_{L} \Ten{C}_{L}\text{ and }\bg_L = \Ten{C}_{L} \Vec{f}_{L}.
	\end{equation}
	Then $\bar{\Vec{u}}_{L} = \Ten{C}_{L} \Vec{u}_{L}$ satisfies $\Ten{A}_{L} \bar{\Vec{u}}_{L}=\Vec{f}_{L}$; that is, $\bar{\Vec{u}}_{L}$ are the (rescaled) nodal values of the Galerkin solution at level $L$.
	Using~\eqref{Eq:OpDecTermwise}, we obtain
	\begin{equation}\label{Eq:PrecOpDec}
		\Ten{B}_{L}
		=
		\sum_{\Tuple{\alpha,\alpha'} \in \mathcal{D}}
		\,
		\Ten{Q}_{L \CQ \alpha}^\MT
		\;
		\Ten{\varLambda}_{L \CQ \alpha \alpha'}
		\,
		\Ten{Q}_{L \CQ \alpha'}
		\, ,
	\end{equation}
	where
	\begin{equation}\label{Eq:DefQ}
		\Ten{Q}_{L \CQ \alpha}
		=
		\Ten{M}_{L \CQ \alpha}
		\,
		\Ten{C}_{L}
	\end{equation}
	for all $\alpha\in\Set{0,1}^D$.
\end{subequations}

For our purposes, the symmetrically preconditioned operator
is preferable mainly for two reasons.
On the one hand,
an important advantage of the symmetric preconditioning~\eqref{Eq:PrecOpDec}
consists in the norm equivalence~\eqref{normequiv},
since ultimately we are interested in numerical schemes
with guaranteed convergence in the $\SSp{1}$ norm.
With low-rank methods using SVD-based rank truncations,
as considered in further detail in Section~\ref{sec:solvers},
for any $\varepsilon>0$ we can find $\bv$ such that $\norm{\bu_L - \bv}_2 \leq\varepsilon$ with $\bu_L$ as in~\eqref{precondsystem}.
With the nodal basis coefficients $\bar{\Vec{v}} = \Ten{C}_{L} \Vec{v}$, for the corresponding finite element functions $v = \sum_{j\in \cJ_L} \bar{\Vec{v}}_j \, \vphiD{L}{j}$ and $u_L = \sum_{j\in \cJ_L} \bar{\Vec{u}}_{L,j} \, \vphiD{L}{j}$ we have
$
  \norm{u_L - v}_{\SSp{1}} \lesssim \norm{\Ten{C}_L^{-1}(\bar{\Vec{u}}_L - \bar{\Vec{v}} ) }_2  = \norm{ \bu_L - \bv}_2 \leq \varepsilon
$
by~\eqref{normequiv}.
On the other hand, the symmetric preconditioning~\eqref{Eq:PrecOpDec}
allows for the explicit assembly
of the preconditioned operator $\Ten{B}_{L}$
directly in the low-rank form,
as considered in detail in Section~\ref{sec:tensorstructure}.

\section{Tensor Train Decomposition}\label{Sc:TT}

In this section,
we recapitulate the definition of the tensor train (TT) decomposition
of multidimensional arrays
and present the notation that we need for the following sections.

\subsection{Tensor train decomposition of multidimensional arrays}

Throughout this section, we assume that $L\in\N$.
\begin{subequations}
	Let $n_1, \ldots, n_L \in\N$
	and $\Vec{u}$ be a multidimensional vector of dimension
	$n_1 \cdots n_L$.
	Let $r_1,\ldots,r_{L-1}\in\N$ and, for $\ell=1,\ldots,L$,
	let $U_\ell$ be arrays of size $r_{\ell-1} \times n_\ell \times r_\ell$,
	where $r_0 = 1$ and $r_L = 1$.
	The vector $\Vec{u}$ is said to be
	represented in the
	\emph{tensor train (TT) decomposition}~\cite{Oseledets:2009:TT,Oseledets:2011:TT}
	with
	\emph{ranks} $r_1,\ldots,r_{L-1}$
	and \emph{cores} $U_1,\ldots,U_L$
	if
	\begin{equation}\label{Eq:DefVecTT}
		\Vec{u}_{j_1, \ldots, j_L}
		=
		\sum_{\alpha_1=1}^{r_1}
		\cdots
		\sum_{\alpha_{L-1}=1}^{r_{L-1}}
		U_1(\alpha_0, j_1, \alpha_1)
		\, \cdots \,
		U_L(\alpha_{L-1}, j_L, \alpha_L)
	\end{equation}
	for all
	$j_\ell = 1,\ldots,n_\ell$
	with
	$\ell = 1,\ldots,L$,
	where $\alpha_0 \equiv 1$ and $\alpha_L \equiv 1$ are dummy indices.

	The TT decomposition for matrices is defined analogously.
	Assume that $m_1,n_1, \ldots, m_L,n_L \in\N$
	and that $\Ten{A}$ is a matrix of size $(m_1 \cdots m_L) \times (n_1 \cdots n_L)$.
	Let $p_1,\ldots,p_{L-1}\in\N$ and, for each $\ell=1,\ldots,L$,
	let $A_\ell$ be an array of size
	$p_{\ell-1} \times m_\ell \times n_\ell \times p_\ell$,
	where $p_0 = 1$ and $p_L = 1$. Then
	the representation
	\begin{equation}\label{Eq:DefMatTT}
		\Ten{A}_{i_1, \ldots, i_L \;\; j_1, \ldots, j_L}
		=
		\sum_{\beta_1=1}^{p_1}
		\cdots
		\sum_{\beta_{L-1}=1}^{p_{L-1}}
		A_1(\beta_0, i_1, j_1, \beta_1)
		\, \cdots \,
		A_L(\beta_{L-1}, i_L, j_L, \beta_L)
	\end{equation}
	for all
	$i_\ell = 1,\ldots,n_\ell$
	with
	$\ell = 1,\ldots,L$,
	where $\beta_0 \equiv 1$ and $\beta_L \equiv 1$ are dummy indices,
	is called a tensor train decomposition of the matrix $\Ten{A}$
	with \emph{ranks} $p_1,\ldots,p_{L-1}$ and \emph{cores} $A_1,\ldots,A_L$.
\end{subequations}

The TT decomposition uses one of many possible ways to separate variables
in multidimensional arrays; see, e.g., the survey~\cite{Kolda:2009}
and the monograph~\cite{Hackbusch:TensorCalculus}.
The TT decomposition is a particular case of the more general \emph{hierarchical tensor representation},
also known as the \emph{hierarchical Tucker
representation}~\cite{Hackbusch:2009:HTF,Grasedyck:2010:HierarchicalSVD}.
Both the TT and hierarchical tensor representations can be interpreted as successive
subspace approximation or low-rank matrix factorization, and this relation
allows for the quasi-optimal low-rank approximation of tensors built upon
standard matrix algorithms.

The number of parameters of the representation,
formally linear in $L$, is mainly governed by the ranks, such as
$r_1,\ldots r_{L-1}$ in~\eqref{Eq:DefVecTT} and $p_1,\ldots,p_{L-1}$ in~\eqref{Eq:DefMatTT}.
In many applications, the complexity is observed, theoretically as well as numerically,
to depend moderately on $L$ (see, e.g., \cite{Grasedyck:2013:Survey}),
which allows to lift or completely avoid the so-called \emph{curse of dimensionality}
associated with the entrywise storage of high-dimensional arrays.

The use of $L$ for the dimensionality of tensors in this section
is not accidental: in the present paper,
the ``dimension'' index $\ell\in\Set{1,\ldots,L}$
enumerates the levels of discretization, and
each of the mode indices ($i_\ell$ and $j_\ell$ with $\ell\in\Set{1,\ldots,L}$ above)
represents
the corresponding $D$ bits of the $D$ ``physical'' dimensions.
In this case, the TT format separates not ``physical'' dimensions of tensors
but rather their multilevel structure, and adaptive low-rank approximation
allows to resolve this structure in vectors and matrices. In this setting, the TT decomposition
is known as the quantized tensor train (QTT)
decomposition~\cite{Oseledets:2009:QTT:Dokl,Oseledets:2010:QTT,Khoromskij:2011:QuanticsApprox,Hackbusch:2011:Conv}.
This idea is further explained in Section~\ref{Sc:MultilevelTT}.

\subsection{Core notation}

In this section, we present
the notation
developed in~\cite{KKh:2012:QTT,KKhT:2013:Toeplitz,KRS:2013:DiffTT},
which we extensively use to
work with TT representations.
For the sake of brevity, several definitions and
properties will be stated for cores with two mode indices, which naturally arise in TT
representations of matrices. The setting with a single mode index per core can be considered a particular case
in the same way as vectors can be considered one-column matrices.

If $\BR{U}{\alpha,\beta}$ with $\alpha=1,\ldots,p$ and $\beta=1,\ldots,q$
are tensors
of size $m \times n$,
we call the array $U$ of size $p \times m \times n \times q$
given by
\begin{equation}\label{Eq:CoreBlock}
	U\Par{\alpha,i,j,\beta}
	=
	\BR{U}{\alpha,\beta}_{i j}
\end{equation}
for all
$\alpha=1,\ldots,p$, $i=1,\ldots,m$, $j=1,\ldots,n$ and $\beta=1,\ldots,q$
a \emph{core} of \emph{rank} $p \times q$
and \emph{mode size} $m \times n$.
Conversely, for any core $U$
of rank $p \times q$ and mode size $m \times n$,
we refer to each tensor
$\BR{U}{\alpha,\beta}$ with $\alpha=1,\ldots,p$ and $\beta=1,\ldots,q$
as \emph{block} $(\alpha,\beta)$ of the core $U$.

For explicitly defining a core $U$, as a tensor of order four as in \eqref{Eq:CoreBlock}, in terms of its blocks (which in turn can be matrices or vectors), we 
use the notation
\begin{equation}\label{Eq:CoreNotation}
	U
	=
	\begin{bmatrix}
		{\;}
		\BR{U}{1,1}
		&
		\cdots
		&
		\BR{U}{1,q}
		{\;}
		\\
		\vdots
		&
		\ddots
		&
		\vdots
		\\
		\BR{U}{p,1}
		&
		\cdots
		&
		\BR{U}{p,q}
		\\
	\end{bmatrix},
\end{equation}
where square brackets are used for distinction from matrices.
The following matrices are examples of blocks
that we frequently use in this paper:
\begin{equation}\label{Eq:DefBlocksIJ}
		I
		=
		\begin{pmatrix}
			1	& 0	\\
			0	& 1	\\
		\end{pmatrix}
		,\quad
		J
		=
		\begin{pmatrix}
			0	& 1	\\
			0	& 0	\\
		\end{pmatrix}
		\quad\text{and}\quad
		I_1
		=
		\begin{pmatrix}
			1	& 0	\\
			0	& 0	\\
		\end{pmatrix}
		,\quad
		I_2
		=
		\begin{pmatrix}
			0	& 0	\\
			0	& 1	\\
		\end{pmatrix}
		\, .
\end{equation}

To apply the usual matrix transposition to TT decompositions of matrices, we will use
the transposition of mode indices of cores:
\begin{equation}\label{Eq:DefCoreT}
	U^\MT \Par{\alpha,i,j,\beta}
	=
	U \Par{\alpha,j,i,\beta}
	\, ,
	\quad\text{i.e.,}\quad
	\BR{\Par[1]{U^\MT \: }}{\alpha,\beta}
	=
	\Par[1]{ \BR{U}{\alpha,\beta} }^\MT
\end{equation}
in terms of matrix transposition,
for all values of the indices.

Similarly to~\eqref{Eq:CoreBlock}, for any core $U$
of rank $p \times q$ and mode size $m \times n$,
we refer to
each matrix
$\BM{U}{i j}$
with $i\in\Set{1,\ldots,m}$ and $j\in\Set{1,\ldots,n}$
given by
\begin{equation}\label{Eq:CoreSlice}
	U\Par{\alpha,i,j,\beta}
	=
	\BM{U}{i j}_{\alpha \beta}
\end{equation}
for all
$\alpha=1,\ldots,p$ and $\beta=1,\ldots,q$
as \emph{slice} $(i,j)$ of the core $U$.

\subsection{Strong Kronecker product}

We are interested in cores as factors of TT decompositions,
and now we present how decompositions of the form~\eqref{Eq:DefVecTT}--\eqref{Eq:DefMatTT} can be
expressed in terms of cores.
For that purpose, we use
the \emph{strong Kronecker product}, introduced
for two-level matrices in~\cite{DeLauney:1994:StrKrProd}.
In order to avoid confusion with the Hadamard and tensor products,
we denote this operation by $\RP$,
as in~\cite[Definition~2.1]{KKh:2012:QTT},
where it was introduced
specifically for connecting cores into TT representations.
\begin{definition}[strong Kronecker product of cores]\label{Df:SKP}
	Let $p,q,r\in\N$ and $m_1,m_2,n_1,n_2\in\N$.
	Consider cores $U$ and $V$ of ranks $p \times r$ and $r \times q$
	and of mode size $m_1 \times m_2$
	and $n_1 \times n_2$
	respectively.
	The \emph{strong Kronecker product} $U \RP V$ of $U$ and $V$
	is the core of rank $p \times q$ and mode size
	$m_1 m_2 \times n_1 n_2$
	given,
	in terms of the matrix multiplication of slices
	(of size $p \times r$ and $r \times q$),
	by
	\begin{equation}\nonumber
		\BM{\Par{U \RP V}}{i_1 \QQ i_2, \: j_1 \QQ j_2}
		=
		\BM{U}{i_1,j_1}
		\,
		\BM{V}{i_2,j_2}
	\end{equation}
	for all combinations of
	$i_k\in\Set{1,\ldots,m_k}$ and $j_k\in\Set{1,\ldots,n_k}$ with $k=1,2$.
\end{definition}
In other words, we define $U \RP V$ as
the usual matrix product of the corresponding core matrices,
their entries (blocks) being multiplied
by means of the Kronecker product. For example,
we have
\begin{equation}\label{Eq:ExRP}
	\begin{bmatrix}
	V_{11} & V_{12}\\
	V_{21} & V_{22}\\
	\end{bmatrix}
	\RP
	\begin{bmatrix}
	W_{11} & W_{12}\\
	W_{21} & W_{22}\\
	\end{bmatrix}
	=
	\begin{bmatrix}
	V_{11} \TProd W_{11} + V_{12} \TProd W_{21} & V_{11} \TProd W_{12} + V_{12} \TProd W_{22}\\
	V_{21} \TProd W_{11} + V_{22} \TProd W_{21} & V_{21} \TProd W_{12} + V_{22} \TProd W_{22}\\
	\end{bmatrix}
\end{equation}
for two cores of rank $2 \times 2$.
Using the strong Kronecker product,
we can rewrite~\eqref{Eq:DefVecTT} and~\eqref{Eq:DefMatTT}
as follows:
\begin{equation}\label{Eq:DecTT-CoreNotation}
	\Vec{u}
	=
	\SqBr{\Vec{u}}
	=
	U_1 \RP\cdots\RP U_L
	\quad\text{and}\quad
	\Ten{A}
	=
	\SqBr{\Ten{A}}
	=
	A_1 \RP\cdots\RP A_L
	\, ,
\end{equation}
where the first equalities
indicate that
any tensor of dimension $m \times n$
can be identified with a core of rank $1 \times 1$ and mode size
$m \times n$.

\subsection{Representation map}

\begin{subequations}
	Since many different tuples of cores may represent (or approximate) the same tensor,
	we need to distinguish representations as tuples of cores.
	We denote such tuples by sans-serif letters; for example,
	\begin{equation}\label{Eq:CoreTuples}
		\Rep{U}
		=
		(U_1,\ldots,U_L)
		\quad\text{and}\quad
		\Rep{A}
		=
		(A_1,\ldots,A_L)
	\end{equation}
	for the decompositions given by~\eqref{Eq:DefVecTT} and~\eqref{Eq:DefMatTT}.
	Further, we denote by $\asm$ the function
	mapping \emph{tuples of cores} into \emph{cores}
	(in particular,
	into tensors when the rank of the resulting core is $1\times 1$):
	\begin{equation}\label{Eq:DefRepMap}
		\asm(U_1,\ldots,U_L)
		=
		U_1 \RP \cdots \RP U_L
	\end{equation}
	for any cores $U_1,\ldots,U_L$ such that the right-hand side
	exists in the sense of Definition~\ref{Df:SKP}.
	Under~\eqref{Eq:CoreTuples},
	this allows to rewrite~\eqref{Eq:DefVecTT}--\eqref{Eq:DefMatTT}
	and~\eqref{Eq:DecTT-CoreNotation}
	as
	\begin{equation}\label{Eq:RepMap}
		\Vec{u}
		=
		\SqBr{\Vec{u}}
		=
		\asm(\Rep{U})
		\quad\text{and}\quad
		\Ten{A}
		=
		\SqBr{\Ten{A}}
		=
		\asm(\Rep{A})
		\, .
	\end{equation}
	For the sets of all tuples of
	$L\in\N$ cores
	with compatible ranks,
we write
	$\TTset_L=\TTset^1_L$ in the case of blocks with one mode index, and $\TTset^2_L$ in the case of two mode indices as in \eqref{Eq:CoreBlock}.
\end{subequations}
Furthermore, let us assume that $\Rep{U} = (U_1,\ldots,U_L) \in \TTset_L$,
i.e., that $U_1,\ldots,U_L$ are cores such that $\asm(U_1,\ldots,U_L)$
is a core of rank $r_0 \times r_L$
and mode size $n$,
where $r_0,r_L,n\in\N$.
Then by
$\asm^-$ and $\asm^+$
we denote the matrices of size
$r_0 n \times r_L$
and
$r_0 \times n r_L$,
respectively,
given as follows:
\begin{subequations}
	\begin{equation}\label{Eq:RepMap-}
		\Par[1]{ \asm^-(U_1,\ldots,U_L) }_{\beta_0 i \;\; \beta_L}
		=
		\Par[1] { \asm(U_1,\ldots,U_L) } (\beta_0, i , \beta_L)
	\end{equation}
	and
	\begin{equation}\label{Eq:RepMap+}
		\Par[1]{ \asm^+(U_1,\ldots,U_L) }_{\beta_0 \;\; i \beta_L}
		=
		\Par[1] { \asm(U_1,\ldots,U_L) } (\beta_0, i , \beta_L)
	\end{equation}
	for all $\beta_0=1,\ldots,r_0$, $i=1,\ldots,n$ and $\beta_L=1,\ldots,r_L$.
	These matrices may be called
	matricizations of the core $\asm(U_1,\ldots,U_L)$: they are obtained
	by interpreting the rank indices as row and column indices,
	which is consistent with~\eqref{Eq:CoreNotation},
	and by interpreting all mode indices as either row
	or column indices.
	For notational convenience,
	we set
	$\asm^-(\varnothing) = 1$ and $\asm^+(\varnothing) = 1$
	for empty lists of cores.
	Moreover, for each $\ell=1,\ldots,L$, we define
	\begin{equation}\label{Eq:RepMap-ell}
		\asm^-_\ell (\Rep{U})
		=
		\asm^-(U_1,\ldots,U_{\ell-1})
		\quad\text{for each}\quad
		\ell=1,\ldots,L+1
	\end{equation}
	and
	\begin{equation}\label{Eq:RepMap+ell}
		\asm^+_\ell (\Rep{U})
		=
		\asm^+(U_{\ell+1},\ldots,U_L)
		\quad\text{for each}\quad
		\ell=0,\ldots,L
		\, .
	\end{equation}
	In particular, we have
	$\asm^-_1(U_1,\ldots,U_L)= 1$,
	$\asm^-_{L+1}(U_1,\ldots,U_L)=\asm^-(U_1,\ldots,U_L)$
	and
	$\asm^+_L(U_1,\ldots,U_L) = 1$,
	$\asm^+_{0}(U_1,\ldots,U_L)=\asm^+(U_1,\ldots,U_L)$.
\end{subequations}

\subsection{Unfolding matrices, ranks, and orthogonality}\label{Sc:UnfRankOrth}

Let us consider a vector $\Vec{u}$ of size $n_1 \cdots n_L$
and a matrix $\Ten{A}$ of size $m_1 \cdots m_L \times n_1 \cdots n_L$.
For every $\ell=1,\ldots,L-1$, we denote by $\Mat{\ell}(\Vec{u})$
and $\Mat{\ell}(\Ten{A})$
the $\ell$th \emph{unfolding matrices} of $\Vec{u}$ and $\Ten{A}$,
which are the matrices
of size $n_1 \cdots n_\ell \times n_{\ell+1} \cdots n_L$
and $m_1 n_1 \cdots m_\ell n_\ell  \times m_{\ell+1} n_{\ell+1} \cdots m_L n_L$
given by
\begin{subequations}
	\begin{align}
		\Par[1]{ \Mat{\ell}(\Vec{u}) }_{j_1,\ldots,j_\ell \;\; j_{\ell+1},\ldots,j_L}
		&=
		\Vec{u}_{j_1, \ldots, j_\ell, j_{\ell+1},\ldots, j_L}
		\, ,
		\label{Eq:DefUnfVec}
		\\
		\Par[1]{ \Mat{\ell}(\Ten{A}) }_{
			i_1 j_1,\ldots,i_\ell j_\ell \;\; i_{\ell+1 } j_{\ell+1},\ldots,i_L j_L
		}
		&=
		\Ten{A}_{
			i_1, \ldots, i_\ell, i_{\ell+1},\ldots, i_L
			\;\;
			j_1, \ldots, j_\ell, j_{\ell+1},\ldots, j_L
		}
		\label{Eq:DefUnfMat}
	\end{align}
	for all $i_k = 1,\ldots,m_k$ and $j_k = 1,\ldots,n_k$ with $k=1,\ldots,L$.
    For the ranks of the unfolding matrices, we use the notation
	\begin{equation}\label{Eq:DefUnfRank}
		\rank_\ell(\Vec{u}) = \rank \Mat{\ell}(\Vec{u})
		\quad\text{and}\quad
		\rank_\ell(\Ten{A}) = \rank \Mat{\ell}(\Ten{A})
	\end{equation}
	for each $\ell=1,\ldots,L-1$.
\end{subequations}

The decompositions given by~\eqref{Eq:DefVecTT}--\eqref{Eq:DefMatTT}
or, equivalently, by~\eqref{Eq:RepMap}
imply $\rank_\ell(\Vec{u}) \leq r_\ell$
and $\rank_\ell(\Ten{A}) \leq p_\ell$ for each $\ell=1,\ldots,L-1$;
furthermore, the decompositions provide
low-rank factorizations of the unfolding matrices
with the respective numbers of rank-one terms.
For example, in the case of a vector, using the notation introduced
in~\eqref{Eq:RepMap-ell}--\eqref{Eq:RepMap+ell}, we can write
$
\Mat{\ell}(\Vec{u})
=
\asm^-_{\ell+1}(\Rep{U})
\QQ
\asm^+_{\ell}(\Rep{U})
$.

Conversely, if $\Vec{u}$ and $\Ten{A}$ are such that,
for every $\ell=1,\ldots,L-1$, the unfolding matrices
$\Mat{\ell}(\Vec{u})$ and $\Mat{\ell}(\Ten{A})$
have approximations of ranks $r_\ell$ and $p_\ell$, respectively, and
of accuracy $\varepsilon_\ell$ in the Frobenius norm,
then representations $\Rep{U}=(U_1,\ldots,U_L)$ and
$\Rep{A}=(A_1,\ldots,A_L)$ of ranks $r_1,\ldots,r_{L-1}$ and
$p_1,\ldots,p_{L-1}$ such that
\begin{equation}\nonumber
	\Norm{ \asm(\Rep{U}) - \Vec{u} }_{2}^2
	\leq
	\varepsilon^2
	\quad\text{and}\quad
	\Norm{ \asm(\Rep{A}) - \Ten{A} }_{\textrm{F}}^2
	\leq
	\varepsilon^2
\end{equation}
with $\varepsilon^2 = \varepsilon_1^2 +\cdots+ \varepsilon_{\ell-1}^2$
exist~\cite[Theorem~2.2]{Oseledets:2011:TT}
and can be constructed
by the TT-SVD algorithm~\cite[Algorithm~1]{Oseledets:2011:TT}.

Next, we recapitulate the notion of orthogonality of decompositions
in terms of the matricization operators
defined in~\eqref{Eq:RepMap-}--\eqref{Eq:RepMap+ell}.
If a core $U$ is such that the matrix
$\asm^-(U)$ has orthonormal columns, then the core is called
\emph{left-orthogonal}.
Similarly, if
the matrix
$\asm^+(U)$ has orthonormal rows, then the core is called
\emph{right-orthogonal}.
Further, if $\Rep{U} \in \TTset_L$
is such that the columns of each matrix $\asm^-_\ell(\Rep{U})$ with
$\ell=2,\ldots,L+1$ are orthonormal, then the decomposition is called
\emph{left-orthogonal}.
Analogously, if the rows of each matrix $\asm^+_\ell(\Rep{U})$ with
$\ell=0,\ldots,L-1$ are orthonormal, then the decomposition is called
\emph{right-orthogonal}.
It is easy to see that any core $U$ of the form
$U = U_1 \RP U_2$ is left- or right-orthogonal
if both $U_1$ and $U_2$ are left- or right-orthogonal, respectively.
As a result, any decomposition
$\Rep{U}=(U_1,\ldots,U_{L})$
is left- or right-orthogonal
if each of the cores $U_1,\ldots,U_{L}$
is left- or right-orthogonal.

Moreover, we say that $\Rep{U}$ is in \emph{left-orthogonal TT-SVD form} if $\asm^-_{\ell+1} (\Rep{U})$ has orthonormal columns and $\asm^+_{\ell}(\Rep{U})$ has orthogonal rows for each $\ell=1,\ldots,L-1$; in other words, these matrices provide the SVD of $\Mat{\ell}(\bu)$ for each $\ell$, where the norms of the rows of $\asm^+_{\ell}(\Rep{U})$ are the corresponding singular values, and $\norm{\bu}_2 = \norm{U_L}_2$. Analogously, $\Rep{U}$ is in \emph{right-orthogonal TT-SVD form} if $\asm^-_{\ell+1} (\Rep{U})$ has orthogonal columns and $\asm^+_{\ell}(\Rep{U})$ has orthonormal rows. These TT-SVD forms can be obtained numerically for any given $\Rep{U}$ by the procedure \cite[Algorithm~1]{Oseledets:2011:TT} without rank truncation.

\subsection{Operations on cores}\label{sec:opcores}

We require
several further operations, which are explained in this section.
We start with the mode product of cores,
which was
introduced in~\cite[Definition~2.2]{KKhT:2013:Toeplitz}
and which
generalizes matrix multiplication to the case of cores.
\begin{definition}[mode product of cores]\label{Df:MP}
	Let $p,p',r,r'\in\N$ and $m,n,k\in\N$.
	Consider cores $A$ and $B$ of ranks $p \times p'$ and $r \times r'$
	and of mode size $m \times k$
	and $k \times n$,
	respectively.
	The \emph{mode core product} $A \MP B$ of $A$ and $B$
	is the core of rank $pq \times p'q'$ and mode size
	$m \times n$
	given,
	in terms of the matrix multiplication of blocks
	(of sizes $m \times k$ and $k \times n$),
	by
	\begin{equation}\nonumber
		\BR{\Par{A \MP B}}{\alpha \beta \!,\, \alpha' \beta'}
		=
		\BR{A}{\alpha,\alpha'}
		\,
		\BR{B}{\beta,\beta'}
	\end{equation}
	for all combinations of
	$\alpha=1,\ldots,p$, $\alpha'=1,\ldots,p'$, $\beta=1,\ldots,q$ and
	$\beta=1,\ldots,q'$.
	If $B$ has only one mode index,
	we apply the above definition,
	introducing a dummy mode size $n=1$ in $B$ and
	discarding it in $A \MP B$.
\end{definition}
For example, for a core $A$ with two mode indices and
a core $B$ with one or two mode indices,
each core being of rank $2 \times 2$,
we have
\begin{equation}\label{Eq:ExMP}
	\begin{bmatrix}
	A_{11} & A_{12}\\
	A_{21} & A_{22}\\
	\end{bmatrix}
	\MP
	\begin{bmatrix}
	B_{11} & B_{12}\\
	B_{21} & B_{22}\\
	\end{bmatrix}
	=
	\begin{bmatrix}
	A_{11} B_{11}  & A_{11} B_{12}  & A_{12} B_{11}  & A_{12} B_{12}  \\
	A_{11} B_{21}  & A_{11} B_{22}  & A_{12} B_{21}  & A_{12} B_{22}  \\
	A_{21} B_{11}  & A_{21} B_{12}  & A_{22} B_{11}  & A_{22} B_{12}  \\
	A_{21} B_{21}  & A_{21} B_{22}  & A_{22} B_{21}  & A_{22} B_{22}  \\
	\end{bmatrix}
\end{equation}
if the first mode size of $B$ equals the second of $A$.

The mode product and the strong Kronecker product inherit distributivity
from the usual matrix product and from the Kronecker product:
for $\Rep{A}=\Tuple{A_1,\ldots,A_L}$ and $\Rep{U}=\Tuple{U_1,\ldots,U_L}$
such that
the products $A_\ell \MP U_\ell$ with $\ell=1,\ldots,L$ are
all defined,
we have that the product $\asm\Par{\Rep{A}} \MP \asm\Par{\Rep{U}}$ is defined
and is given by
\begin{multline}\label{Eq:Distr-RP-MP}
	\asm\Par{\Rep{A}} \MP \asm\Par{\Rep{U}}
	\equiv
	\Par{ A_1 \RP \cdots \RP A_L } \MP \, \Par{ U_1 \RP \cdots \RP U_L }
	\\
	=
	\Par{ A_1 \MP U_1 } \RP \cdots \RP \Par{ A_L \MP U_L }
	\equiv
	\asm\Par{A_1 \MP U_1, \ldots, A_L \MP U_L}
	\, .
\end{multline}
When $\asm\Par{\Rep{A}}$ and $\asm\Par{\Rep{U}}$ are both of rank $1 \times 1$
and can therefore be identified with matrices,
$\asm\Par{\Rep{A}} \MP \asm\Par{\Rep{U}}$ is the core of rank $1 \times 1$
identified with the matrix-matrix product of these matrices,
and~\eqref{Eq:Distr-RP-MP} gives a representation for
the product of
a matrix $\Ten{A}=\asm\Par{\Rep{A}}$ and a vector $\Vec{u}=\asm\Par{\Rep{U}}$
given by~\eqref{Eq:DefMatTT} and~\eqref{Eq:DefVecTT}.

Finally, our derivations involve Kronecker products of cores,
which are defined as the Kronecker product of the corresponding arrays.
For any $p,p',q,q'\in\N$ and $m,n,m',n'\in \N$,
let
$A$ be a core of rank $p \times p'$
and mode size $m \times n$
and let
$B$ be a core of rank $q \times q'$
and mode size $m' \times n'$.
Then the Kronecker product $A \KProd B$ of $A$ and $B$ is
the core of rank $pq \times p'q'$ and mode size
$m m' \times n n'$
given by
\begin{subequations}
		\begin{equation}\label{Eq:DefCoreKron1}
			\BR{\Par{U \KProd V}}{\alpha \beta \!,\, \alpha' \beta'}
			=
			\BR{U}{\alpha,\alpha'}
			\KProd
			\BR{V}{\beta,\beta'}
		\end{equation}
	in terms of the Kronecker products of all pairs of block tensors
	or,
	equivalently, by
		\begin{equation}\label{Eq:DefCoreKron2}
			\BM{\Par{U \KProd V}}{i \QQ i', \: j \QQ j'}
			=
			\BM{U}{i,j}
			\KProd
			\BM{V}{i',\QQ j'}
		\end{equation}
	in terms of the Kronecker products of all pairs of slice matrices.
\end{subequations}
Similarly to~\eqref{Eq:Distr-RP-MP}, we have
\begin{multline}\label{Eq:Distr-RP-KP}
	\asm\Par{\Rep{A}} \KProd \asm\Par{\Rep{B}}
	\equiv
	\Par{ A_1 \RP \cdots \RP A_L } \KProd \, \Par{ B_1 \RP \cdots \RP B_L }
	\\
	=
	\Par{ A_1 \KProd B_1 } \RP \cdots \RP \Par{ A_L \KProd B_L }
	\equiv
	\asm\Par{A_1 \KProd B_1, \ldots, A_L \KProd B_L}
\end{multline}
for any representation $\Rep{A}=\Tuple{A_1,\ldots,A_L}$
and
$\Rep{B}=\Tuple{B_1,\ldots,B_L}$.
The relations~\eqref{Eq:Distr-RP-MP} and~\eqref{Eq:Distr-RP-KP}
indicate the well-known fact that the matrix and Kronecker products
can be recast core-wise;
see, e.g.,~\cite{Kolda:2009,Oseledets:2011:TT,Hackbusch:TensorCalculus}.

One of the most important properties of the TT decomposition of tensors
is that any representation can be made left- or right-orthogonal
in the sense of Section~\ref{Sc:UnfRankOrth} by the successive
application of the QR decomposition~\cite{Hackbusch:2009:HTF,Oseledets:2011:TT,Grasedyck:2010:HierarchicalSVD,Hackbusch:TensorCalculus,Kressner:2014:ToolboxHT}.
We now briefly present an algorithm for the left-orthogonalization of a decomposition,
which we use as an example in the discussion of representation conditioning.
This scheme is also the first step in the computation of the TT-SVD form
of a TT representation, as in~\cite[Algorithm~2]{Oseledets:2011:TT}.

\begin{algorithm}
\caption{left-orthogonalization $\lorth$ of a TT representation
(right-orthogonalization $\rorth$ can be performed analogously)
}
\label{Alg:LeftOrthTT}
\begin{algorithmic}[1]
\Function{$\Rep{V}=\lorth$}{$\Rep{U}$}
	\Require a representation $\Rep{U}=(U_1,\ldots,U_L)\in\TTset^S_L$
	with $L,S\in\N$
	\Ensure a left-orthogonal representation $\Rep{V}=(V_1,\ldots,V_L)\in\TTset^S_L$
	such that $\asm(\Rep{V})=\asm(\Rep{U})$
	\State set $W_1 = U_1$
	\Comment{$U_1 \RP U_2 \RP \cdots \RP U_L = W_1 \RP U_2 \RP \cdots \RP U_L$}
	\For{$\ell=1,\ldots,L-1$}
	\Comment{sweep through the representation from left to right}
		\State compute a matrix QR decomposition:
		$\asm^-(W_\ell) = Q_\ell R_\ell$
		\label{Alg:LeftOrthTT::LineWl}
		\State define $V_\ell$, of same dimensions as $U_\ell$,
		so that $\asm^-(V_\ell) = Q_\ell$
		\State define $W_{\ell+1}$, of same dimensions as $U_{\ell+1}$,
		so that $\asm^+(W_{\ell+1}) = R_\ell  \asm^+(U_{\ell+1})$
		\label{Alg:LeftOrthTT::LineWll}

		\Comment{$
			V_1 \RP \cdots \RP V_{\ell-1} \RP W_{\ell} \RP U_{\ell+1} \RP \cdots \RP U_L
			=
			V_1 \RP \cdots \RP V_{\ell} \RP W_{\ell+1} \RP U_{\ell+2} \RP \cdots \RP U_L
		$}
	\EndFor
	\State set $V_L = W_L$
	\Comment{$V_1 \RP \cdots \RP V_{L-1} \RP W_L = V_1 \RP \cdots \RP V_{L-1} \RP V_L$}
\EndFunction
\end{algorithmic}
\end{algorithm}

	In exact arithmetic, we have $\asm(\Rep{V})=\asm(\Rep{U})$ for
	any $\Rep{U}\in\TTset^S_L$ with $L,S\in\N$
	and $\Rep{V}=\lorth(\Rep{U})$,
	and this is the view adhered to in the references cited above.
	However, the situation is drastically different
	when errors are introduced (e.g., due to round-off)
	in the course of orthogonalization, namely,
	in lines~\ref{Alg:LeftOrthTT::LineWl} and~\ref{Alg:LeftOrthTT::LineWll}
	of Algorithm~\ref{Alg:LeftOrthTT}.

\subsection{Low-rank multilevel decomposition of vectors and matrices}\label{Sc:MultilevelTT}

Here, we discuss how we use the tensor train decomposition
for the resolution of low-rank multilevel structure in vectors and matrices
involved in the solution of~\eqref{varform}.

To reorder the entries of Kronecker products,
we use particular permutation matrices defined as follows.
First, for every $L\in\N$, we define
$\Ten{\varPi}_{L}$
as
the permutation matrix of order $2^{D L}$
such that
\begin{equation}\label{Eq:IndTranspMatrix}
	\Par{\Ten{\varPi}_{L}}_{
	\,
		i_{1,1}\,,\ldots,\,i_{D,1}
		, \ldots\ldots ,\,
		i_{1,L}\,,\ldots,\,i_{D,L}
	\;\;\;
		i_{1,1}\,,\ldots,\,i_{1,L}
		, \ldots\ldots ,\,
		i_{D,1}\,,\ldots,\,i_{D,L}
	}
	=
	1
\end{equation}
for all $i_{k,\ell}=1,2$ with $k=1,\ldots,D$ and $\ell=1,\ldots,L$.
\begin{subequations}
	In our present setting,
	we are interested in functions
	\begin{equation}\label{Eq:MultLevTenDec-LinComb}
		u_L
		=
		\sum_{j\in\cJ_L} \Par{\bar{\bu}_L}_j \, \vphiD{L}{j}
		\in V_L
	\end{equation}
	whose coefficients admit low-rank TT representations in the following sense:
	\begin{equation}\label{Eq:MultLevTenDec-CoeffRepr}
		\Ten{\varPi}_{L} \, \bar{\bu}_L
		=
		\asm(\Rep{U})
		=
		U_1 \RP \cdots \RP U_L
	\end{equation}
	with some $\Rep{U}=\Tuple{U_1,\ldots,U_L}$.
\end{subequations}

The set $\cJ_L$,
which is defined by~\eqref{Eq:RefElD},
is isomorphic to
$\Set{1,2}^{DL}$.
The matrix $\Ten{\varPi}_{L}$,
when applied to a vector
whose components are indexed by $\cJ_L$,
folds the vector into a
$DL$-dimensional array,
transposes the $DL$ indices
according to
the transformation of ordering in the product
$\Set{1,\ldots,D} \DProd \, \Set{1,\ldots,L}$
from big-endian to little-endian,
and unfolds the resulting array back into a vector.

In other words, the matrix $\Ten{\varPi}_{L}$, acting on a vector
whose entries are enumerated so that
the indices corresponding to each dimension and all of the levels
occur one after another,
rearranges the entries in such a way that
the indices corresponding to each level and all of the dimensions
occur one after another.
In the present paper, we will use $\Ten{\varPi}_{L}$
to permute the rows and columns of matrices,
as the following example illustrates.
\begin{example}
	In the case of $D=2$ and $L=3$, the following relation holds:
	\begin{equation}\nonumber
		\Ten{\varPi}_{L}
		\,
		\Par[1]{
			\,
			\underbrace{I \KProd J \KProd J^{\MT}}_{\text{dimension 1}}
			\KProd
			\underbrace{I \KProd I_1 \KProd I_2 \,}_{\text{dimension 2}}
			\,
		}
		\,
		\Ten{\varPi}_{L}^\MT
		=
		\underbrace{I \KProd I}_{\text{level 1}}
		\KProd
		\underbrace{J \KProd I_1 \,}_{\text{level 2}}
		\KProd
		\underbrace{J^{\MT} \KProd I_2 \,}_{\text{level 3}}
		\, ,
	\end{equation}
	where we use the matrices that we defined in~\eqref{Eq:DefBlocksIJ} above.
\end{example}

Similarly,
for every $L\in\N$
and $\alpha\in\Set{0,1}^D$,
we introduce
$\widetilde{\Ten{\varPi}}_{L,\alpha}$
as
a permutation matrix of order $2^{D \Par{L+1}-\IndNorm{\alpha}}$
with rows and columns indexed by
$\cJ_L \times \Set{\alpha_1,1} \times\cdots\times \Set{\alpha_D,1}$,
where $\cJ_L$ is given by~\eqref{Eq:RefElD}.
Specifically, we define
$\widetilde{\Ten{\varPi}}_{L,\alpha}$
by
\begin{equation}\label{Eq:IndTranspMatrixTilde}
	\Par{\widetilde{\Ten{\varPi}}_{L,\alpha}}_{
	\,
		i_{1,1}\,,\ldots,\,i_{1,D}
		, \ldots\ldots ,\,
		i_{1,L}\,,\ldots,\,i_{D,L},\,
		\beta_1,\ldots,\beta_D
	\;\;\;
		i_{1,1}\,,\ldots,\,i_{1,L},\,\beta_1
		, \ldots\ldots ,\,
		i_{D,1}\,,\ldots,\,i_{D,L},\,\beta_D
	}
	=
	1
\end{equation}
for all $i_{k,\ell}=1,2$ with $k=1,\ldots,D$ and $\ell=1,\ldots,L$
and for all
$\beta_k\in\Set{\alpha_k,1}$ with $k=1,\ldots,D$.

\section{Representation Conditioning}
\label{sec:repcond}

Since the TT decomposition is based on low-rank matrix factorization,
redundancy (linear dependence) in explicit TT representations can be eliminated analytically.
This is illustrated in Appendix~\ref{App:RR}:
see~\eqref{Eq:CoreTransformation1}--\eqref{Eq:CoreTransformation3} and, for more practical examples,
the proof of Lemma~\ref{Lm:DecMP}.
On the other hand, in the course of computations, this reduction
has to be done numerically.
In exact arithmetic, it can always be achieved
by the TT rounding algorithm~\cite[Algorithm~2]{Oseledets:2011:TT} using the TT-SVD.
In practice, however, it may fail due to round-off errors:
a small perturbation of a single core in a TT decomposition
may, through catastrophic cancellations, introduce a large perturbation in the represented tensor.
This can occur even in the course of orthogonalization (Algorithm~\ref{Alg:LeftOrthTT}),
which is essential for ensuring the stability of the TT rounding algorithm.
We now turn to an analysis of the potential for such error amplification, which we refer to as \emph{representation conditioning}.

\subsection{Examples of ill-conditioning of tensor representations}\label{sec:examples}

We first consider a simple example of a tensor where relative perturbations on the order of the machine precision can lead to large changes in the represented tensors.

\begin{example}\label{ex1}
	Take $D=1$ (so that $\cI=\{0,1\}$) and let $\Vec{x}$ be the tensor with all entries equal to one, $\Vec{x}_{i_1,\ldots,i_L} = 1$ for $i_1,\ldots, i_L\in\cI$. Clearly, $\Vec{x}$ can be represented by $\Rep{X} = (X_\ell)_{\ell=1,\ldots,L}$ with $\ranks(\Rep{X}) = (1, \ldots, 1)$, where $X_\ell = [(1,1)^\MT \,]$ for each $\ell$. However, we also have an alternative representation $\Rep{Y}$ with $\ranks(\Rep{Y}) = (2, \ldots, 2)$: for any fixed $R>0$ and $y_0 = (0,0)^\MT$, $y_R = (R,R)^\MT$, we instead set
	\begin{equation}\label{Eq:Cancellation1}
	   Y_1 = \begin{bmatrix}
	   	  	  (1 + R^{-L}) y_R & -y_R
	   \end{bmatrix}, \;
	    Y_2  = \ldots = Y_{L-1} = \begin{bmatrix}
	   	y_R &  y_0 \\
	   	y_0 & y_R
	   \end{bmatrix},
	   \;
	   Y_L = \begin{bmatrix}
	   	y_R \\  y_R
	   \end{bmatrix}.
	\end{equation}
	For $\varepsilon>0$, consider a perturbation of $\Rep{Y}$ by replacing $Y_\ell$ for some fixed $1<\ell<L$ by \[ \tilde Y_\ell = \begin{bmatrix}
	   	(1+\varepsilon) y_R &  y_0 \\
	   	y_0 & y_R
	   \end{bmatrix}. \]
	    This corresponds to a relative error of order $\varepsilon$ with respect to $\norm{Y_\ell}_{2}$. The resulting perturbed tensor $\Vec{x}_\varepsilon$ is again constant with entries $1 + (R^L+1)\varepsilon$, and therefore satisfies
	\begin{equation}\label{simpleperturbation}
	   \frac{ \norm{ \Vec{x} - \Vec{x}_\varepsilon}_{2} } { \norm{\Vec{x}}_{2} } = (R^L+1)\varepsilon.
	\end{equation}
	For instance, with $R=4$ and $L\geq 25$, we obtain $R^L> 10^{15}$. Consequently, any numerical manipulation of the representations can then lead to very large round-off errors that leave no significant digits in the output; in particular, an automatic rank reduction of the representation by SVD will in general not produce any useful result.

	To illustrate this numerically, we consider the left-orthogonalization $\lorth(\Rep{Y})$ with $R=4$ and machine precision $\epsilon \approx 2\times 10^{-16}$, which is also the first step in computing the TT-SVD. In exact arithmetic, the tensor $\asm(\lorth(\Rep{Y}))$ is identical to $\asm(\Rep{Y})$; however, in inexact arithmetic,
	this can be far from true.
	The associated relative numerical errors are compared to the bound \eqref{simpleperturbation} in Table~\ref{tab:stability1}. We consider two ways of evaluating the difference in $\ell^2$-norm: by extracting all tensor entries and computing the norm of their differences directly, or by assembling the difference in TT format and computing its norm using another orthogonalization. Due to numerical effects, the resulting values are not identical, but agree in their order of magnitude, which is also the same as predicted for a particular perturbation by \eqref{simpleperturbation}.
\end{example}

\begin{table}\centering\footnotesize
	\begin{tabular}{r|n{1}{2}|n{1}{2}|n{1}{2}|n{1}{2}|n{1}{2}}
		& \multicolumn{1}{c}{$L=5$} & \multicolumn{1}{c}{$L=10$} & \multicolumn{1}{c}{$L=15$} & \multicolumn{1}{c}{$L=20$} & \multicolumn{1}{c}{$L=25$} \\
		diff.\ (a) & 4.17e-13  & 6.06e-10  & 6.95e-07 & 9.64e-04 & 9.48e-01  \\
		diff.\ (b) & 3.51e-13 &  3.82e-10  & 7.10e-07 & 7.02e-04 &  1.07e+00 \\
		$(R^L + 1)\epsilon$ &  2.28e-13  &  2.33e-10  &  2.38e-07 &  2.44e-04 & 2.50e-01
	\end{tabular}
	\caption{Relative errors $\norm{\asm(\Rep{Y}) - \asm(\lorth(\Rep{Y}))}_2/\norm{\asm(\Rep{Y})}_2$ for $\Rep{Y}$ as in Example \ref{ex1} with $R=4$, with difference computed using two different methods: (a) entry-wise, (b) in TT format; compared to $(R^L + 1)\epsilon$.}\label{tab:stability1}
\end{table}

	The type of instability observed in Example \ref{ex1} occurs in a similar way in other operations, for instance in the computation of inner products, or even in the extraction of a single entry of the tensor. Due to its fundamental importance in many algorithms, we use orthogonalization as an illustrative example in what follows.

Example \ref{ex1} may seem artificial, since in the explicit construction of tensor representations one will usually try to avoid such redundant representations that can cause cancellations. However, redundancies of this kind may also be generated when matrix-vector products are performed. We next consider an example of practical relevance where both matrix and vector are each in multilevel tensor representations of minimal ranks, but the resulting representation of their product has a similar ill-conditioning as the previous example.

\begin{example}\label{ex2}
	We consider the representation of the negative Laplacian with homogeneous Dirichlet boundary conditions on $(0,1)$, discretized by piecewise linear finite elements on a uniform grid with $2^L$ interior nodes. The resulting stiffness matrix $\Ten{A}^{\text{DD}}_L\in \R^{2^L\times 2^L}$ satisfies $\Ten{A}^{\text{DD}}_L =  A_1 \RP \cdots \RP A_L$ with
	$A_1 = 4 \begin{bmatrix} I & J^\MT & J \end{bmatrix}$,
	\begin{equation}\label{laplacedd}
			A_2 =\cdots =A_{L-1} =
	    4 \begin{bmatrix} I &  J^\MT & J  \\
	       &  J &  \\
	       &   & J^\MT
	        \end{bmatrix}
	   \quad\text{and}\quad
	   A_L = 4 H^2 \begin{bmatrix} 2I - J - J^\MT \QQ \\  -J \\ -J^\MT  \end{bmatrix},
	\end{equation}
    as derived in \cite[Cor.\ 3.2]{KKh:2012:QTT},
	where $H = 1 + 2^{-L}$ and the
	elementary blocks
	are as defined in~\eqref{Eq:DefBlocksIJ}.
	The first eigenvector of $\Ten{A}^{\text{DD}}_L$, corresponding to the lowest eigenvalue $\lambda_{\text{min},L} \approx \pi^2$, is $\Vec{x}_{\text{min},L} = \bigl( \sin( \pi i 2^{-L} / H) \bigr)_{i=1,\ldots,2^L} = X_1\RP \cdots \RP X_L$, where
	\begin{equation}\label{laplaceddev}
		X_1 = \begin{bmatrix}
			x_{\text{c}}^1 & x_{\text{s}}^1
		\end{bmatrix},
		\qquad X_\ell = \begin{bmatrix}
 			x_{\text{c}}^\ell & x_{\text{s}}^\ell \\
			-x_{\text{s}}^\ell & x_{\text{c}}^\ell
		 \end{bmatrix}
		 \quad\text{for}\quad
		 \ell = 2,\ldots, L-1,
		\qquad X_L = \begin{bmatrix}
			\hat x_{\text{s}}^\ell \\ \hat x_{\text{c}}^\ell
		\end{bmatrix},
	\end{equation}
	with $t_\ell=\pi 2^{-\ell} / H$,
	\[  x_{\text{c}}^\ell = \begin{pmatrix} 1\\ \cos(t_\ell) \end{pmatrix}, \quad
	  x_{\text{s}}^\ell = \begin{pmatrix} 0 \\ \sin(t_\ell) \end{pmatrix}, \quad
	    \hat x_{\text{c}}^\ell = \begin{pmatrix} \cos(t_L) \\ \cos(2t_L) \end{pmatrix}, \quad
	     \hat x_{\text{s}}^\ell = \begin{pmatrix} \sin(t_L) \\ \sin(2t_L)) \end{pmatrix}.
	  \]
	 Then the representation $\Rep{A} \MP \Rep{X}$ of the matrix-vector product $\Ten{A}^{\text{DD}}_L\Vec{x}_{\text{min},L}$ in exact arithmetic satisfies $\asm(\Rep{A} \MP \Rep{X}) = \Ten{A}^{\text{DD}}_L\Vec{x}_{\text{min},L}  = \lambda_{\text{min},L} \Vec{x}_{\text{min},L} = \lambda_{\text{min},L} \asm(\Rep{X})$.

	We consider a similar numerical test as in Example \ref{ex1}, comparing the relative error in $\lorth(\Rep{A} \MP \Rep{X})$ to that of $\lorth(\Rep{X})$. The results are given in Table \ref{tab:stability2}, where differences are computed in the TT format. Whereas the numerical manipulation of $\Rep{X}$ leads to errors close to the machine precision $\epsilon$, in $\lorth(\Rep{A} \MP \Rep{X})$ we observe large relative errors of order $2^{2L}\epsilon$.
Note that the representation of $\Ten{A}^{\text{DD}}_L$ according to \eqref{laplacedd} has a similar structure as the redundant representation \eqref{Eq:Cancellation1} in the previous example: the cores $A_1,\ldots,A_{L-1}$ have only positive entries, whereas $A_L$ can introduce cancellations, in particular when the matrix is applied to low-frequency grid functions as above.
\end{example}

\begin{table}\centering\small
	\begin{tabular}{r|n{1}{2}|n{1}{2}|n{1}{2}|n{1}{2}|n{1}{2}}
		& \multicolumn{1}{c}{$L=20$} & \multicolumn{1}{c}{$L=25$} & \multicolumn{1}{c}{$L=30$} & \multicolumn{1}{c}{$L=35$} & \multicolumn{1}{c}{$L=40$} \\
		$e_{\Rep{V}}$ & 1.70e-15 & 1.42e-15 & 1.92e-15  &  3.15e-15 &  2.73e-15 \\
		$e_{\Rep{A} \MP \Rep{V}}$ & 2.97e-05 & 4.50e-02 & 4.21e+01 &  3.46e+04 & 4.05e+07  \\
		$2^{2L} \epsilon$ &  2.44e-04 & 2.50e-01 &  2.56e+02 & 2.62e+05 & 2.68e+08
	\end{tabular}
	\caption{Relative errors $e_{\Rep{A} \MP \Rep{V}}= {\norm{\asm(\Rep{A} \MP \Rep{V}) - \asm(\lorth(\Rep{A} \MP \Rep{V}))}_2}/{\norm{\asm(\Rep{A} \MP \Rep{V})}_2}$ compared to $e_{\Rep{V}}= {\norm{\asm(\Rep{V}) - \asm(\lorth(\Rep{V}))}_2}/{\norm{\asm(\Rep{V})}_2}$, for $\Rep{A}$, $\Rep{V}$ as in Example \ref{ex2}.}\label{tab:stability2}
\end{table}

\subsection{Representation amplification factors and condition numbers}

We now introduce a quantitative measure for the stability of TT representations under numerical manipulations.  To first order in the size of the perturbation, it is determined by the relative condition numbers of the multilinear mapping $\asm$ with respect to the component tensors in its argument. Here we use the appropriate metric on the components that corresponds to the above considered perturbations arising in linear algebra operations.

\begin{definition}
\label{def:rcond}
We define the \emph{representation amplification factors} of $\Rep{X} \in \TTset_L$, for $\ell=1,\ldots,L$, by
\begin{multline}
		\label{ramp}
		 \ramp_\ell (\Rep{X}) =  \lim_{\varepsilon\to 0} \frac1\varepsilon \sup \Bigl\{ {\norm{\asm(\Rep{\tilde X}) - \asm(\Rep{X})}_{2}} \colon \Rep{\tilde X} \in \TTset_L,\\ \norm{\tilde X_\ell - X_\ell}_{2}  \leq \varepsilon \norm{X_\ell}_{2}  \text{ and $\tilde X_k = X_k$ for $k\neq \ell$} \Bigr\} ,
	\end{multline}
and the \emph{representation condition numbers} by
\begin{equation}
	\label{rcond}
	\rcond_\ell(\Rep{X}) = \frac{\ramp_\ell(\Rep{X})}{\norm{\asm(\Rep{X})}_{2}}.
\end{equation}
	\end{definition}

By multilinearity of $\asm$, if $\Rep{X}, \Rep{\tilde X}\in \TTset_L$ with $\Vec{x} = \asm(\Rep{X})$, $\Vec{\tilde x} = \asm(\Rep{\tilde X})$ are such that $\norm{\tilde X_\ell - X_\ell}_{2}  \leq \varepsilon \norm{X_\ell}_{2}$ for each $\ell$, then for such relative perturbations of size $\varepsilon$ of cores we have the bounds
\[
   \norm{\Vec{x}-\tilde{\Vec{x}}}_2 \leq \sum_{\ell=1}^L \ramp_\ell(\Rep{X})\, \varepsilon + \cO(\varepsilon^2),
   \quad  \frac{\norm{\Vec{x}-\tilde{\Vec{x}}}_2}{\norm{\Vec{x}}_2} \leq \sum_{\ell = 1}^L  \rcond_\ell (\Rep{X}) \, \varepsilon + \cO(\varepsilon^2).
\]

In the following characterization, we use the notation $\asm^-_\ell$ and $\asm^+_\ell$ for left and right partial matricizations as introduced in~\eqref{Eq:RepMap-ell}--\eqref{Eq:RepMap+ell}.

\begin{proposition}\label{rampexpr}
	For any $\Rep{X} \in \TTset_L$ and $\ell = 1,\ldots, L$,
	\[
	 \ramp_\ell (\Rep{X}) = { \norm{\asm^-_\ell(\Rep{X})}_{2\to 2}  \norm{X_\ell}_{2} }  \norm{ \asm^+_\ell(\Rep{X})}_{2\to 2}.
	\]
\end{proposition}

\begin{proof}
	For fixed $\ell$ in \eqref{ramp}, let $\Rep{X}, \Rep{\tilde X}$ satisfy the conditions in the supremum. Then
	\begin{align*}
	\norm{\asm(\Rep{\tilde X}) - \asm(\Rep{X})}_{2}^2
	  &= \sum_{i_1,\ldots,i_L} \SqBr[2]{\BM{X_1}{i_1}\cdots \Par[2]{\BM{X_\ell}{i_\ell} - \BM{\tilde X_\ell}{i_\ell}} \cdots \BM{X_L}{i_L} }^2 \\
	  &= \sum_{i_\ell} \Norm[1]{ \asm^-_\ell (\Rep{X}) (\BM{X_\ell}{i_\ell} - \BM{\tilde X_\ell}{i_\ell}) \asm^+_\ell(\Rep{X})}_{2}^2
	  \, .
	\end{align*}
	The claim thus follows by taking the supremum over $\tilde X_\ell$ such that
	$
		\sum_{i\in\cI} \norm{\BM{\tilde X_\ell}{i} - \BM{X_\ell}{i}}_{F}^2
		\leq
		\varepsilon^2 \norm{X_\ell}_{2}^2
	$, which is in fact attained.
\end{proof}

\begin{remark}
The quantities in Definition \ref{def:rcond} measuring the amplification of perturbations can be defined in an analogous way for more general tensor networks by considering perturbations in the respective components; see \cite{Schollwoeck:2011:DMRG-MPS,Hackbusch:TensorCalculus,Orus:2014:TensorNetworks,BSU} for an overview on  such more general tensor formats.
\end{remark}

We have the following general observations concerning possible representation condition numbers, where in certain special cases, we can also give bounds that depend only on the TT ranks. Here we use the notion of TT-SVD forms introduced in Section \ref{Sc:UnfRankOrth}.

\begin{proposition}\label{prop:rcond}
Let $\Rep{X}\in\TTset_L$, then the following hold for $\ell = 1, \ldots, L$.
\begin{enumerate}[{\rm(i)}]
\item One has $\rcond_\ell(\Rep{X}) \geq 1$.
\item If $\rank_{\ell-1}(\Rep{X}) = \rank_\ell(\Rep{X}) = 1$, then $\rcond_\ell(\Rep{X}) = 1$.
\item If $\Rep{X}$ is in right-orthogonal TT-SVD form, then
	$\rcond_\ell(\Rep{X}) \leq \sqrt{\rank_{\ell-1}(\Rep{X})}$; if it is in left-orthogonal TT-SVD form, then $\rcond_\ell(\Rep{X}) \leq \sqrt{\rank_{\ell}(\Rep{X})}$.
\end{enumerate}
\end{proposition}

\begin{proof}
Statement (i) follows by estimating $\norm{\asm(\Rep{X})}_{2}$ as in the proof of Proposition \ref{rampexpr}; (ii) follows directly from properties of the Kronecker product.
To show (iii), it suffices to consider the right-orthogonal case. With $\Vec{x} = \asm(\Rep{X})$ and $r_\ell = \rank_\ell(\Rep{X})$ for each $\ell$, we need to show that $\ramp_\ell(\Rep{X}) \leq \sqrt{r_{\ell-1}} \norm{\Vec{x}}_{2}$ for each $\ell$. Since $\asm^+_\ell(\Rep{X})$ has orthonormal rows, $\norm{\asm^+_\ell(\Rep{X})}_{2\to 2} =1$ for each $\ell$. For $\ell=1$, we also have $\norm{\asm^-_\ell(\Rep{X})}_{2\to 2} = 1$ by definition and $\norm{X_\ell}_{2}=\norm{\Vec{x}}_{2}$. For $\ell>1$, by right-orthogonality of $X_\ell$ we have $\norm{X_\ell}_{2} = \sqrt{r_{\ell-1}}$. In this case, since the representation is in TT-SVD form, $\asm^-_\ell(\Rep{X})$ has orthogonal columns whose $\ell^2$-norms are the singular values of $\Mat{\ell}(\Vec{x})$, and thus $\norm{\asm^-_\ell(\Rep{X})}_{2\to 2} \leq \norm{\Vec{x}}_{2}$.
\end{proof}

Modifications to the components of a TT representation that leave the represented tensor unchanged can still lead to a change in the representation condition numbers. This change can be bounded from above as follows.

\begin{proposition}\label{reptransformbound}
	For given $\Rep{X} \in \TTset_L$, $1 \leq \ell < L$, and invertible $R \in \R^{r_\ell \times r_{\ell}}$, where $r_\ell = \rank_\ell(\Rep{X})$, let $\Rep{\tilde X}$ be identical to $\Rep{X}$ except for $ \BM{\tilde X_\ell}{i} = \BM{X_\ell}{i} R$, $\BM{\tilde X_{\ell+1}}{i} = R^{-1} \BM{X_{\ell+1}}{i}$ for $i\in \cI$. Then $\asm(\Rep{X}) = \asm(\Rep{\tilde X})$ and
	\begin{equation}\label{transf1}
	  \ramp_\ell (\Rep{\tilde X}) \leq \cond(R) \ramp_\ell (\Rep{X}), \quad
	    \ramp_{\ell+1} (\Rep{\tilde X}) \leq \cond(R) \ramp_{\ell+1} (\Rep{X}).
	\end{equation}
	In the particular case when the matrix $ \asm^+_\ell(\Rep{\tilde X})$ has orthonormal rows, one has the stronger bound
	\begin{equation}\label{transf2}
	   \ramp_\ell (\Rep{\tilde X}) \leq \ramp_\ell (\Rep{X}).
	\end{equation}
	If $ \tilde X_{\ell+1}$ is right-orthogonal, then
	\begin{equation}\label{transf3}
	   \ramp_{\ell+1} (\Rep{\tilde X}) \leq \sqrt{ r_\ell} \ramp_{\ell+1} (\Rep{X}).
	\end{equation}
\end{proposition}

\begin{proof}
	The estimates \eqref{transf1} follow from
	\[
	 \norm{\tilde X_\ell}_{2} \leq \norm{X_\ell}_{2}\norm{R}_{2\to 2}, \quad
	  \norm{\asm^+_\ell(\Rep{\tilde X})}_{2} \leq \norm{R^{-1}}_{2\to 2} \norm{ \asm^+_\ell (\Rep{X})}_{2}
	 \]
	 for the first, and analogous estimates for the second inequality.
	To see \eqref{transf2}, observe that $R \asm^+_\ell(\Rep{\tilde X}) = \asm^+_\ell (\Rep{X})$ and that under the given additional assumption, $\norm{\asm^+_\ell(\Rep{\tilde X})}_{2\to 2} = 1$ and $\norm{\asm^+_\ell (\Rep{X})}_{2\to 2} = \norm{R}_{2\to 2}$. Under the further assumption for \eqref{transf3}, we have $\norm{ X_{\ell+1}}_2 = \norm{R}_F$, and thus
	\begin{align*}
	 \ramp_{\ell+1} (\Rep{\tilde X}) &= \norm{\asm^-_{\ell+1}(\Rep{X}) R }_{2\to 2} \norm{\tilde X_{\ell+1}}_2 \norm{\asm^+_{\ell+1}(\Rep{X})}_{2\to 2}  \\
	   & \leq \norm{\asm^-_{\ell+1} (\Rep{X}) }_{2\to 2} \norm{R}_F \sqrt{r_{\ell}} \norm{\asm^+_{\ell+1}(\Rep{X})}_{2\to 2} \\ &\leq \sqrt{r_\ell} \, \ramp_{\ell+1}(\Rep{X}).\qedhere
	\end{align*}
\end{proof}

Note that the improved bounds \eqref{transf2} and \eqref{transf3}, which do not depend on the particular transformation $R$, correspond to the transformations made in algorithms for right-orthogonalizing $\Rep{X} \in \TTset_L$. When the roles of $\tilde X_{\ell}$, $\tilde X_{\ell+1}$ and the corresponding orthogonality requirements are reversed, \eqref{transf2} and \eqref{transf3} are replaced by $\ramp_{\ell+1} (\Rep{\tilde X}) \leq \ramp_{\ell+1} (\Rep{X})$ and $\ramp_{\ell} (\Rep{\tilde X}) \leq \sqrt{ r_{\ell+1}} \, \ramp_{\ell} (\Rep{X})$.

\subsection{Orthogonalization as an example of a numerical operation}

Orthogonalization of tensor train representations is usually done via QR decompositions of matricized cores. When performed at machine precision $\epsilon$, these decompositions are affected by round-off errors: applied to $M \in \R^{m\times n}$, where $mn \epsilon$ is sufficiently small, as shown in \cite[\S 19]{MR1927606} the standard Householder algorithm yields $\tilde Q, \tilde R$ such that
\begin{equation}\label{householderqr}
  \norm{M - \tilde Q \tilde R}_F \leq C_{\text{QR}} m n^{3/2} \epsilon \norm{M}_F.
\end{equation}

As a consequence of Proposition \ref{reptransformbound}, we obtain a statement on the numerical errors incurred by orthogonalization of TT representations. As a simplifying assumption, let us suppose that the QR factorizations in $\lorth(\Rep{X})$, $\rorth(\Rep{X})$ of $\Rep{X}\in\TTset_L$ are computed with machine precision $\epsilon$ up to the error bound \eqref{householderqr}, but that matrix-matrix multiplications are performed exactly (and hence the computed Householder reflectors act as exactly orthogonal matrices). Then recursively using \eqref{transf2}, \eqref{transf3}, we obtain
\begin{align}
   \label{lorthbound} \norm{\asm(\rorth(\Rep{X})) - \asm(\Rep{X})}_2 &\leq C_{\text{QR}}\sum_{\ell=2}^L ( 2^D r_{\ell-1} r_\ell )^{3/2} \ramp_\ell(\Rep{X}) \, \epsilon + \cO(\epsilon^2) , \\
   \label{rorthbound} \norm{\asm(\lorth(\Rep{X})) - \asm(\Rep{X})}_2 &\leq C_{\text{QR}} \sum_{\ell=1}^{L-1} ( 2^D r_{\ell-1} r_\ell)^{3/2} \ramp_\ell(\Rep{X})  \epsilon + \cO(\epsilon^2),
\end{align}
where $r_{\ell} = \rank_\ell(\Rep{X})$ for $\ell=1,\ldots,L$.
The analogous statements for the relative errors $\norm{\asm(\rorth(\Rep{X})) - \asm(\Rep{X})}_2/{\norm{\asm(\Rep{X})}_2}$ and $\norm{\asm(\lorth(\Rep{X})) - \asm(\Rep{X})}_2/{\norm{\asm(\Rep{X})}_2}$  hold with $\ramp$ replaced by $\rcond$.

Taking into account further numerical effects due to inexact matrix-matrix multiplications leads to substantially more complicated bounds involving additional prefactors depending more strongly on intermediate steps in the algorithms. As our numerical illustrations in Section \ref{sec:examples} demonstrate, however, the order of magnitude of the resulting errors is typically already very well predicted by the bounds \eqref{lorthbound}, \eqref{rorthbound}.

\subsection{Representations of operators}\label{sec:opramp}

\begin{definition}\label{def:mrepcond}
	For $\ell=1,\ldots,L$, we define the representation amplification factor and representation condition number of the matrix representation $\Rep{A} \in \TTset^2_L$ by
	\begin{equation}\label{Eq:oprconddef}
	   \opramp_\ell(\Rep{A}) = \sup_{\Rep{X}\in \TTset_L} \frac{\ramp_\ell(\Rep{A} \MP \Rep{X})}{\ramp_\ell(\Rep{X})}, \quad
	    \oprcond_\ell(\Rep{A}) = \sup_{\Rep{X}\in \TTset_L} \frac{\rcond_\ell(\Rep{A} \MP \Rep{X})}{\rcond_\ell(\Rep{X})}.
	\end{equation}
\end{definition}

In other words, these are the largest factors by which the action of the matrix representation $\Rep{A}$ can possibly change the representation amplification factors and the condition numbers of a vector representation.
By definition, these functions are submultiplicative:
\[
 \opramp_\ell(\Rep{A} \MP \Rep{B}) \leq \opramp_\ell(\Rep{A})\opramp_\ell(\Rep{B}), \qquad
 \oprcond_\ell(\Rep{A} \MP \Rep{B}) \leq \oprcond_\ell(\Rep{A})\oprcond_\ell(\Rep{B}).
\]
We do not have an explicit representation of these quantities as in Proposition \ref{rampexpr}, but we obtain the following upper bound in terms of the components of representations.

\begin{proposition}\label{matrepbound}
For $\Rep{A}\in\TTset^2_L$, we define the matrices
\begin{align*}
  \Ten{A}^-_{\ell, k} &= \bigl( (A_1\RP\cdots\RP A_{\ell-1})(1,i,j,k) \bigr)_{i\in \mathcal{I}^{\ell-1} ,j \in \mathcal{I}^{\ell-1}}, \quad & k &= 1, \ldots, R_{\ell-1}, \\
  \Ten{A}^+_{\ell, k}  &= \bigl( ( A_{\ell+1}\RP\cdots\RP A_{L})(k, i, j,1) \bigr)_{i\in \mathcal{I}^{L-\ell} , j\in \mathcal{I}^{L-\ell}}, \quad & k &=1,\ldots,R_\ell.
\end{align*}
Then $\opramp_\ell(\Rep{A}) \leq \beta_\ell(\Rep{A})$ for $\ell = 1,\ldots,L$, where we define
\begin{equation}\label{betadef}
  \beta_\ell(\Rep{A})
  =
  \Bigl(
		\sum_{k^-=1}^{R_{\ell-1}}  \norm{\Ten{A}^-_{\ell, k^-}}_{2\to 2}^2
		\sum_{k^+=1}^{R_{\ell}}  \norm{\Ten{A}^+_{\ell, k^+}}_{2\to 2}^2
		\sum_{ k^-=1}^{R_{\ell-1}} \sum_{ k^+ = 1}^{R_\ell}  \bignorm{\BR{A_{\ell}}{k^-, k^+}}_{2\to 2}^2
 	\Bigr)^{\frac12}
	\, ,
\end{equation}
and if $\asm(\Rep{A})$ is invertible,
\begin{equation}\label{oprcondbound}
  \oprcond_\ell(\Rep{A}) \leq \norm{\asm(\Rep{A})^{-1}}_{2\to 2} \opramp_\ell(\Rep{A}).
\end{equation}
\end{proposition}

\begin{proof}
By Proposition \ref{rampexpr}, with $\Rep{Y} = \Rep{A} \MP \Rep{X}$,
	\begin{equation*}
	\opramp_\ell(\Rep{A}) = \sup_{\Rep{X}\in \TTset_L} \frac{  \norm{\asm^-_\ell(\Rep{Y})}_{2\to 2} \norm{ \asm^+_\ell(\Rep{Y})}_{ 2\to 2} \norm{Y_\ell}_{2}  }{  \norm{\asm^-_\ell(\Rep{X})}_{2\to 2} \norm{ \asm^+_\ell(\Rep{X})}_{2\to 2} \norm{X_\ell}_{2}  } .
	\end{equation*}
The first statement follows with the estimates
\begin{equation*}
	\norm{Y_\ell}_{2}^2
	 = \sum_{ K^- = 1}^{R_{\ell-1}} \sum_{ K^+ = 1}^{R_{\ell}}  \sum_{ k^- = 1}^{r_{\ell-1}} \sum_{ k^+ = 1}^{r_{\ell}}  \norm{ \BR{A_\ell}{K^-,K^+} \,\BR{X_\ell}{k^-,k^+}  }_2^2
	 \leq  \sum_{ K^- = 1}^{R_{\ell-1}} \sum_{ K^+ = 1}^{R_{\ell}}  \bignorm{\BR{A_{\ell}}{  K^-, K^+} }_{2\to 2}^2
	     \; \norm{X_\ell}_{2}^2
\end{equation*}
and
\begin{equation*}
	\norm{\asm^-_\ell(\Rep{Y})}^2_{2\to 2}
	 \leq \sup_{\norm{{y}}_{2}=1}  \sum_{k=1}^{R_{\ell-1}} \bignorm{ \Ten{A}^-_{\ell, k} \,\asm^-_{\ell}(\Rep{X})\,{y} }_{2}^2
	\leq \sum_{k=1}^{R_{\ell-1}} \norm{\Ten{A}^-_{\ell, k} }_{2 \to 2}^2  \norm{\asm^-_{\ell}(\Rep{X})}_{2\to 2}^2,
\end{equation*}
as well as the analogous bound for $\asm^+_\ell(\Rep{Y})$.
For \eqref{oprcondbound}, note that if $\asm(\Rep{A})$ is invertible, then
	\[
	   \oprcond_\ell(\Rep{A}) \leq \biggl(\sup_{\Rep{X}\in\TTset_L} \frac{\norm{\asm(\Rep{X})}_{2}}{\norm{\asm(\Rep{Y})}_{2}} \biggr) \opramp_\ell(\Rep{A}) =  \norm{\asm(\Rep{A})^{-1}}_{2\to 2}  \opramp_\ell(\Rep{A}). \qedhere
	\]
\end{proof}

In certain situations, Proposition \ref{matrepbound} provides qualitatively sharp bounds. We now demonstrate this in the simple example of the stiffness matrix for the Dirichlet Laplacian on $(0,1)$. Similar results are observed numerically for direct representations
of more general stiffness matrices of second-order elliptic problems.

\begin{proposition}\label{prop:lapddrepcond}
 Let $\Ten{A}^{\mathrm{DD}}_L$ be as in Example \ref{ex2}, and let  $\Rep{A}$ with $\asm(\Rep{A}) = \Ten{A}^{\mathrm{DD}}_L$ be as in \eqref{laplacedd}. Then for $\ell = 1, \ldots, L$, one has
$
\opramp_\ell (\Rep{A}) \sim 2^{2L}
$
and
$
   2^{(3L + \ell)/2} \lesssim  \oprcond_\ell (\Rep{A})  \lesssim 2^{2L}.
$
\end{proposition}

\begin{proof}
	The upper bounds follow by direct computation from Proposition \ref{matrepbound} via evaluation of the auxiliary matrices in \eqref{betadef}.
For the lower bound on $\opramp_\ell(\Rep{A})$, we estimate the supremum from below using the representation $\Rep{X}_{\mathrm{max}}$ analogous to \eqref{laplaceddev} of the eigenvector $\Vec{x}_{\text{max},L} = \bigl( \sin( \pi i / H) \bigr)_{i=1,\ldots,2^L}$ corresponding to the largest eigenvalue $\lambda_{\text{max},L} \sim 2^{2L}$. To this end, it suffices to evaluate $\ramp_\ell(\Rep{A}\MP \Rep{X}_{\mathrm{max}})/\ramp_\ell(\Rep{X}_{\mathrm{max}})$ via Proposition \ref{rampexpr} in a direct but tedious calculation. For the lower bound on $\oprcond_\ell(\Rep{A})$, we instead use $\Vec{x}_{\mathrm{min},L} = \bigl(\sin( \pi i 2^{-L} / H) \bigr)_{i=1,\ldots,2^L}$ in the representation \eqref{laplaceddev}.
\end{proof}

Thus applying the matrix representation $\Rep{A}$ to the tensor decomposition $\Rep{X}$  of a vector may in general increase its representation condition number by a factor proportional to $2^{2L}$. For instance, if $\Rep{X}$ is given in TT-SVD form with representation condition number close to one, the further numerical manipulation of $\Rep{A}\MP\Rep{X}$ can cause errors of order $\cO( 2^{2L} \epsilon \norm{\asm(\Rep{X})}_2)$. This effect is observed also in the numerical tests in Section \ref{sec:noprecond}.

\section{Multilevel Low-Rank Tensor Structure of the Operator}\label{sec:tensorstructure}

In this section, we analyze the low-rank structure
of the preconditioner $\Ten{C}_{L}$,
given by~\eqref{Cdef},
and of the preconditioned discrete differential operator
$\Ten{B}_{L}$ in the form of~\eqref{Eq:PrecOpDec}.
The resulting low-rank representations are designed specifically to have small representation condition numbers
in the sense of
Definition \ref{def:mrepcond}, which is not generally the case for low-rank decompositions of $\Ten{B}_{L}$.

The central idea for obtaining well-conditioned representations is to directly combine the representations of differential operators $\hat{\Ten{M}}_{L \CQ \alpha}$ as in \eqref{Eq:DefM01} with those of the averaging matrices $\hat{\Ten{P}}_{\ell,L}$ in the preconditioner. This leads to a natural rank-reduced representation of the products $\hat{\Ten{M}}_{L \CQ \alpha} \hat{\Ten{P}}_{\ell,L}$, where the cancellations causing representation ill-conditioning that are present in the tensor decomposition of $\hat{\Ten{M}}_{L \CQ \alpha}$ are explicitly absorbed by the preconditioner and thus removed from the final representation.

\subsection{Auxiliary results}\label{Sc:tensorstructure-aux-res}

In this section, for $\ell\in\N$, we present explicit joint representations
of the identity matrix $\hat{\Ten{I}}_{\ell}$ and of the shift matrix $\hat{\Ten{S}}_{\ell}$,
given by~\eqref{Eq:DefIS}, and of
the linear vectors $\hat{\Vec{\xi}}_{\ell}$ and $\hat{\Vec{\eta}}_{\ell}$,
defined in~\eqref{Eq:DefXiEta}. These representations will be presented here in terms of the following cores:
\begin{equation}\label{Eq:DefCoresZZZ0}
		\hat{\Core{U}}
		=
		\begin{bmatrix*}[l]
			I	& J^{\MT} \QQ						\\
				& J											\\
		\end{bmatrix*}
		,
		\qquad
		\hat{\Core{X}}
		=
		\frac12
		\begin{bmatrix*}[c]
			\begin{pmatrix}
				1			\\
				2			\\
			\end{pmatrix}
			&
			\begin{pmatrix}
				0			\\
				1			\\
			\end{pmatrix}
			\\
			\begin{pmatrix}
				1			\\
				0			\\
			\end{pmatrix}
			&
			\begin{pmatrix}
				2			\\
				1			\\
			\end{pmatrix}
		\end{bmatrix*}
		\quad\text{and}\quad
		\hat{\Core{P}}
		=
		\begin{bmatrix}
			1	\\
			0	\\
		\end{bmatrix}
		\, .
\end{equation}
Our derivations will also involve the square Kronecker-product matrices
\begin{equation}\label{Eq:DefJ}
	\hat{\Ten{J}}_{\ell}
	=
	J^{\QQ \KProd \ell}
	=
	\begin{pmatrix}
		0 &        & 1 \\
		  & \ddots &   \\
		  &        & 0 \\
	\end{pmatrix}
\end{equation}
with $\ell\in\N$
and
iterated strong Kronecker products,
such as
$\hat{\Core{U}}^{\RP \ell} = \Core{U} \RP \cdots \RP \Core{U}$ with $\ell\in\N$ factors.

We start with the following auxiliary result, which appeared in slightly different forms
in~\cite{KKh:2012:QTT,KKhT:2013:Toeplitz}. The brief derivation, in the form given here,
provides an illustration and simplifies further proofs given below.
\begin{lemma}\label{Lm:DecSIJ}
	For every $\ell\in\N$,
	the matrices $\hat{\Ten{I}}_{\ell}$, $\hat{\Ten{S}}_{\ell}$ and $\hat{\Ten{J}}_{\ell}$,
	given by~\eqref{Eq:DefIS} and~\eqref{Eq:DefJ}, satisfy
	\begin{equation}\label{Eq:DecISJ}
		\begin{bmatrix}
			\hat{\Ten{I}}_{\ell}
			&
			\hat{\Ten{S}}_{\ell}
			\\
			&
			\hat{\Ten{J}}_{\ell}
			\\
		\end{bmatrix}
		=
		\hat{\Core{U}}^{\RP \ell}
		\equiv
		\begin{bmatrix*}[l]
				I	& J^{\MT} \QQ						\\
					& J											\\
		\end{bmatrix*}^{\RP \ell}
		,
	\end{equation}
	where the blocks $I$ and $J$ and given by~\eqref{Eq:DefBlocksIJ}
	and the core $\hat{\Core{U}}$ is as defined in~\eqref{Eq:DefCoresZZZ0}.
\end{lemma}
\begin{proof}
	For $\ell=1$, the claim is trivial. Let us assume that $\ell > 1$.
	Then, splitting each of the matrices
	$\hat{\Ten{S}}_{\ell}$, $\hat{\Ten{I}}_{\ell}$ and $\hat{\Ten{J}}_{\ell}$
	into four blocks,
	we obtain the following recurrence relations:
	\begin{equation}\label{Eq:KronRecISJ}
		\begin{aligned}
			\hat{\Ten{I}}_{\ell}
			&
			=
			I \KProd \hat{\Ten{I}}_{\ell-1}
			=
			\begin{bmatrix}
				I
			\end{bmatrix}
			\RP
			\begin{bmatrix}
				\hat{\Ten{I}}_{\ell-1}
			\end{bmatrix}
			,
			\\
			\hat{\Ten{S}}_{\ell}
			&
			=
			I \KProd \QQ \hat{\Ten{S}}_{\ell-1}
			+
			J^{\MT} \KProd \hat{\Ten{J}}_{\ell-1}
			=
			\begin{bmatrix}
				I	& J^{\MT} \QQ
			\end{bmatrix}
			\RP
			\begin{bmatrix}
				\hat{\Ten{S}}_{\ell-1}	\\
				\hat{\Ten{J}}_{\ell-1}	\\
			\end{bmatrix}
			,
			\\
			\hat{\Ten{J}}_{\ell}
			&
			=
			J \KProd \hat{\Ten{J}}_{\ell-1}
			=
			\begin{bmatrix}
				J
			\end{bmatrix}
			\RP
			\begin{bmatrix}
				\hat{\Ten{J}}_{\ell-1}	\\
			\end{bmatrix}
			.
		\end{aligned}
	\end{equation}
	Using the core product, these relations can be recast as
	\begin{equation}\label{Eq:RecISJ}
		\begin{bmatrix}
			\hat{\Ten{I}}_{\ell}
			&
			\hat{\Ten{S}}_{\ell}
			\\
			&
			\hat{\Ten{J}}_{\ell}
			\\
		\end{bmatrix}
		=
		\hat{\Core{U}}
		\RP
		\begin{bmatrix}
			\hat{\Ten{I}}_{\ell-1}
			&
			\hat{\Ten{S}}_{\ell-1}
			\\
			&
			\hat{\Ten{J}}_{\ell-1}
			\\
		\end{bmatrix}
		\, .
	\end{equation}
	Applying~\eqref{Eq:RecISJ} recursively, we obtain~\eqref{Eq:DecISJ}.
\end{proof}

As the following auxiliary result shows, a similar technique applies to cores whose blocks are vectors.
\begin{lemma}\label{Lm:DecXiEta}
	For every $\ell\in\Nz$,
	the vectors $\hat{\Vec{\xi}}_{\ell}$ and $\hat{\Vec{\eta}}_{\ell}$,
	given by~\eqref{Eq:DefXiEta}, satisfy
	\begin{equation}\label{Eq:DecXiEta}
		\begin{bmatrix*}[r]
			\hat{\Vec{\eta}}_{\ell}
			\\
			\hat{\Vec{\xi}}_{\ell} - \hat{\Vec{\eta}}_{\ell}
			\\
		\end{bmatrix*}
		=
		\hat{\Core{X}}^{\RP \ell}
		\RP
		\hat{\Core{P}}
		\, ,
	\end{equation}
	where $\hat{X}$ is given by \eqref{Eq:DefCoresZZZ0}.
\end{lemma}
\begin{proof}
	For $\ell=0,1$, the claim is trivial.
	Let us assume that $\ell > 1$.
	Splitting each of the vectors
	$\hat{\Vec{\xi}}_{\ell}$ and $\hat{\Vec{\eta}}_{\ell}$
	into two blocks,
	we arrive at the recursion
	\begin{subequations}
	\begin{equation}\label{Eq:KronRecXiEta}
			\hat{\Vec{\xi}}_{\ell}
			=
			\begin{pmatrix}
				1	\\
				1	\\
			\end{pmatrix}
			\KProd \hat{\Vec{\xi}}_{\ell-1}
			,
			\qquad
			\hat{\Vec{\eta}}_{\ell}
			=
			\begin{pmatrix}
				1/2	\\
				1/2	\\
			\end{pmatrix}
			\KProd \hat{\Vec{\eta}}_{\ell-1}
			+
			\begin{pmatrix}
				0		\\
				1/2	\\
			\end{pmatrix}
			\KProd \hat{\Vec{\xi}}_{\ell-1}
			,
	\end{equation}
	from which it is easy to see that
	\begin{equation}\label{Eq:KronRecXi-Eta}
		\begin{aligned}
			\hat{\Vec{\eta}}_{\ell}
			&
			=
			\begin{pmatrix}
				1/2	\\
				1		\\
			\end{pmatrix}
			\KProd \hat{\Vec{\eta}}_{\ell-1}
			+
			\begin{pmatrix}
				0		\\
				1/2	\\
			\end{pmatrix}
			\KProd \, \Par{ \hat{\Vec{\xi}}_{\ell-1}-\hat{\Vec{\eta}}_{\ell-1} }
			\, ,
			\\
			\hat{\Vec{\xi}}_{\ell} - \hat{\Vec{\eta}}_{\ell}
			&
			=
			\begin{pmatrix}
				1/2	\\
				0				\\
			\end{pmatrix}
			\KProd \hat{\Vec{\eta}}_{\ell-1}
			+
			\begin{pmatrix}
				1				\\
				1/2	\\
			\end{pmatrix}
			\KProd \, \Par{ \hat{\Vec{\xi}}_{\ell-1}-\hat{\Vec{\eta}}_{\ell-1} }
			\, .
		\end{aligned}
	\end{equation}
	\end{subequations}
	Using the core product,
	the relations~\eqref{Eq:KronRecXiEta} and~\eqref{Eq:KronRecXi-Eta}
	can be recast as
	\begin{equation}\label{Eq:RecXiEta}
		\begin{bmatrix*}[r]
			\hat{\Vec{\eta}}_{\ell}
			\\
			\hat{\Vec{\xi}}_{\ell} - \hat{\Vec{\eta}}_{\ell}
			\\
		\end{bmatrix*}
		=
		\hat{\Core{X}}
		\RP
		\begin{bmatrix*}[r]
			\hat{\Vec{\eta}}_{\ell-1}
			\\
			\hat{\Vec{\xi}}_{\ell-1} - \hat{\Vec{\eta}}_{\ell-1}
			\\
		\end{bmatrix*}
		.
	\end{equation}
	Applying~\eqref{Eq:RecXiEta} recursively and
	comparing $\hat{\Vec{\xi}}_1$ and $\hat{\Vec{\eta}}_1$
	with the first column of the core $\hat{\Core{X}}$,
	which is given by $\hat{\Core{X}} \RP \hat{\Core{P}}$,
	we obtain~\eqref{Eq:DecXiEta}.
\end{proof}

\subsection{Explicit analysis of univariate factors}\label{Sc:tensorstructure1d}
In this section, we show how the auxiliary results of Section~\ref{Sc:tensorstructure-aux-res}
translate into low-rank decompositions of the univariate factors
$\hat{\Ten{M}}_{L \CQ \alpha}$ with $\alpha\in\Set{0,1}$
and $\hat{\Ten{P}}_{\ell,L}$ with $\ell=0,\ldots,L$,
where $L\in\N$.
These matrices are introduced in~\eqref{Eq:DefM01} and~\eqref{Eq:DefP}.
Let
\begin{equation}\label{Eq:DefCoresZZZ1}
	\begin{gathered}
		\hat{\Core{A}}
		=
		\begin{bmatrix*}[l]
			1			&
			0
		\end{bmatrix*}
		,
		\quad
		\hat{\Core{T}}_{0}
		=
		\begin{bmatrix*}[r]
			1	& 1						\\
			1	& -1					\\
		\end{bmatrix*}
		,
		\\
		\hat{\Core{V}}
		=
		\frac12
		\:
		\hat{\Core{T}}_{0}
		\RP
		\hat{\Core{U}}
		\RP
		\hat{\Core{T}}_{0}
		=
		\frac12
		\begin{bmatrix*}[c]
			I+J^{\MT}+J
			&
			I-J^{\MT}-J
			\\
			I+J^{\MT}-J
			&
			I-J^{\MT}+J
		\end{bmatrix*}
		,
		\quad
		\hat{\Core{M}}_0
		=
		\frac12
		\begin{bmatrix*}[r]
			\begin{pmatrix*}[r]
				1	\\
				0	\\
			\end{pmatrix*}
			\\
			\begin{pmatrix*}[r]
				0	\\
				1	\\
			\end{pmatrix*}
		\end{bmatrix*}
		\,
		,
		\quad
		\hat{\Core{M}}_1
		=
		\begin{bmatrix}
			0	\\
			1	\\
		\end{bmatrix}
		\, .		%
	\end{gathered}
\end{equation}
\begin{lemma}\label{Lm:DecM}
For every $L\in\N$ and for $\alpha=0,1$,
the matrix $\hat{\Ten{M}}_{L \CQ \alpha}$,
given by~\eqref{Eq:DefM01}, satisfies
\begin{equation}\label{Eq:DecM}
	\hat{\Ten{M}}_{L \CQ \alpha}
	=
	2^{ \Par{\alpha+\frac12} L }
	\;
	\hat{\Core{A}}
	\RP
	\hat{\Core{U}}^{\RP \ell}
	\RP
	\hat{\Core{T}}_{0}
	\RP
	\hat{\Core{V}}^{\RP \Par{L-\ell}}
	\RP
	\hat{\Core{M}}_\alpha
\end{equation}
for every $\ell=0,\ldots,L$,
where the cores $\hat{\Core{A}}$, $\hat{\Core{U}}$,
$\hat{\Core{V}}$, $\hat{\Core{T}}_{0}$
and
$\hat{\Core{M}}_\alpha$ with $\alpha=0,1$
are given by~\eqref{Eq:DefCoresZZZ0} and~\eqref{Eq:DefCoresZZZ1}.
\end{lemma}
\begin{proof}
	Consider $L\in\N$ and $\alpha\in\Set{0,1}$.
	Immediately from~\eqref{Eq:DefM01}, we obtain the representation
	\begin{equation}\nonumber
		\hat{\Ten{M}}_{L \CQ \alpha}
		=
		2^{ \Par{\alpha+\frac12} L }
		\;
		\hat{\Core{A}}
		\RP
		\begin{bmatrix}
			\hat{\Ten{I}}_{L}
			&
			\hat{\Ten{S}}_{L \,}
			\\
			&
			\hat{\Ten{J}}_{L}
			\\
		\end{bmatrix}
		\RP
		\hat{\Core{T}}_{0}
		\RP
		\hat{\Core{M}}_\alpha
		\, .
	\end{equation}
	Applying Lemma~\ref{Lm:DecSIJ},
	we arrive at the claimed decomposition in the case of $\ell=L$,
	\begin{equation}\nonumber
		\hat{\Ten{M}}_{L \CQ \alpha}
		=
		2^{ \Par{\alpha+\frac12} L }
		\;
		\hat{\Core{A}}
		\RP
		\hat{\Core{U}}^{\RP L}
		\RP
		\hat{\Core{T}}_{0}
		\RP
		\hat{\Core{M}}_\alpha
		\, .
	\end{equation}
	Using that
	$\hat{\Core{T}}_{0} \RP \hat{\Core{T}}_{0} = 2 \QQ \hat{\Core{I}}$,
	we obtain
	\begin{equation}\nonumber
		\hat{\Ten{M}}_{L \CQ \alpha}
		=
		2^{ \Par{\alpha+\frac12} L }
		\;
		\hat{\Core{A}}
		\RP
		\hat{\Core{U}}^{\RP \ell}
		\RP
		\hat{\Core{T}}_{0}
		\RP
		\Par[3]{
			\frac12
			\QQ
			\hat{\Core{T}}_{0}
			\RP
			\hat{\Core{U}}
			\RP
			\hat{\Core{T}}_{0}
			\QQ
		}^{\RP \Par{L-\ell}}
		\RP
		\hat{\Core{M}}_\alpha
	\end{equation}
	for every $\ell=0,\ldots,L-1$,
	which completes the proof due to~\eqref{Eq:DefCoresZZZ1}.
\end{proof}

\begin{lemma}\label{Lm:DecP}
For all $L\in\Nz$ and $\ell=0,\ldots,L$,
the matrix $\hat{\Ten{P}}_{\ell,L}$, given by~\eqref{Eq:DefP},
has the representation
\begin{equation}\label{Eq:DecP}
	\hat{\Ten{P}}_{\ell,L}
	=
	2^{ -\frac12 \Par{L-\ell} }
	\;
	\hat{\Core{A}}
	\RP
	\hat{\Core{U}}^{\RP \ell}
	\RP
	\hat{\Core{X}}^{\RP \Par{L-\ell}}
	\RP
	\hat{\Core{P}}
	\, ,
\end{equation}
where
$\hat{\Core{A}}$, $\hat{\Core{U}}$, $\hat{\Core{X}}$
and
$\hat{\Core{P}}$
are the cores given by~\eqref{Eq:DefCoresZZZ0} and~\eqref{Eq:DefCoresZZZ1}.
\end{lemma}
\begin{proof}
	We start with rewriting~\eqref{Eq:DefP} in terms of the core product as
	\begin{equation}\nonumber
		\hat{\Ten{P}}_{\ell,L}
		=
		2^{ -\frac12 \Par{L-\ell} }
		\;
		\hat{\Core{A}}
		\RP
		\begin{bmatrix}
			\hat{\Ten{I}}_{\ell}
			&
			\hat{\Ten{S}}_{\ell}
			\\
			&
			\hat{\Ten{J}}_{\ell}
			\\
		\end{bmatrix}
		\RP
		\begin{bmatrix*}[r]
			\hat{\Vec{\eta}}_{L-\ell}
			\\
			\hat{\Vec{\xi}}_{L-\ell}-\hat{\Vec{\eta}}_{L-\ell}
		\end{bmatrix*}
		,
	\end{equation}
	where the middle core should be omitted when
	$\ell=0$.
	Applying
	Lemma~\ref{Lm:DecSIJ} (for $\ell > 0$)
	and
	Lemma~\ref{Lm:DecXiEta}
	to expand the middle and the last cores,
	we prove the claim.
\end{proof}

\subsection{Explicit analysis of univariate factors under preconditioning}\label{Sc:tensorstructure1dprec}
Here, obtain an optimal-rank representation of the product
$\Ten{M}_{L \CQ \alpha} \, \Ten{P}_{\ell,L}$ and note how the products
$ \hat{\Ten{M}}_{L \CQ \alpha} \, \hat{\Ten{P}}_{\ell,L} \, \hat{\Ten{P}}_{\ell,L}^\MT$
and
$\hat{\Ten{P}}_{\ell,L} \, \hat{\Ten{P}}_{\ell,L}^\MT$ can be represented,
all for
$L\in\N$, $\alpha\in\Set{0,1}^D$ and $\ell=0,\ldots,L$.

The optimal-rank representation of the product $\Ten{M}_{L \CQ \alpha} \, \Ten{P}_{\ell,L}$
is obtained in terms of the following cores:
\begin{equation}
	\label{Eq:DefCoresZZZ4}
	\begin{gathered}
		\hat{\Core{T}}_{1}
		=
		\begin{bmatrix*}[r]
			 1					\\
			-1					\\
		\end{bmatrix*}
		,
		\quad
		\hat{\Core{I}}
		=
		\begin{bmatrix}
			1 & 0 \\
			0 & 1 \\
		\end{bmatrix}
		,
		\\
		\hat{\Core{Y}}_{0}
		=
		\frac12
		\begin{bmatrix*}[r]
			\begin{pmatrix}
				2			\\
				2			\\
			\end{pmatrix}
			&
			\\
			\begin{pmatrix*}[r]
				-1			\\
				 1			\\
			\end{pmatrix*}
			&
			\begin{pmatrix}
				1			\\
				1			\\
			\end{pmatrix}
		\end{bmatrix*}
		,
		\quad
		\hat{\Core{Y}}_{1}
		=
		\frac12
		\begin{bmatrix*}
			\begin{pmatrix}
				1			\\
				1			\\
			\end{pmatrix}
		\end{bmatrix*}
		\, ,
		\quad
		\hat{\Core{N}}_1
		=
		\begin{bmatrix}
			1	\\
		\end{bmatrix}
		\quad\text{and}\quad
		\hat{\Core{N}}_0
		=
		\frac12
		\begin{bmatrix*}[r]
			\begin{pmatrix*}[r]
				1	\\
				0	\\
			\end{pmatrix*}
			\\
			\begin{pmatrix*}[r]
				0	\\
				1	\\
			\end{pmatrix*}
		\end{bmatrix*}
		\, .
	\end{gathered}
\end{equation}
The proof of the following lemma is rather technical and is therefore given in Appendix~\ref{App:RR}.
\begin{lemma}\label{Lm:DecMP}
	For all $L,\ell\in\Nz$ such that $\ell \leq L$,
	the matrices
	$
		\hat{\Ten{M}}_{L \CQ \alpha}
		\,
		\hat{\Ten{P}}_{\ell,L}
	$
	with $\alpha=0,1$,
	where the factors are given by~\eqref{Eq:DefM01} and~\eqref{Eq:DefP},
	admit the representation
	\begin{equation}\label{Eq:DecMPbequadro}
		\hat{\Ten{M}}_{L \CQ \alpha}
		\,
		\hat{\Ten{P}}_{\ell,L}
		=
		2^{ \Par{\alpha+\frac12} \ell }
		\;
		\hat{\Core{A}}
		\RP
		\hat{\Core{U}}^{\RP \ell}
		\RP
		\hat{\Core{T}}_{\alpha}
		\RP
		\hat{\Core{Y}}_{\alpha}^{\RP \Par{L-\ell}}
		\RP
		\hat{\Core{N}}_\alpha
		\, ,
	\end{equation}
	where the cores
	$\hat{\Core{A}}$, $\hat{\Core{U}}$ and $\hat{\Core{T}}_{\alpha}$,
	$\hat{\Core{Y}}_{\alpha}$, $\hat{\Core{N}}_{\alpha}$
	with $\alpha=0,1$
	are as in~\eqref{Eq:DefCoresZZZ0} and~\eqref{Eq:DefCoresZZZ1}.
\end{lemma}
Combining
the decomposition~\eqref{Eq:DecP}
and its transpose,
we can rewrite the product
$
	\hat{\Ten{P}}_{\ell,L}
	\QQ
	\hat{\Ten{P}}_{\ell,L}^\MT
$
core-wise:
	\begin{equation}\label{Eq:DecPP}
		\hat{\Ten{P}}_{\ell,L}
		\,
		\hat{\Ten{P}}_{\ell,L}^\MT
		=
		2^{ - \Par{L-\ell} }
		\;
		\hat{\Core{A}}_\flat
		\RP
		\hat{\Core{U}}_\flat^{\RP \ell}
		\RP
		\hat{\Core{X}}_\flat^ {\RP \Par{L-\ell}}
		\RP
		\hat{\Core{P}}_{\, \flat}
		\, ,
	\end{equation}
where the factors are
\begin{equation}\label{Eq:DefCoresBemolle1}
	\hat{\Core{A}}_\flat
	=
	\hat{\Core{A}} \MP \hat{\Core{A}}
	\, ,
	\quad
	\hat{\Core{U}}_\flat
	=
	\hat{\Core{U}} \MP \hat{\Core{U}}^\MT
	\, ,
	\quad
	\hat{\Core{X}}_\flat
	=
	\hat{\Core{X}} \MP \hat{\Core{X}}^\MT
	\, ,
	\quad
	\hat{\Core{P}}_{\, \flat}
	=
	\hat{\Core{P}} \MP \hat{\Core{P}}
	\, .
\end{equation}
We remark that the ranks of the decomposition~\eqref{Eq:DecPP}
are $4,\ldots,4$.

Applying the same argument to the product
$
		\hat{\Ten{M}}_{L \CQ \alpha}
		\QQ
		\Par{
			\hat{\Ten{P}}_{\ell,L}
			\,
			\hat{\Ten{P}}_{\ell,L}^\MT
		}
$,
the factors $\hat{\Ten{M}}_{L \CQ \alpha}$ and $\hat{\Ten{P}}_{\ell,L} \, \hat{\Ten{P}}_{\ell,L}^\MT$ being taken
in the form of~\eqref{Eq:DecM} and~\eqref{Eq:DecPP},
we could obtain its explicit decomposition with
ranks $2^3,\ldots,2^3$.
Instead, we multiply
$\hat{\Ten{M}}_{L \CQ \alpha} \, \hat{\Ten{P}}_{\ell,L}$ and $\hat{\Ten{P}}_{\ell,L}^\MT$
using the representations~\eqref{Eq:DecMPbequadro} and~\eqref{Eq:DecP}
to form a representation of the same product
$
		\hat{\Ten{M}}_{L \CQ \alpha}
		\,
		\hat{\Ten{P}}_{\ell,L}
		\,
		\hat{\Ten{P}}_{\ell,L}^\MT
$.
This representation has the ranks $2^2,\ldots,2^2,2^{2-\alpha},\ldots,2^{2-\alpha}$,
which means that the ranks of unfolding matrices
$1,\ldots,\ell-1$ and $\ell,\ldots,L-\alpha$ are bounded by $4$ and $2^{2-\alpha}$ respectively.
As we discuss in Section~\ref{Sc:tensorstructureDd} below,
this reduction is substantial in the case of multiple dimensions,
when the exponents
($2$ or $2-\alpha$ instead of $3$)
that correspond to the dimensions
are summed.

Specifically, combining~\eqref{Eq:DecMPbequadro}
and~\eqref{Eq:DecP},
we arrive at
	\begin{equation}\label{Eq:DecBemolle}
		\hat{\Ten{M}}_{L \CQ \alpha}
		\,
		\hat{\Ten{P}}_{\ell,L}
		\,
		\hat{\Ten{P}}_{\ell,L}^\MT
		=
		2^{ \Par{\alpha+\frac12} L - \Par{L-\ell} }
		\;
		\hat{\Core{A}}_\flat
		\RP
		\hat{\Core{U}}_\flat^{\RP \ell}
		\RP
		\hat{\Core{W}}_\alpha
		\RP
		\hat{\Core{Z}}_\alpha^ {\RP \Par{L-\ell}}
		\RP
		\hat{\Core{K}}_\alpha
		\, ,
	\end{equation}
where
\begin{equation}\label{Eq:DefCoresBemolle2}
	\hat{\Core{W}}_\alpha
	=
	\hat{\Core{T}}_{\alpha}
	\MP
	\hat{\Core{I}}
	\, ,
	\quad
	\hat{\Core{Z}}_\alpha
	=
	\hat{\Core{Y}}_{\alpha}
	\MP
	\hat{\Core{X}}^\MT
	\, ,
	\quad
	\hat{\Core{K}}_\alpha
	=
	\hat{\Core{N}}_\alpha
	\MP
	\hat{\Core{P}}
	\quad\text{with}\quad
	\alpha=0,1
	\, .
\end{equation}
The decomposition~\eqref{Eq:DecBemolle} is exact and explicit,
the latter meaning that all the cores involved are
provided in closed form.
Since $\hat{\Core{U}}_\flat$ and $\hat{\Core{Y}}_{\alpha}$
are of ranks $2^2 \times 2^2$ and $2^{2-\alpha} \times 2^{2-\alpha}$
respectively,
the ranks of the decomposition~\eqref{Eq:DecBemolle} are
$2^2,\ldots,2^2,2^{2-\alpha},\ldots,2^{2-\alpha}$.

Direct calculation with expressions given in~\eqref{Eq:DefCoresZZZ0}--\eqref{Eq:DefCoresZZZ1}
leads to
$
		\hat{\Core{A}}_\flat
		=
		\begin{bmatrix}
			1 & 0 & 0 & 0 \\
		\end{bmatrix}
$,
	\begin{equation}\label{Eq:DefCoresBemolleExpl}
		\hat{\Core{P}}_{\, \flat}
		=
		\begin{bmatrix}
			1	\\
			0	\\
			0	\\
			0	\\
		\end{bmatrix}
		\, ,
		\quad
		\hat{\Core{W}}_0
		=
		\begin{bmatrix*}[r]
			1 & 0 &  1 &  0 \\
			0 & 1 &  0 &  1 \\
			1 & 0 & -1 &  0 \\
			0 & 1 &  0 & -1 \\
		\end{bmatrix*}
		\quad\text{and}\quad
		\hat{\Core{W}}_{ 1}
		=
		\begin{bmatrix*}[r]
			 1 &  0 \\
			 0 &  1 \\
			-1 &  0 \\
			 0 & -1 \\
		\end{bmatrix*}
		\, .
	\end{equation}
Explicit expression for $\hat{\Core{U}}_\flat$,
$\hat{\Core{X}}_\flat$
and
$\hat{\Core{Z}}_\alpha$, $\hat{\Core{K}}_\alpha$
with $\alpha=0,1$
can be likewise calculated
based on~\eqref{Eq:DefCoresZZZ0} and~\eqref{Eq:DefCoresZZZ1},
from which we refrain to keep exposition concise.

\subsection{Analysis in $D$ dimensions by tensorization}\label{Sc:tensorstructureDd}

In this section, we generalize
the results of
Sections~\ref{Sc:tensorstructure1dprec}
to the case of multiple dimensions
and analyze the
low-rank tensor structure
of the preconditioner $\Ten{C}_{L}$,
given by~\eqref{Cdef},
and of the preconditioned discrete differential operator
$\Ten{B}_{L}$ in the form of~\eqref{Eq:PrecOpDec}.
For the latter, we first derive a representation of
the matrices $\Ten{Q}_{L \CQ \alpha}$
with $L\in\N$ and $\alpha\in\Set{0,1}^D$,
which are defined in~\eqref{Eq:DefQ}.

The representations derived below are composed from
the following cores:
\begin{equation}\label{Eq:DefCoresBemolleD}
	\begin{gathered}
		\Core{A}_\flat
		=
		\hat{\Core{A}}_\flat^{\QQ \otimes \QQ D}
		, \quad
		\Core{U}_\flat
		=
		\hat{\Core{U}}_\flat^{\QQ \otimes \QQ D}
		, \quad
		\Core{X}_\flat
		=
		\hat{\Core{X}}_\flat^{\QQ \otimes \QQ D}
		, \quad
		\Core{P}_{\, \flat}
		=
		\hat{\Core{P}}_{\, \flat}^{\QQ \otimes \QQ D}
		,
		\\
		\Core{W}_\alpha
		=
		\KProd_{k=1}^D
		\hat{\Core{W}}_{\alpha_k}
		, \quad
		\Core{Z}_\alpha
		=
		\KProd_{k=1}^D
		\hat{\Core{Z}}_{\alpha_k}
		, \quad
		\Core{K}_\alpha
		=
		\KProd_{k=1}^D
		\hat{\Core{K}}_{\alpha_k}
	\end{gathered}
\end{equation}
for all
$\alpha\in\Set{0,1}^D$,
where the factors are given by~\eqref{Eq:DefCoresBemolle1} and~\eqref{Eq:DefCoresBemolle2}.

Tensorizing~\eqref{Eq:DecPP} core-wise and distributing the scaling factor
over the cores, we obtain the decompositions
	\begin{equation}\label{Eq:DecPPD}
		2^{-\ell}
		\QQ
		\Ten{\varPi}_{L}
		\QQ
		{\Ten{P}}_{\ell,L}
		\QQ
		{\Ten{P}}_{\ell,L}^\MT
		\,
		\Ten{\varPi}_{L}^\MT
		\begin{aligned}[t]
			&=
			2^{ -\ell - D\Par{L-\ell}}
			\,
			\Core{A}_\flat
			\RP
			\Core{U}_\flat^{\RP \ell}
			\RP
			\Core{X}_\flat^{\RP \Par{L-\ell}}
			\RP
			\Core{P}_{\, \flat}
			\\
			&=
			2^{ -\ell}
			\,
			\Core{A}_\flat
			\RP
			\Core{U}_\flat^{\RP \ell}
			\RP
			\Par{
				2^{-D}
				\,
				\Core{X}_\flat
			}^{\RP \Par{L-\ell}}
			\RP
			\Core{P}_{\, \flat}
		\end{aligned}
	\end{equation}
of ranks $2^{2D},\ldots,2^{2D}$,
where the cores are given by~\eqref{Eq:DefCoresBemolleD}
and the permutation matrix $\Ten{\varPi}_{L}$
is as defined in~\eqref{Eq:IndTranspMatrix}.
Applying~\cite[Lemma~5.5]{KKh:2012:QTT} to the sum of such matrices
with $\ell=1,\ldots,L$ and adding the term corresponding to $\ell=0$,
we obtain the following result.
\begin{theorem}\label{Th:DecC}
	For any $L\in\N$, the matrix $\Ten{C}_{L}$,
	defined by~\eqref{Cdef}, admits
	the decomposition
	\begin{equation}\label{Eq:DecC}
		\Ten{\varPi}_{L}
		\QQ
		\Ten{C}_{L}
		\Ten{\varPi}_{L}^\MT
		=
		\begin{bmatrix}
			\Core{A}_\flat
			&
			\Core{A}_\flat
		\end{bmatrix}
		\RP
		\Core{C}_1
		\RP \cdots \RP
		\Core{C}_L
		\RP
		\begin{bmatrix}
			\\
			\Core{P}_{\, \flat}
		\end{bmatrix}
	\end{equation}
	of ranks $2^{2D}+2^{2D},\ldots,2^{2D}+2^{2D}$,
	all equal to $2^{2D+1}$,
	where the middle cores are
	\begin{equation}\nonumber
		\Core{C}_{\ell}
		=
		\begin{bmatrix*}[r]
			\Core{U}_\flat
			&
			2^{-\ell}
			\,
			\Core{U}_\flat
			\\
			&
			2^{-D}
			\,
			\Core{X}_\flat
		\end{bmatrix*}
		\quad\text{with}\quad
		\ell=1,\ldots,L
		\, ,
	\end{equation}
	the subcores being as in~\eqref{Eq:DefCoresBemolleD}.
\end{theorem}

For any $L\in\N$, $\ell=0,1\ldots,L$ and $\alpha\in\Set{0,1}^D$,
tensorizing~\eqref{Eq:DecBemolle} core-wise
and distributing the scaling factor
over the cores
results in the decompositions
\begin{multline}\label{Eq:DecBemolleD}
	2^{-\ell}
	\,
	\Ten{\widetilde{\varPi}}_{L \CQ \alpha}
	\,
	\Ten{M}_{L \CQ \alpha}
	\,
	\Ten{P}_{\ell,L}
	\,
	\Ten{P}_{\ell,L}^\MT
	\,
	\Ten{\varPi}_{L}^\MT
	\\
	\begin{aligned}[t]
		&=
		2^{ -\ell + \Par{\IndNorm{\alpha}+\frac12 D} L - D\Par{L-\ell} }
		\;
		\Core{A}_\flat
		\RP
		\Core{U}_\flat^{\RP \ell}
		\RP
		\Core{W}_{\!\QQ \alpha}
		\RP
		\Core{Z}_\alpha^ {\RP \Par{L-\ell}}
		\RP
		\Core{K}_\alpha
		\\
		&=
		2^{ -\Par{1-\IndNorm{\alpha}} \QQ \ell }
		\;
		\Core{A}_\flat
		\RP
		\Par{
			2^{\frac12 D}
			\,
			\Core{U}_\flat
		}^{\RP \ell}
		\RP
		\Core{W}_{\!\QQ \alpha}
		\RP
		\Par{
			2^{\IndNorm{\alpha}-\frac12 D}
			\,
			\Core{Z}_\alpha
		}^{\RP \Par{L-\ell}}
		\RP
		\Core{K}_\alpha
	\end{aligned}
\end{multline}
of ranks
$2^{2D},\ldots,2^{2D},2^{2D-\IndNorm{\alpha}},\ldots,2^{2D-\IndNorm{\alpha}}$,
where
$\Ten{M}_{L \CQ \alpha}$
and
$\Ten{P}_{\ell \CQ L}$
are given by~\eqref{Eq:DefMD} and~\eqref{Eq:DefPD},
the cores are given by~\eqref{Eq:DefCoresBemolleD}
and the permutation matrices $\Ten{\varPi}_{L}$ and
$\Ten{\widetilde{\varPi}}_{L \CQ \alpha}$
are as defined in~\eqref{Eq:IndTranspMatrix} and~\eqref{Eq:IndTranspMatrixTilde}.

Similarly as for $\Ten{C}_{L}$ above,
we can apply~\cite[Lemma~5.5]{KKh:2012:QTT} to the sum of the matrices
given by~\eqref{Eq:DecBemolleD}
with $\ell=1,\ldots,L$ and add the term corresponding to $\ell=0$.
This leads to the following result, which is analogous to Theorem~\ref{Th:DecC}.
\begin{theorem}\label{Th:DecQ}
	For any $L\in\N$ and $\alpha\in\Set{0,1}^D$,
	the matrix $\Ten{Q}_{L \CQ \alpha}$,
	given by~\eqref{Eq:DefQ}, admits
	the decomposition
	\begin{equation}\label{Eq:DecQ}
			\Ten{\widetilde{\varPi}}_{L \CQ \alpha}
			\QQ
			\Ten{Q}_{\ell \CQ L \CQ \alpha}
			\QQ
			\Ten{\varPi}_{L}^\MT
			=
		\begin{bmatrix*}[r]
			\Core{A}_\flat
			&
			\Core{A}_\flat
			\RP
			\Core{W}_\alpha
		\end{bmatrix*}
		\RP
		\Core{Q}_1
		\RP \cdots \RP
		\Core{Q}_L
		\RP
		\begin{bmatrix}
			\\
			\Core{K}_\alpha
		\end{bmatrix}
	\end{equation}
	of ranks $2^{2D}+2^{2D-\IndNorm{\alpha}},\ldots,2^{2D}+2^{2D-\IndNorm{\alpha}}$,
	all bounded from above by $2^{2D+1}$,
	where the middle cores are
	\begin{equation}\nonumber
		\Core{Q}_{\ell}
		=
		\begin{bmatrix*}[r]
			\Core{U}_\flat
			&
			2^{-\Par{1-\IndNorm{\alpha}} \QQ \ell}
			\,
			\Core{U}_\flat
			\RP
			\Core{W}_\alpha
			\\
			&
			2^{\IndNorm{\alpha}-\frac12 D}
			\,
			\Core{Z}_\alpha
		\end{bmatrix*}
		\quad\text{with}\quad
		\ell=1,\ldots,L
		\, ,
	\end{equation}
	the subcores being defined by~\eqref{Eq:DefCoresBemolleD}.
\end{theorem}
\begin{subequations}
	In Example~\ref{Ex:LambdaLaplace},
	the case of the Laplace operator was considered
	and
	the factors
	$\Ten{\varLambda}_{L \CQ \alpha \alpha'}$
	with $\Tuple{\alpha,\alpha'} \in \mathcal{D}$
	for the suitable $\mathcal{D}$ were
	explicitly given in the Kronecker product
	form~\eqref{Eq:LambdaLaplaceKronProd}.
	That form immediately
	leads to a multilevel TT decomposition of ranks $1,\ldots,1$
	for each $\Ten{\varLambda}_{L \CQ \alpha \alpha'}$.
	Here, we analyze the structure of
	$\Ten{\varLambda}_{L \CQ \alpha \alpha'}$
	with $\Tuple{\alpha,\alpha'} \in \mathcal{D}$
	in the general setting of
	Section~\ref{Sc:FactDiffOp},
	for an arbitrary
	$\mathcal{D} \subset \Set{0,1}^D \times \Set{0,1}^D$
	of differential indices,
	under the additional assumption that
	the coefficient functions~\eqref{Eq:DefBLFc} exhibit low-rank structure.

	Specifically, for each $\Tuple{\alpha,\alpha'} \in \mathcal{D}$,
	we assume that
	the coefficient vector
	$
		\Vec{c}_{L \CQ \alpha \alpha'}
		\in
		\R^{\cJ_L \times \varGamma_{\! \alpha \alpha'}}
		\simeq
		\R^{2^{DL} R_{\alpha \alpha'}}
	$
	parametrizing the
	coefficient function $c_{\alpha \alpha'}$ through~\eqref{Eq:DefBLFc}
	is given in a multilevel TT representation of ranks
	$r_{0 \CQ \alpha \alpha'},\ldots,r_{L \CQ \alpha \alpha'}$:
	\begin{equation}\label{Eq:DecBLFc}
		\Ten{\widetilde{\varPi}}_{L \CQ \alpha}
		\,
		\Vec{c}_{L \CQ \alpha \alpha'}
		=
		C_{L \CQ 0 \CQ \alpha \CQ \alpha'}
		\RP
		C_{L \CQ 1 \CQ \alpha \CQ \alpha'}
		\RP \cdots \RP
		C_{L \CQ L \CQ \alpha \CQ \alpha'}
		\RP
		C_{L \CQ L+1 \CQ \alpha \CQ \alpha'}
		\, ,
	\end{equation}
	where each of
	$C_{L \CQ 1 \CQ \alpha \CQ \alpha'},\ldots,C_{L \CQ L \CQ \alpha \CQ \alpha'}$
	is of mode size $2^D$, whereas
	$C_{L \CQ 0 \CQ \alpha \CQ \alpha'}$
	is of mode size $1$
	and
	$C_{L \CQ L+1 \CQ \alpha \CQ \alpha'}$
	is of mode size $R_{\alpha \alpha'}=\Card{\varGamma_{\! \alpha \alpha'}}$.
	Then the corresponding factor $\Ten{\varLambda}_{L \CQ \alpha \alpha'}$,
	given by~\eqref{Eq:OpDecTermwiseLambda}, can as well be represented
	with ranks $r_{0 \CQ \alpha \alpha'},\ldots,r_{L \CQ \alpha \alpha'}$:
	\begin{equation}\label{Eq:DecTermwiseLambda}
		\Ten{\widetilde{\varPi}}_{L \CQ \alpha}
		\,
		\Ten{\varLambda}_{L \CQ \alpha \CQ \alpha'}
		\QQ
		\Ten{\widetilde{\varPi}}_{L \CQ \alpha}^\MT
		=
		\varLambda_{L \CQ 0 \CQ \alpha \CQ \alpha'}
		\RP
		\varLambda_{L \CQ 1 \CQ \alpha \CQ \alpha'}
		\RP \cdots \RP
		\varLambda_{L \CQ L \CQ \alpha \CQ \alpha'}
		\RP
		\varLambda_{L \CQ L+1 \CQ \alpha \CQ \alpha'}
		\, ,
	\end{equation}
	where the cores are defined in terms of those appearing in~\eqref{Eq:DecBLFc}
	as follows.
	First, one sets
	$
		\varLambda_{L \CQ 0 \CQ \alpha \CQ \alpha'}
		=
		C_{L \CQ 0 \CQ \alpha \CQ \alpha'}
	$
	and defines each core
	$\varLambda_{L \CQ \ell \CQ \alpha \CQ \alpha'}$
	with $\ell=1,\ldots,L$ by
	\begin{equation}\label{Eq:DecTermwiseLambdaMid}
		\Par{\varLambda_{L \CQ \ell \CQ \alpha \CQ \alpha'}}_{
			\gamma_{\ell-1} \; i_\ell \; i'_\ell \; \gamma_\ell
		}
		=
		2^{-D}
		\,
		\delta_{i_\ell \QQ i'_\ell}
		\,
		\Par{C_{L \CQ \ell \CQ \alpha \CQ \alpha'}}_{
			\gamma_{\ell-1} \; i_\ell \; \gamma_\ell
		}
	\end{equation}
	for all $\gamma_{\ell-1}=1,\ldots,r_{\ell-1 \CQ \alpha \alpha'}$,
	$\gamma_\ell=1,\ldots,r_{\ell \CQ \alpha \alpha'}$
	and
	$i_\ell,i'_\ell =1,2$.
	Then the last core should be defined by
	\begin{equation}\label{Eq:DecTermwiseLambdaLast}
		\Par{\varLambda_{L \CQ L+1 \CQ \alpha \CQ \alpha'}}_{
			\gamma_L \; \beta \; \beta'
		}
		=
		2^{-D}
		\!\!
		\sum_{\gamma\in\varGamma_{\! \alpha \alpha'}}
		\!\!
		\Par{C_{L \CQ L+1 \CQ \alpha \CQ \alpha'}}_{ \gamma_L \; \gamma }
		\!\!\!\!
		\int\displaylimits_{\IntOO{-1}{1}^D}
		\!\!\!\!\!\!
		\chi_{\alpha \alpha' \gamma}
		\,
		\Par{\partial^{\QQ \alpha} \psi_{\beta}}
		\,
		\Par{\partial^{\QQ \alpha' \!} \psi_{\beta' \!}}
	\end{equation}
	for all
	$\gamma_L = 1,\ldots,r_{L \CQ \alpha \alpha'}$,
	$\beta\in\Set{\alpha_1,1} \times\cdots\times \Set{\alpha_D,1}$
	and
	$\beta'\in\Set{\alpha'_1,1} \times\cdots\times \Set{\alpha'_D,1}$,
	cf.~\eqref{Eq:OpDecTermwiseLambda}.
\end{subequations}

Using the fact that the ranks add under addition and multiply
under multiplication~\cite{Oseledets:2011:TT}, we obtain the following result.

\begin{theorem}\label{Th:DecB}
	For $\mathcal{D} \subset \Set{0,1}^D \times \Set{0,1}^D$
	and $L\in\N$, consider a bilinear form
	of the type~\eqref{Eq:DefBLF2}--\eqref{Eq:DefBLFc},
	where each coefficient vector $\Vec{c}_{L \CQ \alpha \alpha'}$
	with $\Tuple{\alpha,\alpha'} \in \mathcal{D}$
	admits a multilevel TT decomposition of the form~\eqref{Eq:DecBLFc}
	with ranks
	$r_{0 \CQ \alpha \alpha'},\ldots,r_{L \CQ \alpha \alpha'}$
	not exceeding $r \in\N$.
	Then the preconditioned matrix $\Ten{B}_L$ of $a$,
	defined by~\eqref{Eq:DefBLFA}, \eqref{Cdef}
	and~\eqref{precondsystem},
	admits a multilevel TT decomposition
	\begin{equation}\nonumber
		\Ten{\varPi}_{L}
		\,
		\Ten{B}_{L}
		\QQ
		\Ten{\varPi}_{L}^\MT
		=
		B_{L \CQ 0}
		\RP
		B_{L \CQ 1}
		\RP \cdots \RP
		B_{L \CQ L}
		\RP
		B_{L \CQ L+1}
	\end{equation}
	of ranks
	$R_0,\ldots,R_L$, where
	\begin{equation}\label{Eq:RankBoundB}
		R_\ell
		=
		2^{\QQ 4D}
		\!\!
		\sum_{\Tuple{\alpha,\alpha'} \in \mathcal{D}}
		\!\!
		\Par[1]{ \QQ 1+2^{-\IndNorm{\alpha}} \QQ }^2
		\:
		r_{\ell \CQ \alpha \alpha'}
		\leq
		2^{\QQ 4D + 2}
		\!\!
		\sum_{\Tuple{\alpha,\alpha'} \in \mathcal{D}}
		\!\!
		r_{\ell \CQ \alpha \alpha'}
		\leq
		12 \QQ D^{\QQ 2}
		\,
		2^{\QQ 4D}
		\,
		r
	\end{equation}
	for $\ell=0,\ldots,L$.
\end{theorem}
\begin{remark}[sharper bounds in specific cases]
			The last inequality of~\eqref{Eq:RankBoundB}
			is given for a general case with $D^2$ second-order terms
			(no symmetry is assumed),
			$D$ first-order terms and a zero-order term.
			However, for the Laplacian in the case $D=2$,
			the first equality given in~\eqref{Eq:RankBoundB} results in
			$R_\ell = 1152$, which is a marked reduction from the bound $R_\ell \leq 12288$ obtained for a general second-order bilinear form
			with constant coefficients.
\end{remark}

\begin{remark}[inexact application]
			In computations, algorithms using
			products of $\Ten{B}_L$ with vectors
			rather than explicit representations of $\Ten{B}_L$
			may be expected to be more efficient.
			Indeed, such products can be formed by adding the products of the terms
			in the sum~\eqref{Eq:PrecOpDec},
			and for each term the product can be computed
			by three multiplications.
			On the intermediate results obtained between these multiplications and additions, low-rank re-approximation
			can be performed, as explained further in the example of the discretized Laplacian in Section \ref{sec:stabillustr}.
			The given bounds for TT ranks appear to be highly pessimistic
			for such inexact schemes.
\end{remark}

\begin{remark}
	The analysis in $D$ dimensions is given here for the most generic
			discretization obtained by tensorization.
			The approach can be applied to discretizations that are not of tensor product form
			in order to mitigate the growth of the rank bounds
			with respect to $D$.
\end{remark}

\subsection{Numerical illustrations}\label{sec:stabillustr}

In summary, we obtain a combined tensor representation $\Rep{B}_L$ with
$\tau(\Rep{B}_L) = \Ten{\varPi}_{L} \QQ \Ten{B}_L \QQ \Ten{\varPi}_{L}^\MT = \Ten{\varPi}_{L} \QQ \Par{\Ten{C}_L\Ten{A}_L\Ten{C}_L} \QQ \Ten{\varPi}_{L}^\MT$. Similarly, from Theorem \ref{Th:DecC} we also have $\Rep{C}_L$ with $\tau(\Rep{C}_L) = \Ten{\varPi}_{L} \QQ \Ten{C}_L \QQ \Ten{\varPi}_{L}^\MT$. With a representation $\Rep{A}_L$ of the stiffness matrix $\Ten{A}_L$, such that $\tau(\Rep{A}_L) = \Ten{\varPi}_{L} \QQ \Ten{A}_L \QQ \Ten{\varPi}_{L}^\MT$, one can alternatively consider the simple product representation $\Rep{C}_L \MP \Rep{A}_L \MP \Rep{C}_L$, which corresponds to performing the action of the preconditioner $\Ten{C}_L$ separately from that of $\Ten{A}_L$.

Note that, in Section \ref{sec:opramp}, we have assumed decompositions consisting of $L$ cores. The decompositions in Theorems \ref{Th:DecC}, \ref{Th:DecQ}, and \ref{Th:DecB} comprise $L+2$ cores, with first and last playing special roles since they can be merged with the respective adjacent cores. The cores in these extended decompositions are thus indexed by $\ell=0,\ldots,L+1$ in what follows, so that again the bounds for $\ell=1,\ldots,L$ are relevant.

One benefit of the combined representation $\Rep{B}_L$ is the rank reduction compared to $\Rep{C}_L \MP \Rep{A}_L \MP \Rep{C}_L$. More importantly, however, the decomposition $\Rep{B}_L$ is constructed so that the representation condition numbers $\oprcond_\ell(\Rep{B}_L)$, $\ell=1,\ldots,L$, remain moderate even for large $L$. In contrast, the representation condition numbers of $\Rep{C}_L \MP \Rep{A}_L \MP \Rep{C}_L$ are in general of the same order of magnitude as those of $\Rep{A}_L$ -- in other words, whereas the \emph{matrix} condition number of $\Ten{C}_L \Ten{A}_L \Ten{C}_L$ is uniformly bounded, for improving also the \emph{representation} condition number, applying the preconditioner $\Ten{C}_L$ separately is insufficient and one instead needs a carefully constructed \emph{combined} representation $\Rep{B}_L$.

We now present numerical observations that illustrate how different
the decompositions
$\Rep{A}_L$, $\Rep{C}_L \MP \Rep{A}_L \MP \Rep{C}_L$ and $\Rep{B}_L$
are in terms of representation conditioning
and demonstrate the improvement afforded by our findings presented
in Sections \ref{sec:repcond}, \ref{Sc:tensorstructure1d}, and \ref{Sc:tensorstructureDd}.
As in Example~\ref{Ex:LambdaLaplace},
we consider the case of the Laplacian:
$\Ten{A}_L= \Ten{D}_L$ with $\Ten{D}_L$ as in \eqref{Eq:DL}.
Using \eqref{Eq:DefBlocksIJ}, for $D=1$ we have $\Ten{A}_L =  A_1 \RP \cdots \RP A_L$ with
	$A_1 = 4 \, \SqBr{ \, I \;\, J^\MT \;\, J \;\, I_2 \,}$,
	\begin{equation*}
			A_2 =\cdots =A_{L-1} =
	    4 \begin{bmatrix}
			 I &  J^\MT & J  & \\
	       &  J &  & \\
	       &   & J^\MT & \\
	       &  &  & I_2
	        \end{bmatrix},\quad
	   A_L = 4 \begin{bmatrix} 2I - J - J^\MT \QQ \\  -J \\ -J^\MT  \\ - I_2 \end{bmatrix},
	\end{equation*}
	as derived in \cite{KKh:2012:QTT}; similar representations can be obtained for $D>1$
	by tensorization.

We first consider the upper bounds $\beta_\ell$, defined in \eqref{betadef}, for $\opramp_\ell$ from Proposition \ref{matrepbound}. Since both $\norm{\Ten{A}_L^{-1}}$ and $\norm{\Ten{B}_L^{-1}}$ are bounded independently of $L$, by \eqref{oprcondbound}, up to fixed constants the respective $\beta_\ell$ are also upper bounds of the corresponding representation condition numbers $\oprcond_\ell$.

For $\Rep{B}_L$, instead of directly computing the estimates for
$\opramp_\ell(\Rep{B}_L)$ with $\ell=1,\ldots,L$
given by Proposition~\ref{matrepbound},
we will do this for the factors of a
decomposition that is equivalent to $\Rep{B}_L$
and is also based on~\eqref{Eq:PrecOpDec}.
Let us note that
the equality
\begin{equation}\label{Eq:DecBTheta}
	\Rep{B}_L = \sum_{k=1}^D \Rep{\Theta}_{L \CQ k}^\MT \MP \Rep{\Theta}_{L \CQ k}
\end{equation}
of decompositions
holds in terms of the factors
$\Rep{\Theta}_{L \CQ k}$
with $k=1,\ldots,D$
given as follows: for every $k$, we set
$
	\Rep{\Theta}_{L \CQ k}
	=
	\Rep{\Lambda}_{L \CQ k}^{1/2}
	\MP
	\Rep{Q}_{L \CQ \alpha}
$
with $\alpha = \Tuple{\delta_{k 1},\ldots,\delta_{k D}}$,
where $\Rep{\Lambda}_{L \CQ \alpha \CQ \alpha}^{1/2}$
is the decomposition of $\Ten{\Lambda}_{L \CQ \alpha \CQ \alpha}^{1/2}$,
which is diagonal
and of Kronecker product form~\eqref{Eq:LambdaLaplaceKronProd};
thus its decomposition with ranks $1,\ldots,1$ is obtained
by element-wise application of the square root to each core.
Equality~\eqref{Eq:DecBTheta}
results in the second of the following inequalities:
\begin{equation}\label{Eq:betaBndThetaB}
	\max_{\ell=1,\ldots,L} \oprcond_\ell (\Rep{B}_L)
	\lesssim
	\max_{\ell=1,\ldots,L} \opramp_\ell (\Rep{B}_L)
	\lesssim
	\max_{\ell=1,\ldots,L}[\beta_\ell( \Rep{\Theta}_{L,1} )]^2
	\, ,
\end{equation}
where
the equivalence is uniform with respect to $L\in\N$
and,
for each $L\in\N$,
$\beta_\ell$ with $\ell=1,\ldots,L$ are as defined in~\eqref{betadef}.
As well as in~\eqref{Eq:betaBndThetaB},
the alternate form~\eqref{Eq:DecBTheta} of $\Rep{B}_L$ is used to improve
the efficiency of residual approximation
in the numerical tests of Section \ref{sec:numexp}.

Figure \ref{fig:stabest}(a) shows the computed values of $\max_\ell \beta_\ell( \Rep{\Theta}_{L\CQ 1} )$ for different values of $L$ and $D=1,2$, where we observe $\max_\ell \beta_\ell( \Rep{\Theta}_{L\CQ 1} ) = \cO(L)$ in both cases, corresponding to
\[
 \max_{\ell=1,\ldots,L} \oprcond_\ell (\Rep{B}_L)\lesssim \max_{\ell=1,\ldots,L} \opramp_\ell (\Rep{B}_L) \lesssim \max_{\ell=1,\ldots,L} \beta_\ell (\Rep{B}_L) \lesssim L^2.
\]
In contrast, as shown in Figure \ref{fig:stabest}(b), both $\max_\ell \beta_\ell(\Rep{A}_L)$ and $\max_\ell \beta_\ell(\Rep{C}_L \MP \Rep{A}_L \MP \Rep{C}_L)$ increase exponentially with respect to $L$.

Although Proposition \ref{prop:lapddrepcond} shows that they can lead to useful qualitative statements, the upper bounds provided by $\beta_\ell$ cannot be expected to be quantitatively sharp.
The direct evaluation of the suprema in the definitions \eqref{Eq:oprconddef} is in general infeasible, but testing with concrete $\Rep{V} \in \TTset_L$ can provide some further insight.
For $D=1$, we use TT-SVD representations $\Rep{V}_1$, $\Rep{V}_\mathrm{min}$, $\Rep{V}_\mathrm{max}$ (of maximum ranks $1$, $2$, and $2$, respectively) of the vectors
\begin{gather*}
 \Vec{v}_1 =  \bigl( c_1 \bigr)_{k=1,\ldots,2^L}, \quad \Vec{v}_{\mathrm{min}} = \bigl(c_\mathrm{min} \sin(\textstyle\frac\pi2\displaystyle x_i) \bigr)_{i=1,\ldots,2^L}, \\
  \Vec{v}_{\mathrm{max}} =  \bigl( c_\mathrm{max} \sin(\textstyle\frac\pi2\displaystyle (1+2^{L+1})x_i) \bigr)_{i=1,\ldots,2^L},
 \end{gather*}
 with $x_i = 2^{-L}i$ and with constants $c_1$, $c_\mathrm{min}$, $c_\mathrm{max}$ chosen so that $\norm{\Vec{v}_1}_2 = \norm{\Vec{v}_{\mathrm{min}}}_2 = \norm{\Vec{v}_{\mathrm{max}}}_2 = 1$. By Proposition \ref{prop:rcond}(iii), $\rcond_\ell(\Rep{V}_1) = 1$ and $1\leq \rcond_\ell(\Rep{V}_\mathrm{min}) \leq \sqrt{2}$, $1\leq \rcond_\ell(\Rep{V}_\mathrm{max})\leq \sqrt{2}$. Consequently, as in the examples of Section \ref{sec:examples}, for each such a choice of $\Rep{V}$ and any representation of a matrix $\Rep{M}$, the absolute and relative errors incurred by the orthogonalization of $\Rep{M} \MP \Rep{V}$
 give an indication of the order of magnitude of $\ramp_\ell(\Rep{M} \MP \Rep{V})$ and $\rcond_\ell(\Rep{M} \MP \Rep{V})$.

The results are summarized in Tables~\ref{tab:absstabtest} and~\ref{tab:relstabtest}. We see that in all cases, the absolute and relative errors for $\Rep{B}_L$ are close to machine precision $\epsilon \approx 2.2\times 10^{-16}$, which is quantitatively better than indicated by the upper bounds in Figure \ref{fig:stabest}. For $\Rep{A}_L$ and $\Rep{C}_L \MP \Rep{A}_L \MP \Rep{C}_L$, we observe an amplification of relative errors that is exponential in $L$ (and in fact slightly worse for $\Rep{C}_L \MP \Rep{A}_L \MP \Rep{C}_L$). The absolute errors for $\Rep{C}_L$ are close to $\epsilon$, which is important for the evaluation of preconditioned right-hand sides; the corresponding relative errors increase with $L$ in the case of $\Rep{V}_\mathrm{max}$, which is to be expected since $\Ten{C}_L$ damps high-frequency oscillations.

\begin{figure}\centering
\begin{tabular}{cc}
	\includegraphics[width=6cm]{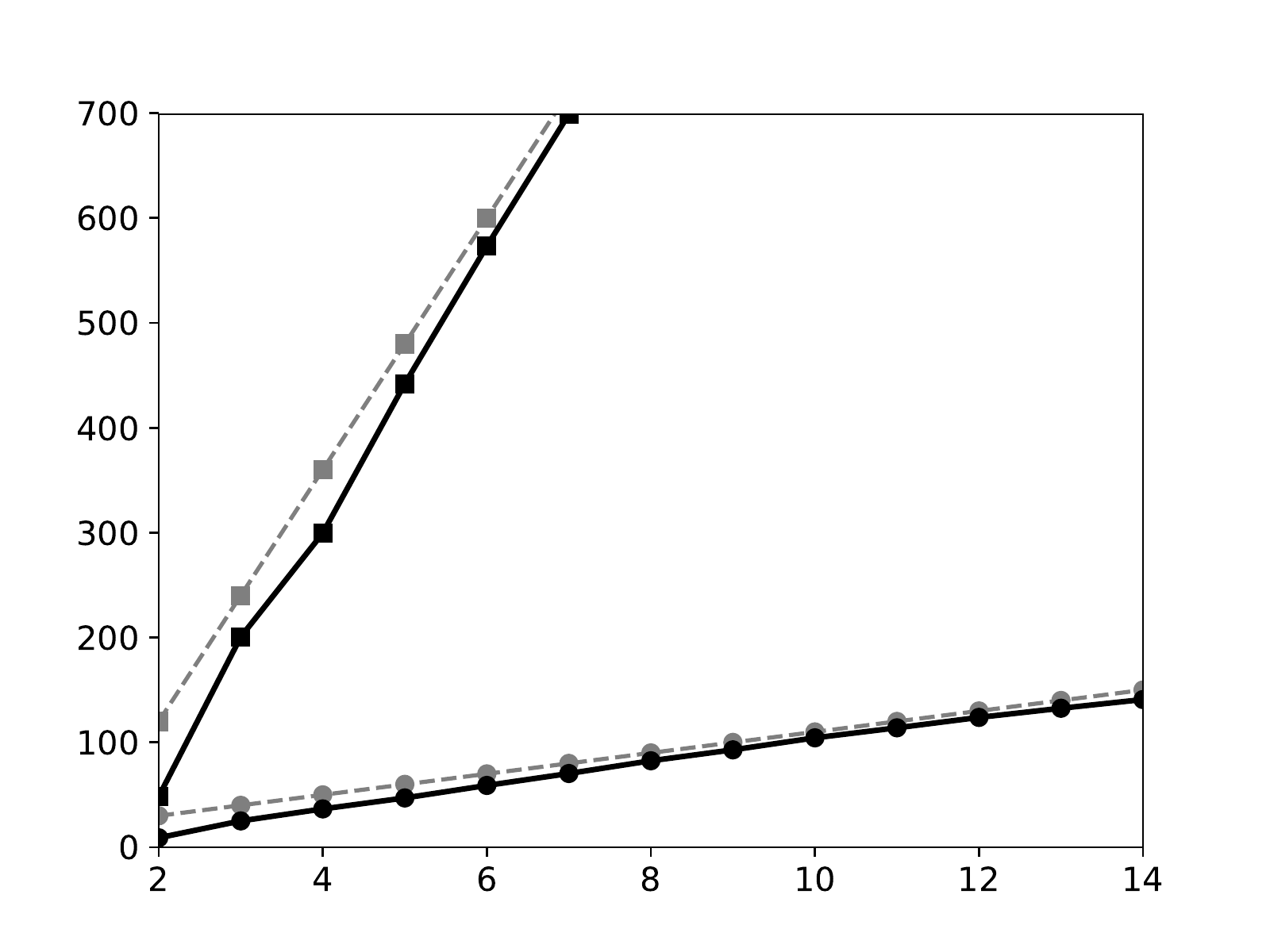} &
	\includegraphics[width=6cm]{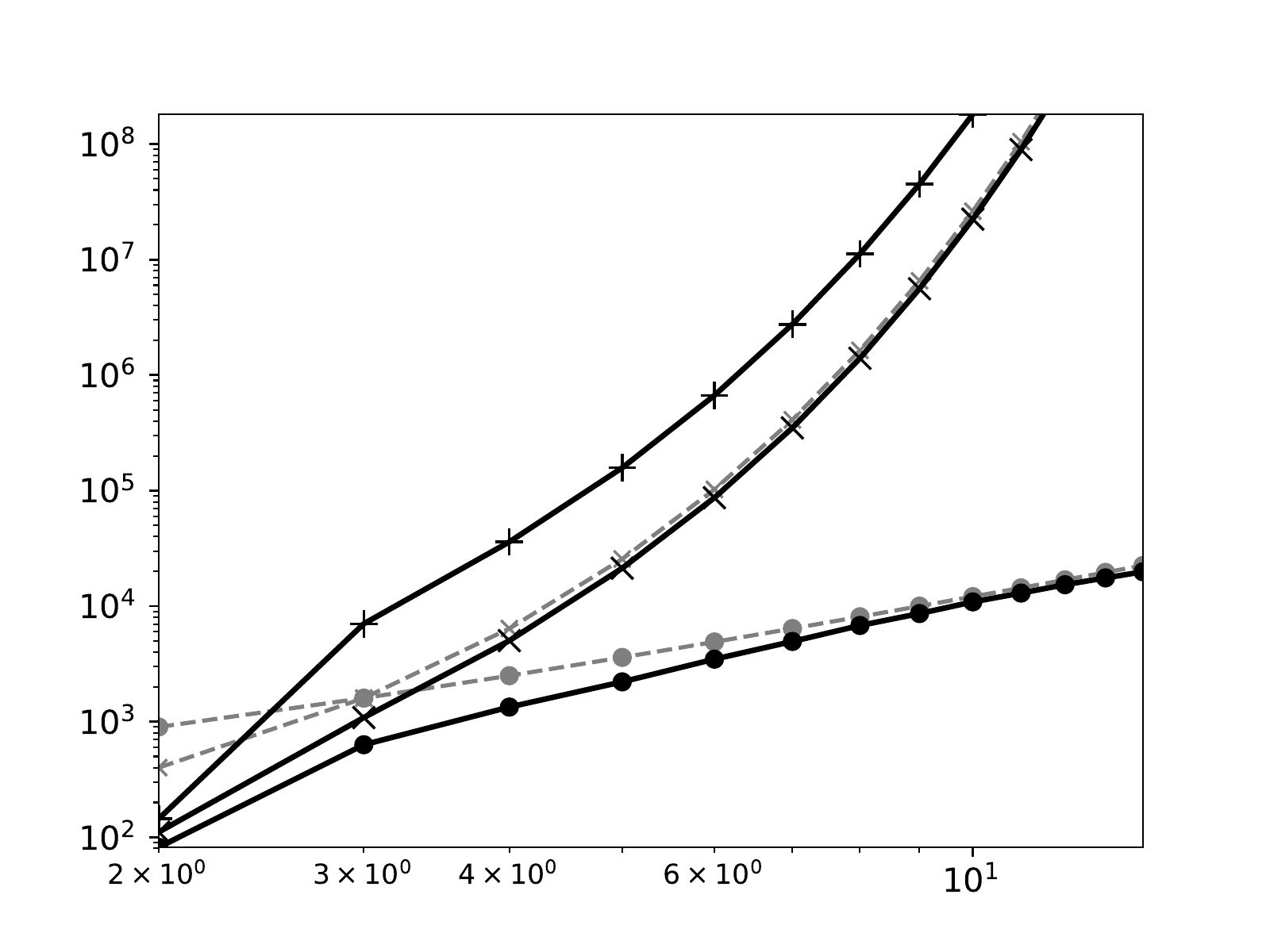} \\[-6pt]
	\footnotesize(a)  & \footnotesize(b)
\end{tabular}
\caption{Upper bounds $\max_{\ell=1,\ldots,L} \beta_\ell$ as in \eqref{betadef} for $\max_{\ell=1,\ldots,L} \opramp_\ell$ from Proposition \ref{matrepbound}, in dependence of $L$: (a) $\max_\ell \beta_\ell( \Rep{\Theta}_{L,1} )$ for $D=1$ (circles) and $D=2$ (squares), with dashed lines representing $10 (L+1)$ and $120 (L-1)$, respectively; (b) $\max_\ell [\beta_\ell( \Rep{\Theta}_{L,1} )]^2$ (circles), $\max_\ell \beta_\ell( \Rep{A}_L )$ (crosses), and $\max_\ell \beta_\ell( \Rep{C}_L \MP \Rep{A}_L \MP \Rep{C}_L )$ (plusses), for $D=1$, with dashed lines representing $(11 L)^2$ and $25\times 2^{2L}$, respectively.
The quantities $\max_\ell [\beta_\ell( \Rep{\Theta}_{L,1} )]^2$
bound
$\max_\ell [\beta_\ell( \Rep{B}_{L} )]$
up to a constant independent of $L$, see~\eqref{Eq:betaBndThetaB}.
}\label{fig:stabest}
\end{figure}

\begin{table}
\centering\small
	\begin{tabular}{ll|n{1}{2}|n{1}{2}|n{1}{2}}
 $\Rep{V}$	& $\Rep{M}$	 & \multicolumn{1}{c}{$L=20$}  & \multicolumn{1}{c}{$L=30$} & \multicolumn{1}{c}{$L=40$} \\ \hline
 $\Rep{V}_1$	&  $\Rep{B}_L$	 & 1.47e-14 & 2.08e-14 & 3.30e-14  \\
 	&   $\Rep{C}_L$  &  1.16e-15 & 2.05e-15 & 5.70e-15 \\
 	&	$\Rep{C}_L \MP \Rep{A}_L \MP \Rep{C}_L$ & 3.06e-04 & 2.65e+02 & 3.27e+08   \\
	&	$\Rep{A}_L$  &  2.66e-04 & 2.08e+02 &  2.13e+08 \\  \hline
	$\Rep{V}_\mathrm{min}$	&  $\Rep{B}_L$	 & 1.89e-14 & 3.78e-14 & 2.96e-14  \\
	&   $\Rep{C}_L$  &  2.69e-15 & 1.70e-15 & 2.20e-15 \\
 	&	$\Rep{C}_L \MP \Rep{A}_L \MP \Rep{C}_L$ & 4.58e-04 & 3.60e+02 & 5.23e+08   \\
	&	$\Rep{A}_L$  &  4.99e-04 & 5.96e+02 &  4.27e+08 \\ \hline
	$\Rep{V}_\mathrm{max}$	&  $\Rep{B}_L$	 & 1.31e-14 & 1.20e-14 & 9.29e-15  \\
	&   $\Rep{C}_L$  &  9.82e-17 & 1.20e-16 & 1.07e-16 \\
 	&	$\Rep{C}_L \MP \Rep{A}_L \MP \Rep{C}_L$ & 1.08e-04 & 1.80e+02 & 1.26e+08   \\
	&	$\Rep{A}_L$  &  6.62e-03 & 1.43e+04 &  1.12e+10
	\end{tabular}
		\caption{Absolute errors ${\norm{\asm(\Rep{M} \MP \Rep{V}) - \asm(\lorth(\Rep{M} \MP \Rep{V}))}_2}$ with $\Rep{M}=\Rep{B}_L,\,\Rep{C}_L,\,\Rep{C}_L \MP \Rep{A}_L \MP \Rep{C}_L,\,\Rep{A}_L$ and $\Rep{V}=\Rep{V}_1,\Rep{V}_\mathrm{min},\Rep{V}_\mathrm{max}$, as given in Section \ref{sec:stabillustr}.}\label{tab:absstabtest}
	\end{table}

\begin{table}
\centering\small
	\begin{tabular}{ll|n{1}{2}|n{1}{2}|n{1}{2}}
 $\Rep{V}$	& $\Rep{M}$	 & \multicolumn{1}{c}{$L=20$}  & \multicolumn{1}{c}{$L=30$} & \multicolumn{1}{c}{$L=40$} \\ \hline
 $\Rep{V}_1$	&  $\Rep{B}_L$	 & 2.87e-15 & 4.06e-15 & 6.44e-15  \\
 	&   $\Rep{C}_L$  &  1.11e-15 & 1.95e-15 & 5.44e-15 \\
 	&	$\Rep{C}_L \MP \Rep{A}_L \MP \Rep{C}_L$ & 5.97e-05 & 1.89e+00 & 2.18e+00   \\
	&	$\Rep{A}_L$  &  2.48e-13 & 5.92e-12 &  1.85e-10 \\ \hline
	$\Rep{V}_\mathrm{min}$	&  $\Rep{B}_L$	 & 4.17e-15 & 8.32e-15 & 6.52e-15  \\
	&   $\Rep{C}_L$  &  2.40e-15 & 1.52e-15 & 1.97e-15 \\
 	&	$\Rep{C}_L \MP \Rep{A}_L \MP \Rep{C}_L$ & 1.01e-04 & 1.50e+00 & 5.41e+00   \\
	&	$\Rep{A}_L$  &  2.02e-04 & 7.22e-01 &  6.59e-01 \\ \hline
	$\Rep{V}_\mathrm{max}$	&  $\Rep{B}_L$	 & 3.28e-15 & 3.00e-15 & 2.32e-15  \\
	&   $\Rep{C}_L$  &  6.91e-11 & 8.61e-08 & 7.91e-05 \\
 	&	$\Rep{C}_L \MP \Rep{A}_L \MP \Rep{C}_L$ & 2.70e-05 & 6.32e+00 & 2.78e+00   \\
	&	$\Rep{A}_L$  &  1.51e-15 & 3.10e-15 &  2.31e-15
	\end{tabular}
	\caption{Relative errors ${\norm{\asm(\Rep{M} \MP \Rep{V}) - \asm(\lorth(\Rep{M} \MP \Rep{V}))}_2}/{\norm{\asm(\Rep{M} \MP \Rep{V})}_2}$ with $\Rep{M}$ and $\Rep{V}$ as in Table \ref{tab:absstabtest}.}\label{tab:relstabtest}
\end{table}

\section{Complexity of Solvers}\label{sec:solvers}

We now consider the numerical computation of $\Vec{u}_{L}$ solving $\Ten{B}_{L} \Vec{u}_{L} = \Vec{f}_{L}$
with $\Ten{B}_{L} = \Ten{C}_{L} \Ten{A}_{L} \Ten{C}_{L}$ and $\bg_L = \Ten{C}_{L} \Vec{f}_{L}$ as in \eqref{precondsystem}. Here the objective is to find $u_\varepsilon \in V_{L(\varepsilon)}$ such that $\norm{u - u_\varepsilon}_{\SSp{1}}\lesssim \varepsilon$, and we obtain an estimate for the computational complexity of achieving this goal. Assuming that $L(\varepsilon) \sim \abs{\log\varepsilon}$ is suitably chosen a priori and that the TT singular values of $\Vec{u}_{L}$ satisfy a natural decay estimate, we show that the number of arithmetic operations for computing a tensor train representation of $u_\varepsilon$ is of order $\cO(\abs{\log \varepsilon}^\theta)$, where $\theta>0$ depends only on the low-rank approximability of the $\Vec{u}_{L}$.

\begin{remark}
	The methods we consider rely on the accurate evaluation of residuals $\Ten{B}_L \Vec{v} - \Ten{C}_{L} \Vec{f}_{L}$. As we have seen in Section \ref{sec:stabillustr}, for the representations $\Rep{B}_L$ and $\Rep{C}_L$ of $\Ten{B}_L$ and $\Ten{C}_L$ that we have constructed, the quantities $\opramp_\ell(\Rep{B}_L)$ and $\opramp_\ell(\Rep{C}_L)$ grow only moderately with respect to $L$. Indeed, the results of Table \ref{tab:absstabtest} indicate that provided that $\Vec{v}$ and $\Vec{f}_{L}$ are given in well-conditioned representations, the corresponding residuals can be evaluated with an absolute error close to machine precision, which is corroborated also by our further numerical tests in Section \ref{sec:numexp}.
	For the convergence analysis of this section, we assume exact arithmetic.
\end{remark}

\subsection{Estimates of ranks and computational costs}

To estimate the computational complexity of finding approximate solutions, we use the quasi-optimality properties of an iterative method using soft thresholding of hierarchical tensors introduced in \cite{BS16}. This construction directly carries over to the special case of the TT format, leading to a soft thresholding operation $\cS_\alpha$ that is non-expansive with respect to the $\ell^2$-norm. It can be realized numerically for TT representations, described in \cite[Sec.\ 3]{BS16}, at essentially the same cost as the TT-SVD.

Note that since $\Ten{B}_{L}$ is well-conditioned uniformly with respect to $L$, as a consequence of Theorem \ref{thm:bpx} we can choose $\omega >0$ such that $\xi =\sup_{L>0}\norm{ I - \omega \Ten{B}_{L}}$ satisfies $\xi < 1$.
The basic iterative method applied to the present problem has the form
\begin{equation}\label{stiter1}
   \Vec{u}^{n+1}_L = \mathcal{S}_{\alpha_n} \bigl( \Vec{u}^n_L - \omega ( \Ten{B}_{L} \Vec{u}^n_L - \Vec{g}_L) \bigr), \quad n\geq 0,
\end{equation}
with $\Vec{u}^0_L=0$ and $\alpha_n \to 0$ determined (according to \cite[Alg.\ 2]{BS16}) as follows: set $\alpha_0 = \omega\norm{\Vec{g}_L}_2/(d-1)$, and for a fixed $\bar B > \norm{\Ten{B}_{L}}_{2\to2}$, take
\begin{equation}\label{stiter2}
 \alpha_{n+1} = \begin{cases}
 	  \frac12 \alpha_n, & \text{if } \norm{\bu^{n+1}_L - \bu^n_L}_2\leq \frac{1 - \xi }{\xi \bar B} \, \norm{\Ten{B}_{L} \bu^{n+1}_L - \bg_L}_2,\\
 	  \alpha_n, &\text{else.}
 \end{cases}
\end{equation}
In what follows, we refer to the algorithm given by \eqref{stiter1}, \eqref{stiter2} as \stsolve.

Recall that $u_L = \sum_{j\in\cJ_L} (\Ten{C}_{L} \Vec{u}_{L})_j \vphiD{L}{j}$, with analogous notation for the iterates, where $\norm{u_L}_V \sim \norm{\Vec{u}_L}_{2}$. Our convergence analysis is based on the following assumption on uniform decay of singular values, which is discussed further in Section \ref{sec:lrapprox}.

\begin{assumption}\label{svddecay}
For all $L\in\N$ and $\ell=1,\ldots,L$, let the singular values
$\sigma_{\ell,j}(\Vec{u}_L)$
with $j=1,\ldots,2^{D \QQ \max(\ell,L-\ell)}$
of the $\ell$th unfolding matrix $\Mat{\ell}(\Vec{u}_L)$,
defined as in~\eqref{Eq:DefUnfVec},
satisfy the bound
	\begin{equation}\label{unifsvs}
		\sigma_{\ell,j}(\Vec{u}_L)
		\leq
		C e^{- c j^\beta}
		\quad\text{for all}\quad
		j=1,\ldots,2^{D \QQ \max(\ell,L-\ell)}
	\end{equation}
	with $C,c,\beta >0$ independent of $\ell$ and $L$.
\end{assumption}

\begin{theorem}\label{thm:stconv}
	Let $\varepsilon >0$. Then \stsolve\ stops with $\bu_{L,\varepsilon}$ such that \[ \norm{u_L - u_{L,\varepsilon}}_{\SSp{1}}\lesssim \norm{\Vec{u}_{L} - \bu_{L,\varepsilon}}_{2} \leq \varepsilon \] after finitely many steps. In addition, let Assumption \ref{svddecay} hold.
	Then there exist $c_1,c_2>0$ and $\rho \in (0,1)$ independent of $L$ and $n$ such that with $\varepsilon_n = \rho^{n/\log L}$,
	\[
	  \norm{ u_L - u^n_L }_{\SSp{1}} \leq c_1 L \varepsilon_n , \quad
	    \max_{\ell=1,\ldots,L-1} \rank_\ell(\Vec{u}^n_L) \leq c_2 L^2 \bigl( 1 + \abs{\log \varepsilon_n} \bigr)^{\frac1\beta}.
	\]
\end{theorem}

\begin{proof}
	This is the statement of \cite[Thm.\ 5.1(ii)]{BS16} applied to our setting, combined with \cite[Rem.\ 5.6]{BS16} concerning the dependence of $\varepsilon_n$ on $L$.
\end{proof}

The above statement makes assumptions on the low-rank approximability of the approximations $u_L$. We next relate this, by an appropriate choice of $L$, to the approximability of the exact solution $u \in V$ of \eqref{varform}.

\begin{corollary}\label{cor:iterrankbounds}
 Assume that there exist $C_1>0$ and $s >0$ such that
 $
   \norm{u - u_L}_{\SSp{1}} \leq C_1 2^{-s L}
 $.
 Then for given $\varepsilon \in (0,1)$, taking $L= \frac1s ( 1 + \abs{\log \varepsilon})$, with $c_1,c_2>0$ and $\varepsilon_n = \rho^{n/\log L}$ as in Theorem \ref{thm:stconv}, 	for $n>0$ we have
 \[
 \begin{aligned}
	  \norm{ u_L - u^n_L }_{\SSp{1}} &\leq c_1 s^{-1} ( 1 + \abs{\log \varepsilon}) \varepsilon_n , \\
	    \max_{\ell=1,\ldots,L-1} \rank_\ell(\Vec{u}^n_L) &\leq c_2 s^{-2} ( 1 + \abs{\log \varepsilon})^2 \bigl( 1 + \abs{\log \varepsilon_n} \bigr)^{\frac1\beta},
 \end{aligned}
	\]
and for $N= (\abs{\log \varepsilon} + \log L) \log L \lesssim (1 + \abs{\log \varepsilon}) \log ( 1 + \abs{\log \varepsilon})$, we obtain
\[
	\norm{ u - u^N_L }_{\SSp{1}} \leq C_2 \varepsilon , \qquad
	  \max_{\ell=1,\ldots,L-1} \rank_\ell(\Vec{u}^N_L) \leq C_3 ( 1 + \abs{\log \varepsilon})^{2 + \frac1\beta},
\]
where $C_2,C_3>0$ depend on $c_1$, $c_2$, $\rho$, $C_1$, and $s$.
\end{corollary}

\begin{remark}[Complexity bounds]
If $\Ten{B}_{L}$ has fixed representation ranks, as in the case of the Laplacian, the costs of each step are dominated by those of applying $\cS_{\alpha_n}$, which are of order $\cO(L (\max_\ell \rank_\ell(\Vec{u}^n_L))^3)$. By Corollary \ref{cor:iterrankbounds}, the total number of operations for $N$ steps to guarantee an $H^1$-error of order $\varepsilon$ is thus bounded by
\begin{equation} \label{complexityestimate}
C (1 + \abs{\log \varepsilon})^{8 + \frac{3}{\beta}} \log( 1 + \abs{\log\varepsilon})
\end{equation}
 with a uniform constant $C>0$.
\end{remark}

In cases with variable coefficients such that $\Ten{B}_{L}$ does not have an exact low-rank form, but needs to be applied approximately, the iteration given in \eqref{stiter1} and \eqref{stiter2} can be adapted to residual approximations with prescribed tolerance as given in \cite[Alg.\ 3]{BS16}, which preserves the statement of Theorem \ref{thm:stconv} as shown in \cite[Prop.\ 5.9]{BS16}. Depending on the $L$- and $\varepsilon$-dependent rank bounds for $\Ten{B}_{L}$, one may then obtain additional factors in the estimate \eqref{complexityestimate}.

\begin{remark}
Complexity estimates are also given in \cite{BD:15} for a similar iterative method based on hierarchical SVD truncation (which in the present setting translates to a direct TT-SVD truncation).
A simplified version of this method operating on fixed discretizations is given in \cite[Alg.\ 4]{BS16}.
Based on the theory for this method, one can also derive rank and complexity bounds similar to \eqref{complexityestimate}, but with a less favorable exponent: For this method, one arrives at a number of operations bounded by $ C (1 + \abs{\log \varepsilon})^{t + \frac{3}{\beta}} $ for some $C>0$, where $t>0$ now depends on the representation ranks and condition number of $\Ten{B}_L$, and the bound can be substantially worse than \eqref{complexityestimate}.
 The practical performance of the scheme from \cite{BD:15}, however, tends to be comparable to the one of \stsolve\ considered above.
\end{remark}

\begin{remark}
Alternatively, the linear systems $\bB_L \bu_L = \bg_L$ can be solved by the \amen\ methods introduced in \cite{DS:14}. The basic version analyzed in \cite[Sec.\ 5]{DS:14} relies on residual approximations of a certain quality and increases approximation ranks in each iteration. However, the available convergence results only lead to a complexity bound that increases faster than exponentially in $L$. In the practical implementation that we also consider for comparison in Section \ref{sec:numexp}, the basic method is combined with a faster heuristic residual approximation scheme based on the alternating least squares (ALS) method and with additional rank reduction steps. Although no convergence analysis is available for this version, the method performs well in our tests with well-conditioned $\bB_L$.
\end{remark}

\subsection{Low-rank approximability assumptions}\label{sec:lrapprox}

	For the case of one or two dimensions,
	a low-rank approximation analysis
	for the solution of the problem~\eqref{varform}
	under certain analyticity assumptions
	on the coefficients and right-hand side,
	following
	from the regularity analysis developed
	in~\cite{Babuska:1988:RegularityI,Babuska:1988:hpCurvedBoundary},
	is available in \cite{KS:2015:PAMM,KazeevPhD,KazeevSchwab}.
	The following result can be obtained as an immediate consequence of \cite[Theorem~5.16]{KazeevSchwab}.

	\begin{theorem}\label{Th:QTT-FE-approx}
		Consider the problem~\eqref{varform} with $D=2$ dimensions
		under the ellipticity and regularity assumptions made in
		Section~\ref{Sc:DiscrPrec}.
		Assume additionally that the data (the diffusion coefficient
		and the right-hand side) are analytic on $\overline{\varOmega}$.
		Then the following holds with positive constants $C,C',b,b'$.
		For all $L,R\in\N$, the exact solution
		$u$
		admits an approximation
		$u_{L,R}\in V_L$ that can be exactly represented in the multilevel TT
		decomposition in the sense
		of~\eqref{Eq:MultLevTenDec-LinComb}--\eqref{Eq:MultLevTenDec-CoeffRepr},
		with ranks not exceeding $R$
		and such that
		\begin{equation}\label{Eq:QTT-FE-approx}
			\Norm{
				\QQ
				u - u_{L,R}
			}_{
				\SSp{1}\Par{\varOmega}
			}
			\leq
			C \QQ e^{- b \QQ L }
			+
			C' e^{- b' \sqrt{R} \, }
			\, .
		\end{equation}
	\end{theorem}

	Theorem~\ref{Th:QTT-FE-approx} and analogous
	results for highly oscillatory solutions~\cite{KORS:2017:Multiscale1d}
	cover the tensor approximation of
	exact solutions in the nodal basis, described in Section~\ref{Sc:BasisFunctionsD}.
	The requirements of Assumption \ref{svddecay} are somewhat different:
	they refer to the solution of the Galerkin discretization
	(uniformly in the discretization level $L$),
	and the application of $\Ten{C}_{L}^{-1}$ to the corresponding
	coefficient $\bar{\bu}_L$ (which is with respect to the nodal basis)
	yields the coefficient $\bu_L = \Ten{C}_{L}^{-1} \bar{\bu}_L$ with respect to the preconditioned basis.
	Nevertheless, the $\SSp{1}$-errors bounded
	implicitly by the decay of singular values
	in Assumption~\ref{svddecay} and explicitly by the second term
	in the right-hand side of~\eqref{Eq:QTT-FE-approx} both correspond to
	low-rank tensor approximation within the underlying finite element space $V_L$.

  The verification of the low-rank approximability of $\bu_L$, $L\in\N$,
  stipulated in Assumption \ref{svddecay}
  requires the result of Theorem~\ref{Th:QTT-FE-approx}
  to be complemented by two further ingredients:
  bounds on the ranks of Galerkin discretizations
  (as opposed to interpolants of the exact solution);
  and
  suitable low-rank approximations of $\Ten{C}_{L}^{-1}$,
  (which, unlike $\Ten{C}_L$, does not have an explicit low-rank form).

	In the present work, we restrict ourselves to studying the resulting approximability of $\bu_L$ numerically.
    We are not aware of existing analysis that would allow to arrive at conclusions on Galerkin solution ranks, covering also the convergence behavior for accuracies below the size of the Galerkin discretization error; this appears to be a question of independent interest.
	In certain special cases, such as Poisson problems in $D=1$, the Galerkin solution can in fact be shown to be the nodal interpolant of the exact solution. For more general problems and for $D>1$, however, this is in general not the case.

	The numerically observed decay of matricization singular values of the preconditioned solution coefficients $\bu_L$ (with $\norm{\bu_L}_2 \sim \norm{u_L}_{\SSp{1}}$) and of the vector of scaled nodal values $\bar{\bu}_L = \bC_L \bu_L$ (with $\norm{\bar{\bu}_L}_2 \sim \norm{u_L}_{\LSp{2}}$) for a Poisson problem in spatial dimension $D=2$ is illustrated in Figure \ref{fig:solranks}. We find that the action of $\bC_L^{-1}$ on the vector of nodal values preserves the exponential decay of singular values, but at a slightly modified rate. This is consistent with the further numerical tests for this problem in Section \ref{sec:test2d}.
	Similar results are also observed in further experiments presented in Section \ref{sec:numexp}.

\begin{figure}\centering
\begin{tabular}{cc}\hspace{-6pt}
	\includegraphics[width=6cm]{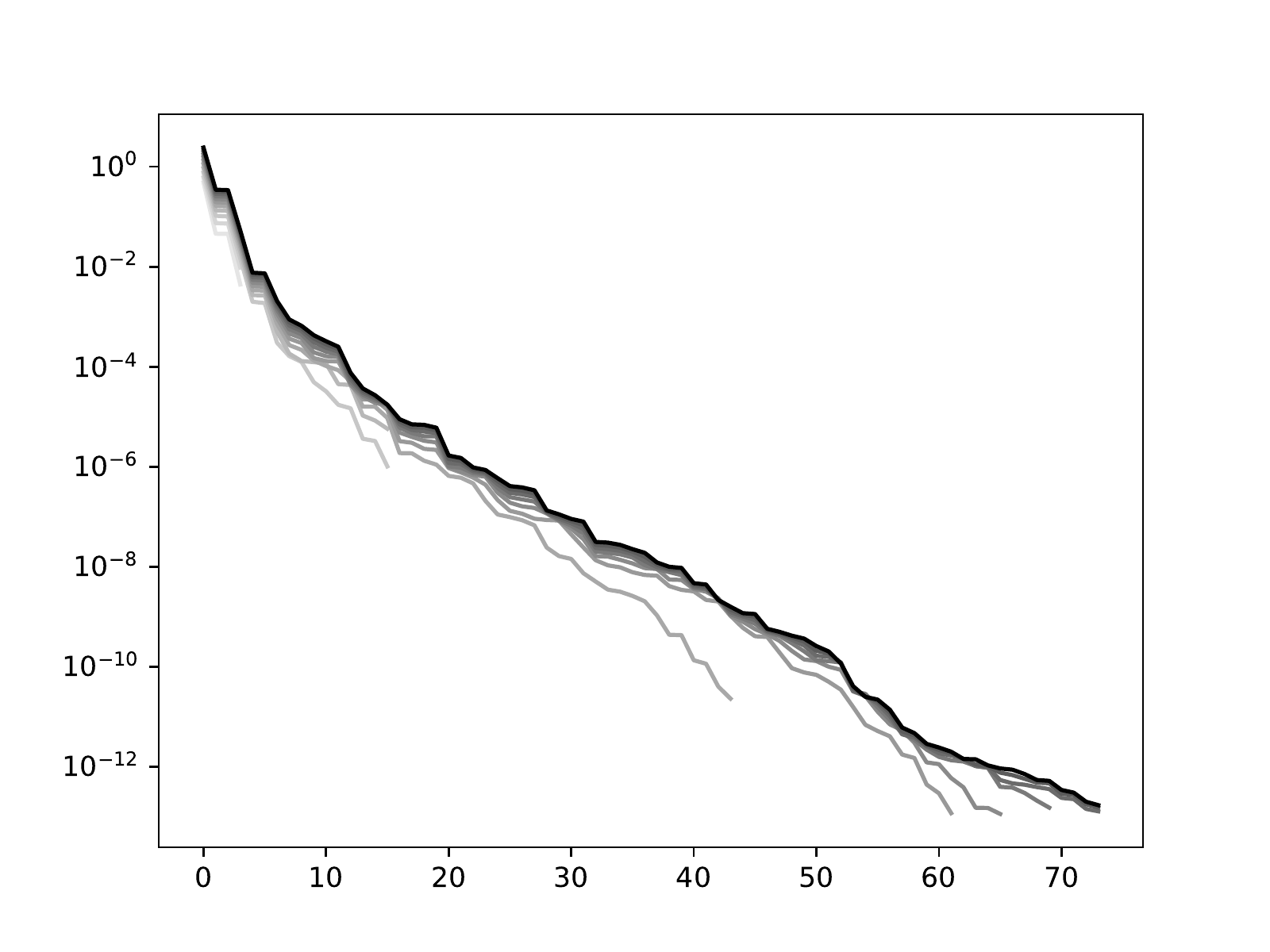} &
	\hspace{-6pt}\includegraphics[width=6cm]{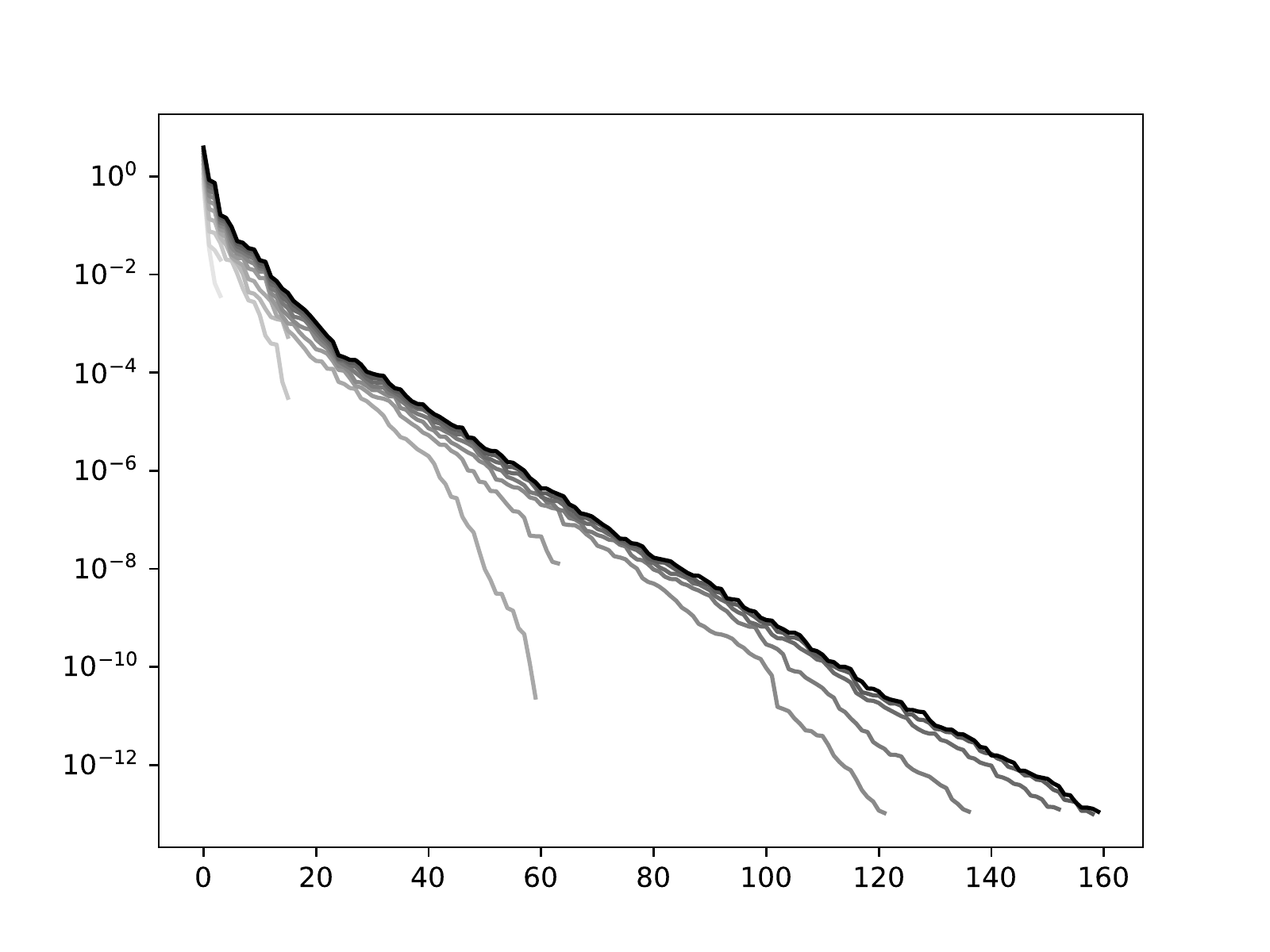}	\\
	\small(a) $\displaystyle \max_{\ell=1,\ldots,L-1} \sigma_{\ell,j}(\Ten{C}_L \bu_L)$~vs.~$j$
	&
	\small(b) $\displaystyle \max_{\ell=1,\ldots,L-1} \sigma_{\ell,j}(\bu_L)$~vs.~$j$
\end{tabular}
\caption{Singular values of unfolding matrices
(see Assumption~\ref{svddecay})
for $u$ solving $-\Delta u=1$ on $(0,1)^2$ with boundary conditions according
to \eqref{Eq:Space}, for $L=2,\ldots,12$.
}\label{fig:solranks}
\end{figure}

\section{Numerical Experiments}\label{sec:numexp}

In our numerical tests, we apply the preconditioned discretization matrices in well-conditioned tensor representations obtained in Section \ref{sec:tensorstructure} to different problems of the type \eqref{varform}, both with constant and with highly oscillatory diffusion coefficients $A$ in \eqref{Eq:DefBLF}.

For solving the resulting systems of equations, on the one hand we use \stsolve\ analyzed in Section \ref{sec:solvers}, implemented in the Julia programming language;
on the other hand, we compare to results obtained using a Fortran implementation of the \amen\ solver \cite{DS:14} wrapped by the Python version of the TT Toolbox by I.\ Oseledets.

These two solvers have quite distinct characteristics. The parameters for \stsolve\ are chosen such that the convergence and complexity estimates of Theorem \ref{thm:stconv} are guaranteed, which leads to a very conservative control of the iteration. Since residuals are approximated with guaranteed accuracy, this method yields rigorous error bounds. In contrast, the considered version of \amen\ uses several heuristic extensions, as described in \cite[Sec.\ 6]{DS:14}. In particular, it uses a simplified ALS-type residual approximation that has strongly reduced complexity, but does not give any error guarantees.

Moreover, in the given results, iteration numbers for \amen\ need to be interpreted differently, where each iteration in the convergence plots comprises several substeps with local residual evaluations for each core.

\subsection{Results without preconditioning}\label{sec:noprecond}

We first illustrate the results obtained by a direct application of multilevel tensor representations of stiffness matrices $\bA_L$ without preconditioning. Such representations have been derived, for instance, in \cite{KKh:2012:QTT}. In the present case of mixed Dirichlet and Neumann boundary conditions, this leads to representations similar to the pure Dirichlet case in \eqref{laplacedd}. Here we consider the case $D=1$, where for simplicity we take reaction coefficient $c = 0$ and right-hand side $f=1$, that is, we solve the weak formulation of
\begin{equation}\label{test1d}
 - u'' = 1, \quad u(0)=0, \; u'(1) = 0.
\end{equation}

Using \amen\ directly with system matrix $\bA_L$ and right-hand side $\bbf_L$, we observe that the resulting residual indicators stagnate at values above $2^{2L} \epsilon$, where $\epsilon\approx 2.2 \times 10^{-16}$ is the relative machine precision. This is to be expected in view of the matrix and representation ill-conditioning of $\bA_L$.

If we instead implement the preconditioned matrix $\bC_L\bA_L \bC_L$ by pre- and post-multiplying with a separate tensor representation $\bC_L$ of the preconditioner, we still obtain essentially the same type of stagnation at approximately $2^{2L} \epsilon$. Since the represented matrix $\bC_L\bA_L \bC_L$ is now well-conditioned, these remaining catastrophic round-off errors and the resulting stagnation are entirely due to \emph{representation} ill-conditioning, which is not removed by simply multiplying by the preconditioner. This effect is observed both with \amen\ and with \stsolve. The results are shown in Figure \ref{fig:noprec}, with the residual values with respect to the system matrices $\bA_L$ and $\bC_L\bA_L \bC_L$, respectively.

\begin{figure}\centering
\begin{tabular}{ccc}
	\includegraphics[width=5cm]{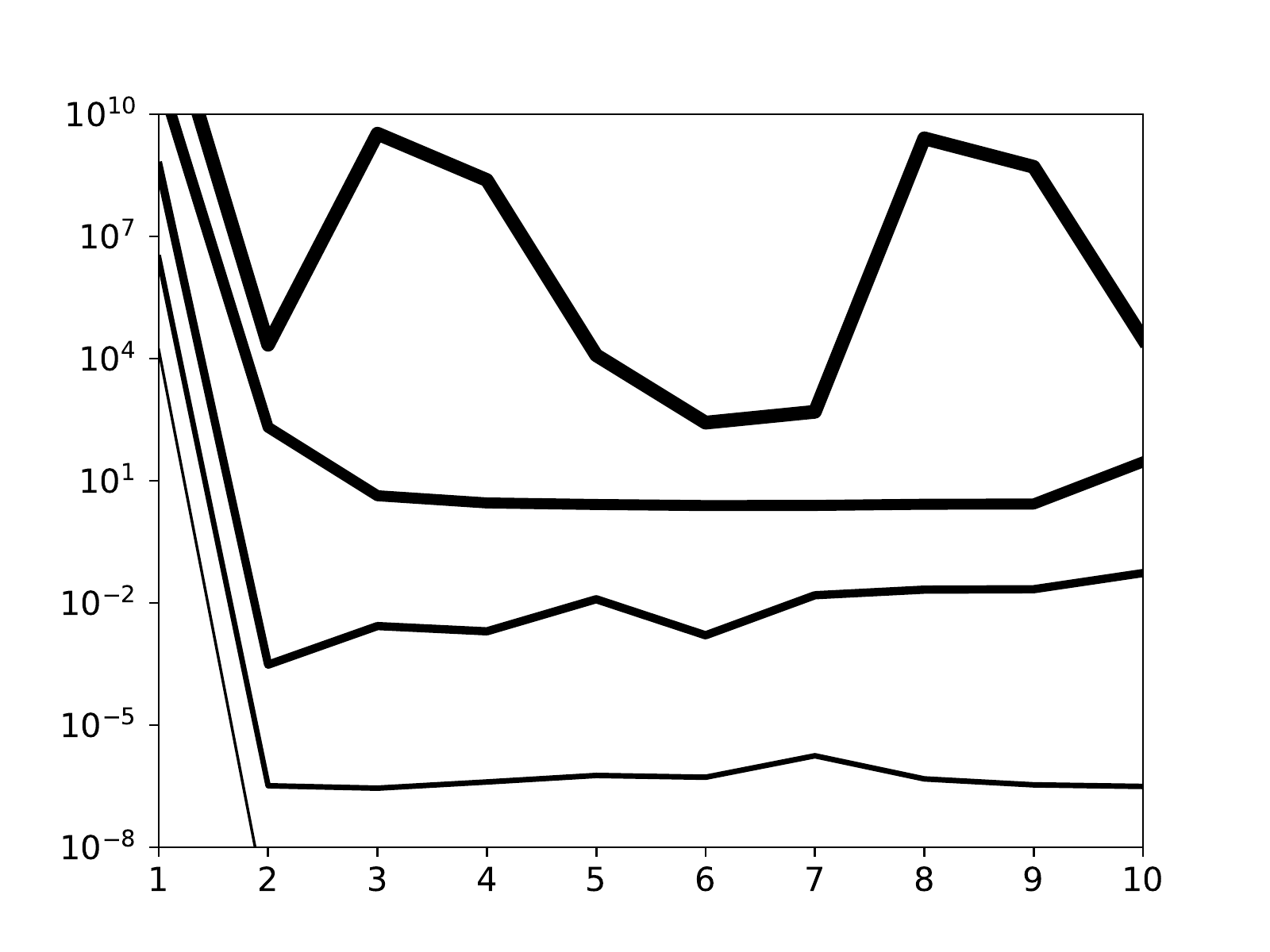} &
	\includegraphics[width=5cm]{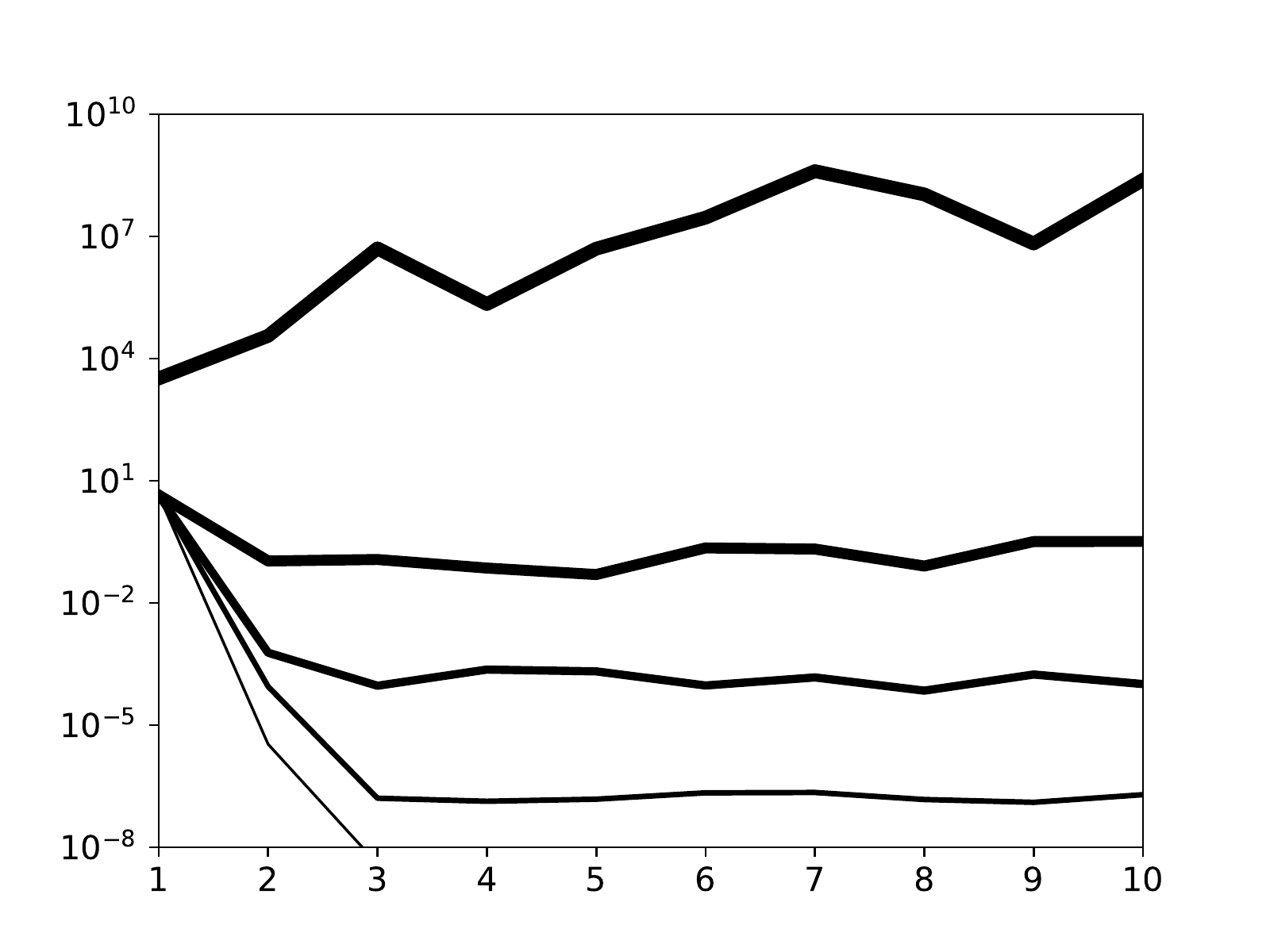} &
	\includegraphics[width=5cm]{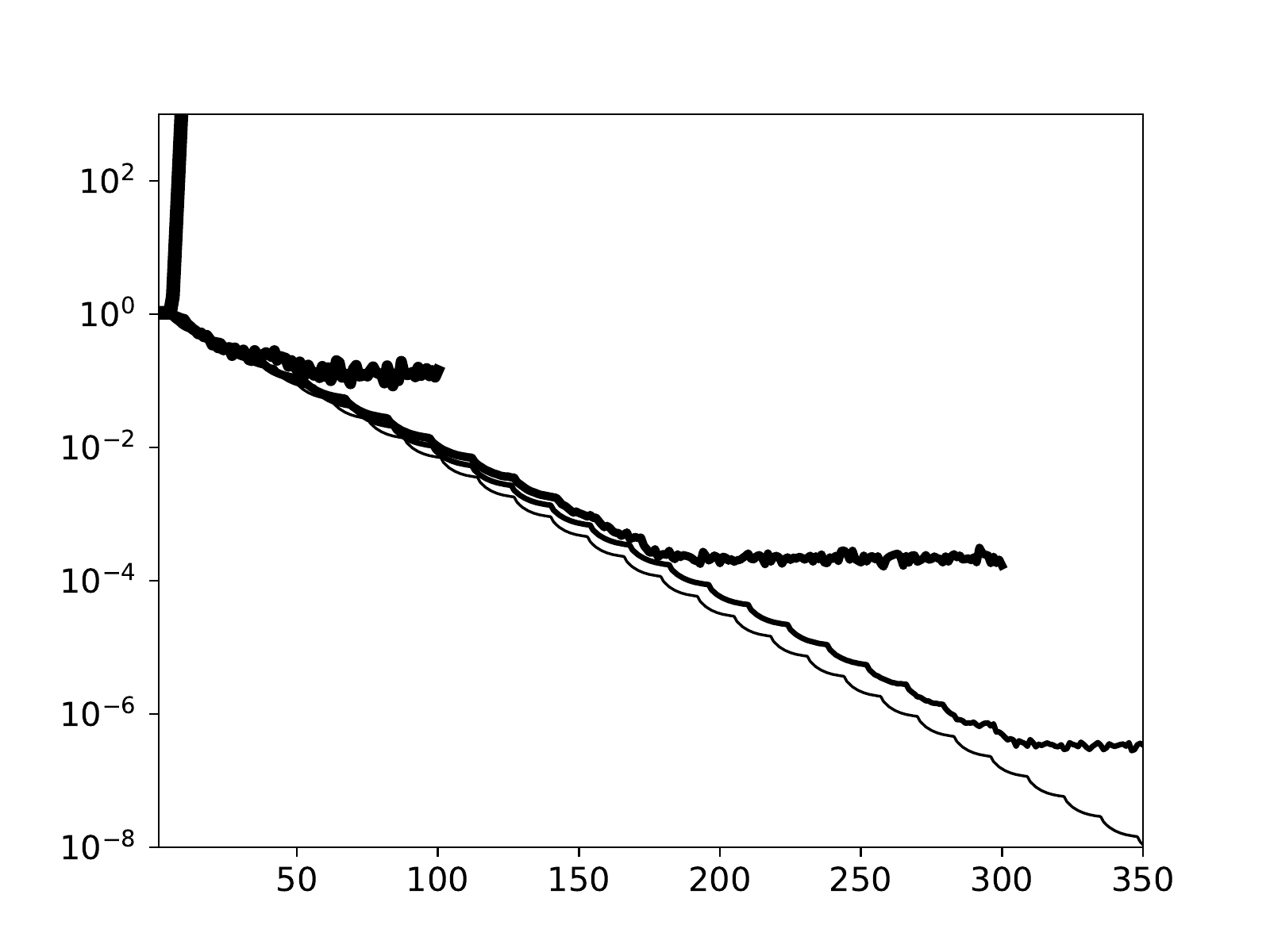} \\
	\footnotesize(a) \amen\ with $\bA_L$ & \footnotesize(b) \amen\ with $\bC_L\bA_L \bC_L$ & \footnotesize(c) \stsolve\ with $\bC_L\bA_L \bC_L$
\end{tabular}
\caption{Results for Section \ref{sec:noprecond}, computed residual bounds in dependence on iteration count: (a) \amen\ applied directly to $\bA_L$, (b) \amen\ with directly multiplied $\bC_L\bA_L \bC_L$, (c) \stsolve\ with directly multiplied $\bC_L\bA_L \bC_L$; each for $L=10,15,20,25,30$ (by increasing line thickness).}\label{fig:noprec}
\end{figure}

\subsection{Constant-coefficient diffusion, $D=1$}\label{sec:test1d}

We now consider the same basic test case \eqref{test1d}, but with $\bB_L =\bC_L\bA_L \bC_L$ in the combined tensor representation constructed in Section \ref{sec:tensorstructure}. In this and the following tests, residual values always refer to the preconditioned residuals $\norm{ \bB_L \cdot - \bg_L }_2$, which is proportional to the $\SSp{1}$-errors in the corresponding grid functions. With a target residual of $10^{-12}$, both \amen\ and \stsolve\ converge unaffected by any round-off errors for very large values of $L$. Indeed, this remains true for values $L$ that are substantially larger than in the case $L=50$ shown here, but since the corresponding mesh widths are then smaller than machine precision, the results are  more difficult to interpret.

For the \amen\ solver, we assemble the complete representation of $\bB_L$. In exact arithmetic, this would in fact be equivalent to applying representations $\Rep{A}_L$ and $\Rep{C}_L$ separately, and differences are entirely due to the different tensor decomposition in the previous case. With \stsolve, we have the additional option of using error-controlled inexact residual evaluations as in \cite[Alg.\ 3]{BS16} to reduce the arising  ranks of intermediate results; as shown in \cite[Prop.\ 5.9]{BS16}, the statement of Theorem \ref{thm:stconv} still applies to this modification. To this end, we use that the tensor representation can be directly rewritten in the form $\bB_L = \Ten{\varTheta}_{L\CQ 1}^\MT \: \Ten{\varTheta}_{L\CQ 1}$ as in \eqref{Eq:DecBTheta}, where $\norm{\Ten{\varTheta}_{L\CQ 1}}$ is uniformly bounded with respect to $L$, and apply an additional recompression by TT-SVD after applying $\Ten{\varTheta}_{L\CQ 1}$.

\begin{figure}\centering
\begin{tabular}{cc}
	\includegraphics[width=6cm]{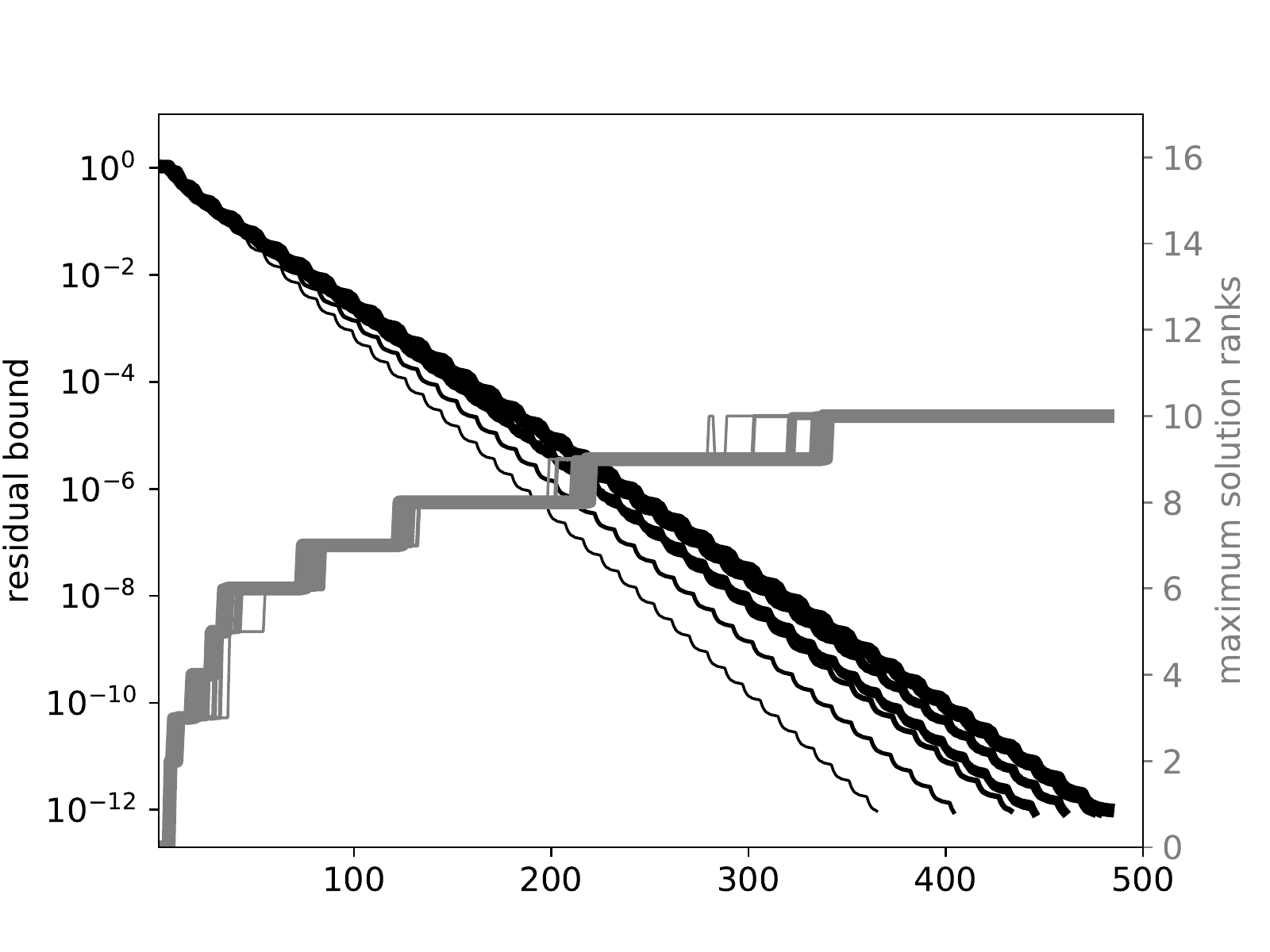} &
	\includegraphics[width=6cm]{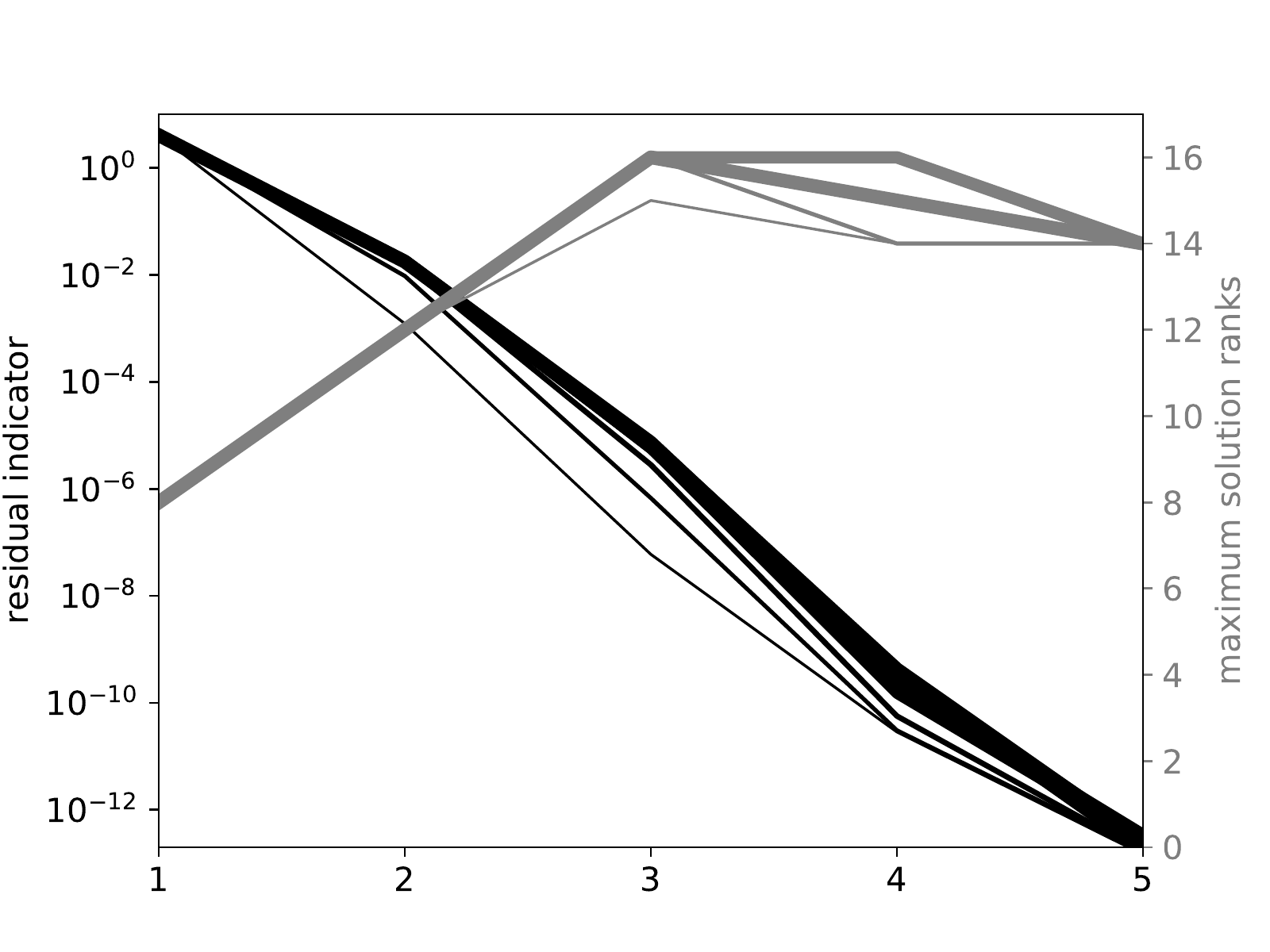} \\
	\footnotesize(a) \stsolve & \footnotesize(b) \amen
\end{tabular}
\caption{Results for Section \ref{sec:test1d}: residual bounds (black) and maximum approximation ranks (grey), with well-conditioned combined representation of $\bB_L = \bC_L\bA_L \bC_L$ for $L=10,15,20,25,30,35,40,45,50$ (by increasing line thickness).}\label{fig:lap1}
\end{figure}

\subsection{Highly oscillatory diffusion coefficients, $D=1$}\label{sec:testosc}

We next consider the family of problems with oscillatory diffusion coefficients on $\varOmega=(0,1)$ given by
\begin{equation}\label{testosc}
 - ( a_K u')' = 1, \quad u(0)=0, \; u'(1) = 0, \quad  a_K(x) = \bigl( 2 + \cos (K\pi x)\bigr)^{-1}
\end{equation}
for large values of $K$. The exact solution reads
\begin{equation}\label{testoscsol}
 u(x) = x(2-x) + (K\pi)^{-1} \bigl[ (1-x) \sin(K\pi x) + (K\pi)^{-1}\bigl(1 - \cos(K\pi x)\bigr)\bigr]
 \, .
\end{equation}
For $K\in 4\N$, we represent the vectors $\Vec{u}_\mathrm{ex}$ and $\Vec{v}_\mathrm{ex}$ of nodal values of $u$ and $u'$
in the multiscale TT format
with ranks bounded by seven and six, respectively.

The coefficient $a_K$ does not have an explicit low-rank form, and we compute approximations as follows: using the explicit rank-three representation of $c(x) = 2 + \cos (K\pi x)$, using \stsolve\ we solve the equation
$c(x_i) \,a_K(x_i) = 1$ in the points $x_i = 2^{-L}(i-\frac12)$, $i=1,\ldots,2^L$, as an elliptic problem on $\ell^2(\{1,\ldots,2^L\})$ for $a_K$; the tolerance is chosen to ensure a sufficient uniform error bound.

We compare the results for the values $K=2^{10}, 2^{20}, 2^{30}, 2^{40}$ with $L = 50$ in Figure \ref{fig:osc}.  The observed convergence patterns of both methods show hardly any influence of the value of $K$. Note that the computed preconditioned coefficients $\bu_L$ do not satisfy the same rank bound as \eqref{testoscsol} (which holds for $\bC_L \bu_L$, the corresponding vector of scaled nodal values).
In each case, comparison with the explicit low-rank form of $\Vec{u}_\mathrm{ex}$, $\Vec{v}_\mathrm{ex}$ shows that the expected total error bounds are achieved.

More specifically, approximations of the $\SSp{1}$-error in the solutions can be obtained in a numerically stable way by evaluating $\norm{\Vec{u}_\mathrm{ex} - \bC_L \bu_L}_2$ and $\norm{\Vec{v}_\mathrm{ex} -\Ten{\varTheta}_{L\CQ 1} \bu_L}_2$, where $\Ten{\varTheta}_{L\CQ 1}$ is the factor of the preconditioned Laplacian stiffness matrix as in Section \ref{sec:test1d}.
In Table \ref{tab:osc}, we summarize the obtained approximations of $\SSp{1}$-errors for different solver tolerances and parameters $L$. We observe an effect that is particular to the present one-dimensional setting, where the accuracy in the nodal values is limited only by the solver tolerance as soon as $L$ is sufficiently large for resolving the oscillations in the solution.

\begin{figure}[ht]\centering
\begin{tabular}{cc}
	\includegraphics[width=6cm]{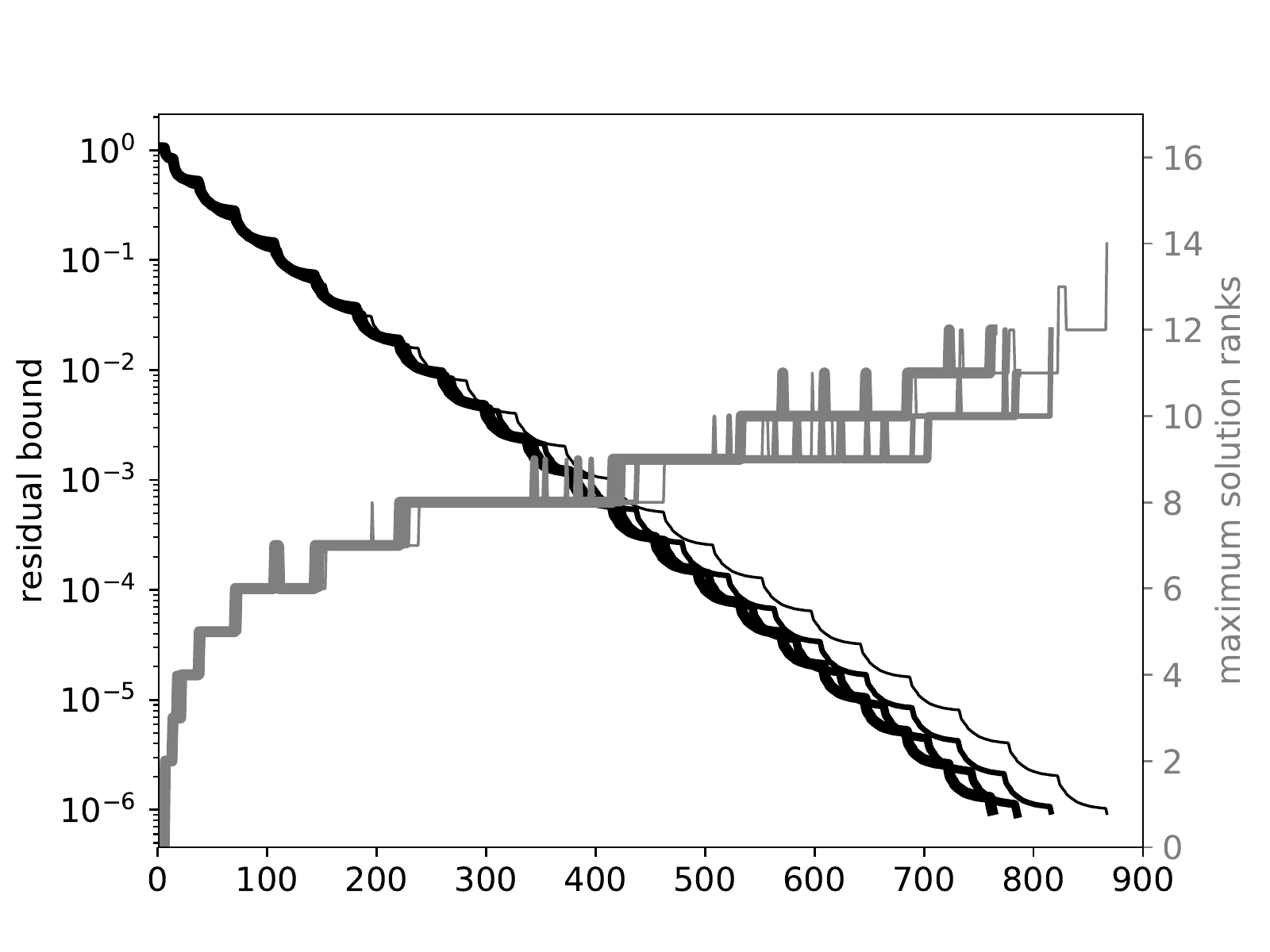} &
	\includegraphics[width=6cm]{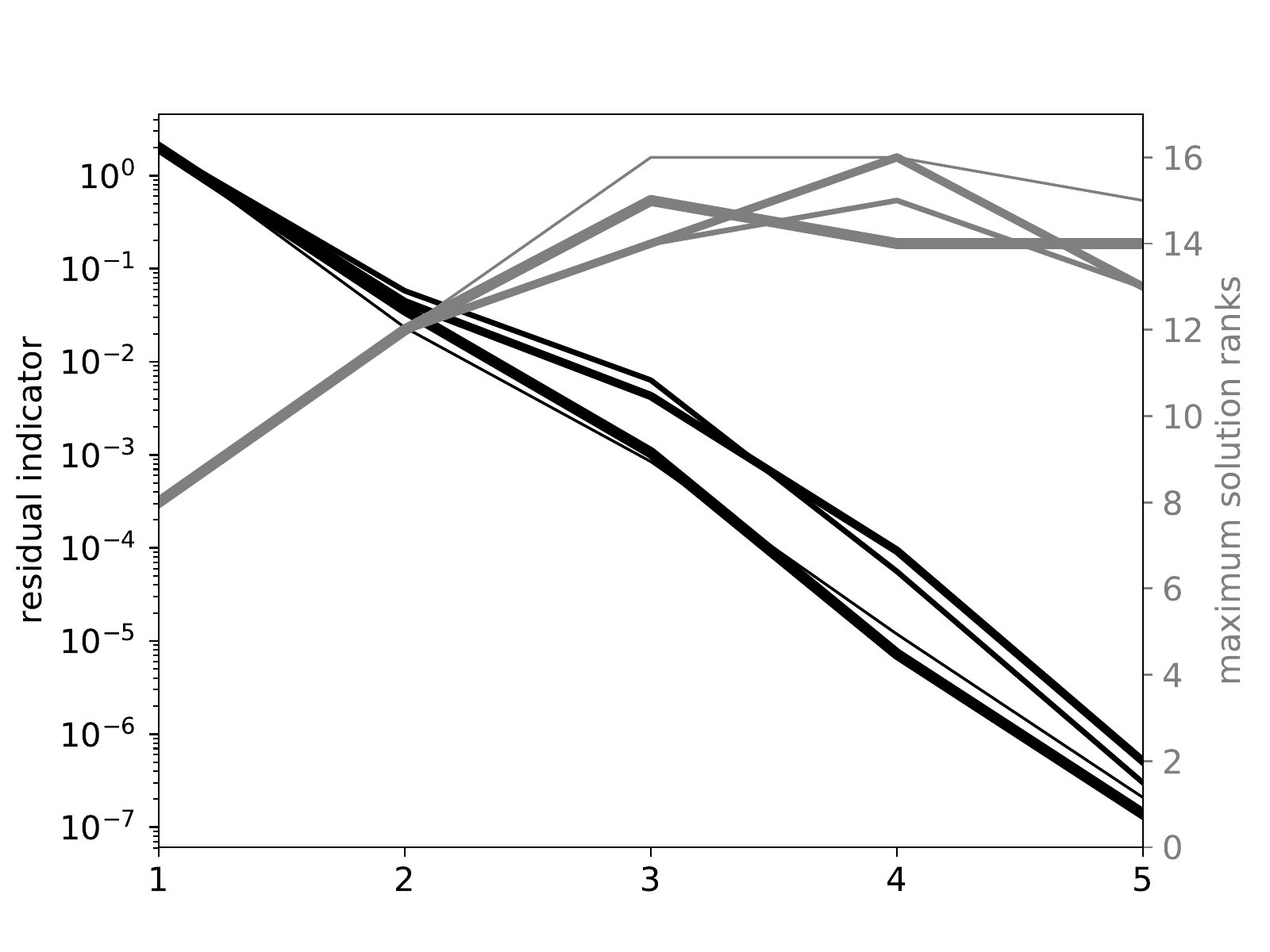} \\
	\footnotesize(a) \stsolve & \footnotesize(b) \amen
\end{tabular}
\caption{Results for Section \ref{sec:testosc}: residual bounds (black) and maximum approximation ranks (grey), with well-conditioned representation of $\bB_L$ for oscillatory coefficient $a_K$ with $K= 2^{10},2^{20},2^{30},2^{40}$ (by increasing line thickness) and $L=50$.}\label{fig:osc}
\end{figure}

\begin{table}[ht]\centering\footnotesize
	\begin{tabular}{c|n{1}{2}|n{1}{2}|n{1}{2}|n{1}{2}}
		tol.\ & \multicolumn{1}{c}{$L= 10$} &  \multicolumn{1}{c}{$L=20$} & \multicolumn{1}{c}{$L=30$} & \multicolumn{1}{c}{$L=40$} \\
$10^{-4}$ & 3.65e-01 & 3.65e-01 & 3.21e-05 & 3.45e-05 \\
$10^{-6}$ & 3.65e-01 & 3.65e-01 & 2.89e-07 & 2.88e-07 \\
$10^{-8}$ & 3.65e-01 & 3.65e-01 & 3.73e-08 & 2.71e-08
	\end{tabular}
	\caption{$\SSp{1}$-errors in approximations computed by {\amen} with $K=2^{30}$, solver tolerances $10^{-4}$, $10^{-6}$, $10^{-8}$, and discretization parameters $L$.}\label{tab:osc}
\end{table}

\subsection{Constant-coefficient diffusion, $D=2$}\label{sec:test2d}

On $\varOmega = (0,1)^2$, we consider \eqref{varform} with $A=1$, $c=0$ and $f=1$, that is, the weak form of
\begin{equation}\label{test2d}
 -\Delta u = 1, \quad u|_\varGamma = 0, \;\; \partial_n u|_{\partial\varOmega\setminus\varGamma} = 0,
\end{equation}
with $\varGamma$ as in \eqref{Eq:Space}. Both \stsolve\ and \amen\ show the expected convergence for $L=50$, with ranks that are consistent with the singular value decay of discretized solutions of Figure \ref{fig:solranks}(b).

Similarly to Section \ref{sec:test1d}, \stsolve\ is used with inexact residual evaluation, now using that the tensor representation of $\bB_L$ can be written in the form $\bB_L = \Ten{\varTheta}_{L\CQ 1}^\MT \; \Ten{\varTheta}_{L\CQ 1} + \Ten{\varTheta}_{L\CQ 2}^\MT \; \Ten{\varTheta}_{L\CQ 2}$ as in \eqref{Eq:DecBTheta}. Here $\Ten{\varTheta}_{L\CQ 1}$ and $\Ten{\varTheta}_{L\CQ 2}$ are uniformly bounded, and each has maximum representation rank $24$.
Although these ranks remain independent of $L$,
additional rank reductions in this decomposition are important
from a quantitative point of view:
since $\bB_L$ has maximum representation rank $1152$,
applying it directly would lead to very large ranks.
In the available version of \amen, the decomposition of $\bB_L$ needs to be
used directly,
but the impact of large residual ranks is limited due to the ALS-type
residual approximation.
In this case, the main downside of the direct assembly of $\bB_L$ is
in the higher memory requirements for large $L$.

\begin{figure}\centering
\begin{tabular}{cc}
	\includegraphics[width=6cm]{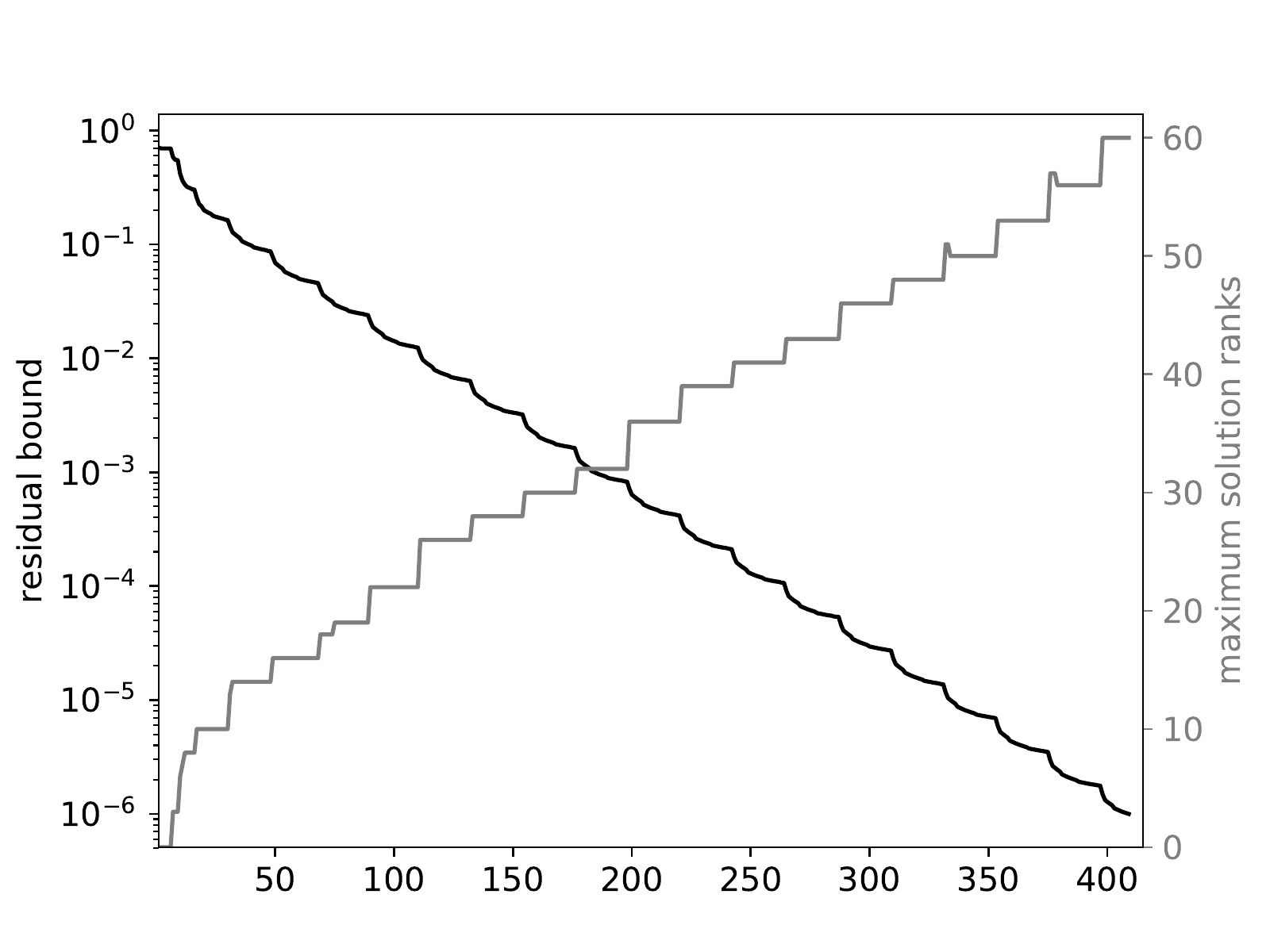} &
	\includegraphics[width=6cm]{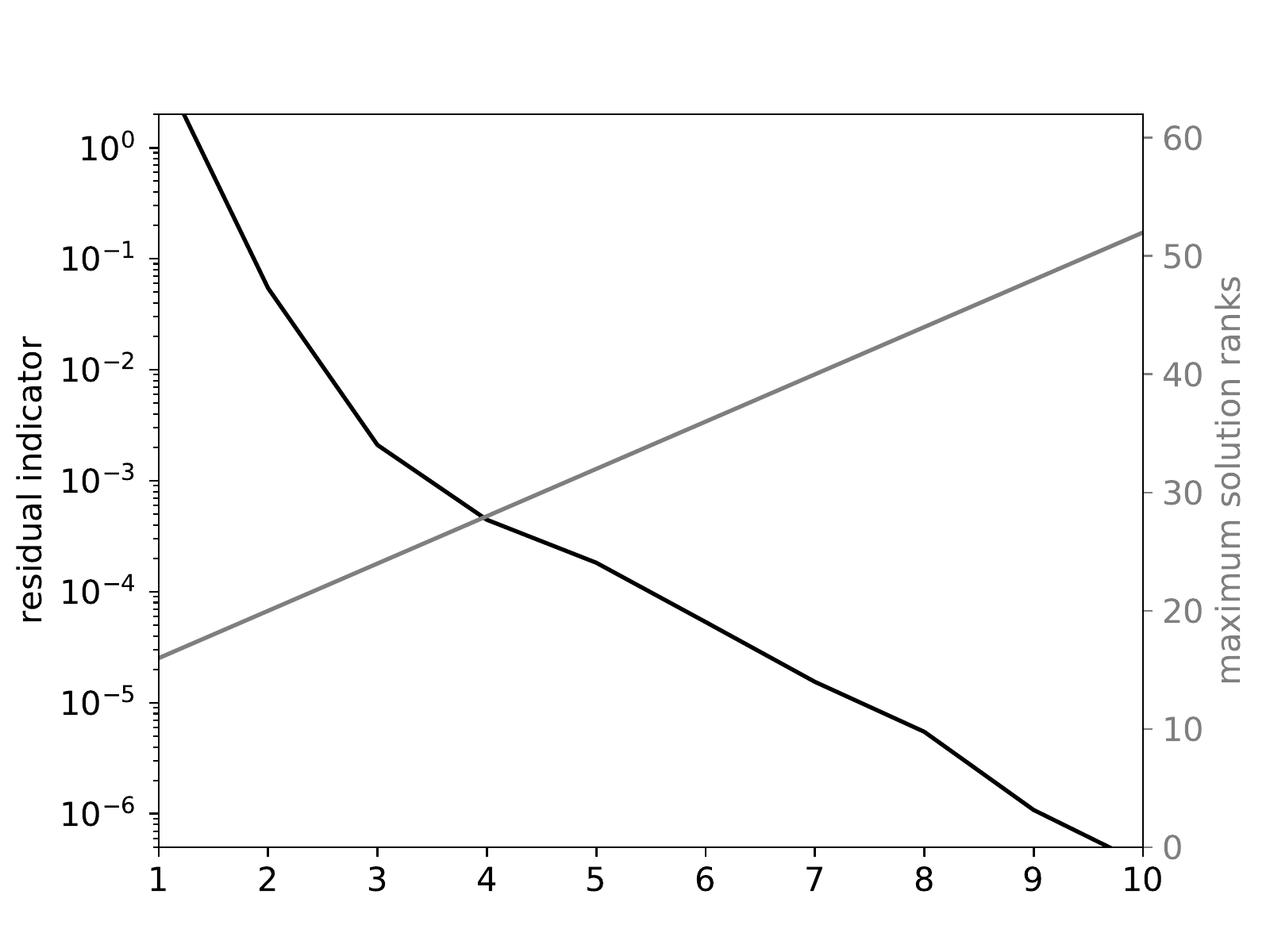} \\
	\footnotesize(a) \stsolve & \footnotesize(b) \amen
\end{tabular}
\caption{Results for Section \ref{sec:test2d}: residual bounds (black) and maximum approximation ranks (grey), for well-conditioned representation of $\bB_L$,  $L=50$.}\label{fig:lap2}
\end{figure}

In terms of computational costs, the error-controlled full residual approximation used by \stsolve\ is substantially more expensive in all considered tests than the heuristic ALS-based residual approximation used by \amen. The precise CPU timings are of limited significance due to the different implementations, but we observe running times on the order of several minutes with \stsolve\ and of seconds with \amen\ in the tests with $D=1$, and of several hours with \stsolve\ and several minutes with \amen\ in the case of $D=2$.
Although no convergence analysis is available for this \amen\ implementation, especially for the present well-conditioned representations it is thus an interesting practical choice.

\section{Conclusion and Outlook}

We have identified notions of condition numbers of tensor representations that determine the propagation of errors in numerical algorithms. In the application to multilevel tensor-structured discretizations of second-order elliptic PDEs, the careful construction of tensor representations of preconditioned system matrices guided by these notions leads to solvers that remain numerically stable also for very large discretization levels. For one such method based on soft thresholding of tensors, we have shown that the total number of arithmetic operations scales like a fixed power of the logarithm of the prescribed bound on the total solution error.

The new variant of BPX preconditioning that we have analyzed leads to a very natural low-rank structure of the symmetrically preconditioned stiffness matrix. Remarkably, unlike the rank increase with discretization levels observed in the case of separation of \emph{spatial coordinates} \cite{MR3474850}, in the present case of tensor separation of \emph{scales}, we obtain preconditioner representation ranks that remain uniformly bounded with respect to the discretization level. Similar results can be obtained for related preconditioners based on wavelet transforms, which are the subject of ongoing work.

For the preconditioned solvers, the relevant approximability properties of solutions we have identified are slightly different from the ones for nodal basis coefficients studied, e.g., in \cite{KazeevSchwab}. The numerically observed favorable decay of TT singular values of preconditioned quantities thus requires further investigation; it also depends on the particular choice of preconditioner.

The practical application to more general problems was not considered here to avoid further technicalities, but one can similarly treat different boundary conditions, more general coefficients (such as highly oscillatory diffusion coefficients in $D>1$) or more general domains by techniques developed in \cite{KazeevPhD}.
We also expect that our basic considerations concerning the combined low-rank representations of preconditioners and discretization matrices of differential operators can be applied, with potentially more technical effort, to other types of basis expansions and to different classes of PDE problems.

Although the representation ranks of preconditioned matrices that we obtain are bounded independently of the discretization level, they are fairly large for $D>1$. This suggests the further investigation of solvers with improved quantitative performance, in particular the combination of \amen-type methods with efficient residual approximation strategies for preconditioned operator representations.

We expect that the framework we have proposed here for studying the conditioning of tensor representations can be developed further to provide more detailed information, as well as sharper computable bounds for representations of matrices.

\bibliographystyle{amsplain}
\bibliography{tensorbpx}

\begin{appendix}

\section{Preconditioner Optimality}\label{App:PO}

In preparation of the proof of Theorem \ref{Th:Prec}, we define the square matrix $\Ten{D}_L$ of order $2^{DL}$ by
\begin{equation}\label{Eq:DL}
 \Par{\Ten{D}_L}_{j \; j'} = \IProd[1]{ \nabla\vphiD{L}{j} }{ \nabla \vphiD{L}{j'} }_{\LSp{2}{(\varOmega)}}
 \quad\text{for all}\quad
 j,j'\in \cJ_L.
\end{equation}
Since the bilinear form $a$ is elliptic on $V$ with $\norm{\cdot}_V = \norm{\nabla \cdot \,}_{\LSp{2}\Par{\varOmega}^d}$, we obtain
\[
\begin{aligned}
 \langle \Ten{C}_L \bA_L \Ten{C}_L \bv, \bv\rangle
   & = a\Bigl(\sum_{\ell = 0}^L 2^{-\ell} \sum_{j\in \cJ_\ell} \vphiD{\ell}{j} (\Ten{ P}_{\ell,L}^\MT \Ten{v})_j, \sum_{\ell = 0}^L 2^{-\ell} \sum_{j\in \cJ_\ell} \vphiD{\ell}{j} (\Ten{ P}_{\ell,L}^\MT \Ten{v})_j \Bigr)   \\
  & \sim \Bignorm{\nabla \Bigl( \sum_{\ell = 0}^L 2^{-\ell} \sum_{j\in \cJ_\ell} \vphiD{\ell}{j} (\Ten{ P}_{\ell,L}^\MT \Ten{v})_j \Bigr) }^2_{\LSp{2}}
  = \langle \Ten{C}_L \Ten{D}_L \Ten{C}_L \bv, \bv\rangle,
\end{aligned}
 \]
and it thus suffices to show \eqref{bpxtobeshown} for $\Ten{D}_L$ in place of $\Ten{A}_L$.

For $\ell=0,\ldots,L$, we introduce the nested subspaces
$
  \cV_\ell = \ran \Ten{P}_{\ell,L} \subseteq \R^{\cJ_L}
 $;
that is, the spaces $\cV_\ell$ are spanned by vectors of finest-grid nodal values of the functions $\vphiD{\ell}{j}$, $j\in \cJ_\ell$. In particular, $\cV_L = \R^{\cJ_L}$.

\begin{lemma}\label{Lm:strengthenedCS}
For $\ell, \ell' \in \{0,\ldots,L\}$, let
\begin{equation}\label{Def:Lll}
	{\Ten{L}}_{\ell, \ell'}
	=
	2^{-\ell-\ell'}
	\,
	\Ten{ P}_{\ell,L}^\MT
	\:
	\Ten{D}_{L}
	\,
	\Ten{ P}_{\ell',L}.
\end{equation}
Then for $0\leq k\leq \ell$, $0\leq k'\leq \ell'$,
\begin{equation}\label{combinedoffdiagest}
	\Abs[1]{
		\IProd[1]{
			{\Ten{L}}_{\ell,\ell'} \Ten{ P}_{\ell',L}^\MT\bw_{k'}
		}{
			\Ten{ P}_{\ell,L}^\MT \bw_k
		}
	}
	\lesssim
	2^{-\frac12 \abs{\ell'-\ell}}
	\;
	2^{\frac12 (k'-\ell')}
	\,
	\norm{\bw_{k'}}_2
	\;
	2^{\frac12 (k-\ell)}
	\,
	\norm{\bw_{k}}_2
\end{equation}
for all $\bw_k \in \cV_k$ and $\bw_{k'} \in \cV_{k'}$.
\end{lemma}

\begin{proof}
	The matrices defined in \ref{Def:Lll} can also be expressed in terms of
\begin{equation}\label{eq:LEexpr}
		\hat{\Ten{L}}_{\ell,\ell'}
		=
		\bigl( 2^{-\ell-\ell'} \langle\vphi{\ell}{j}' ,\vphi{\ell'}{j'}'\rangle_{\LSp{2}(0,1)} \bigr)_{j \in \hat{\cJ}_\ell, j'\in\hat{\cJ}_{\ell'}}
		\, ,
		\qquad
		\hat{\Ten{E}}_{\ell,\ell'}
		=
		\bigl( \langle\vphi{\ell}{j} ,\vphi{\ell'}{j'}\rangle_{\LSp{2}(0,1)} \bigr)_{j \in \hat{\cJ}_\ell, j'\in\hat{\cJ}_{\ell'}}
\end{equation}
as
\begin{equation}\label{Akronecker}
  {\Ten{L}}_{\ell, \ell'} = \sum_{d=1}^D \Bigl(\bigotimes_{i=1}^{d-1} \hat{\Ten{E}}_{\ell,\ell'}\Bigr)  \otimes  \hat{\Ten{L}}_{\ell,\ell'} \otimes \Bigl(\bigotimes_{i=d+1}^D \hat{\Ten{E}}_{\ell,\ell'} \Bigr).
\end{equation}
The matrices $\hat{\Ten{L}}_{\ell, \ell'}$ for $\ell > \ell'$ can be written in terms of $\hat{\Ten{L}}_{\ell',\ell'}$ as follows: since $\vphi{\ell}{j}'$ for $j\in \hat{\cJ}_\ell$ with $j < 2^\ell$ are $\LSp{2}$-orthogonal to constants, the inner products of these functions with $\vphi{\ell'}{j'}'$, $j'\in \hat{\cJ}_{\ell'}$, can be nonzero only when $j = 2^{\ell - \ell'} j'$.
For $\ell \geq \ell'$, we thus define
$\hat{\Ten{\varXi}} \in \R^{ \hat{\cJ}_\ell \times \hat{\cJ}_{\ell'} }$ by
\[
 \Par[1]{ \hat{\Ten{\varXi}}_{\ell,\ell'} }_{j \; j'} = \delta_{j , \, 2^{\ell - \ell'} j'}
 \quad\text{for all}\quad
 j \in \hat{\cJ}_\ell, \, j'\in \hat{\cJ}_{\ell'}
 \, .
\]
Additionally taking into account the difference in $\LSp{2}$-normalization factors between levels $\ell$ and $\ell'$, we obtain
\[
 \hat{\Ten{L}}_{\ell, \ell'} =2^{-\frac12\abs{\ell'-\ell}} {\hat{\Ten{\varXi}}}_{\ell,\ell'} \hat{\Ten{L}}_{\ell',\ell'}\;.
\]

Let $\hat\cV_{k,\ell} = \ran \hat{\Ten{P}}_{k,\ell} \subseteq \R^{\hat{\cJ}_\ell}$ and $\hat{\cV}_k = \ran \hat{\Ten{ P}}_{k,L}$. For $k\leq \ell$ and $\bw \in \hat\cV_{k,\ell}$, let $w \in V_k$ be the function represented by $\bw$. Then by \eqref{eq:LEexpr} and the standard inverse estimate for $V_k$ (see, e.g., \cite[Sec.\ 8.8.3]{Hackbusch:book:2017}), we have $\langle \hat{\Ten{L}}_{\ell,\ell} \bw , \bw\rangle = 2^{-2\ell} \abs{w}_{\SSp[0]{1}}^2 \lesssim 2^{2(k-\ell)} \norm{w}_{\LSp{2}}^2$, and thus
\begin{equation}
\label{eq:bernstein}
	\langle \hat{\Ten{L}}_{\ell,\ell} \bw , \bw\rangle \leq 2^{2(k-\ell)} \norm{\bw}_2^2 ,\qquad
	  k \leq \ell, \;\bw \in \hat\cV_{k,\ell};
\end{equation}
in particular, we also have $\norm{\hat{\Ten{L}}_{\ell,\ell}}_{2\to 2} \lesssim 1$.
Moreover, one has
\begin{equation}
\label{eq:bernsteinmod}
	\langle \hat{\Ten{\varXi}}_{\ell,\ell'} \hat{\Ten{L}}_{\ell',\ell'} \hat{\Ten{\varXi}}_{\ell,\ell'}^\MT\, \bw , \bw\rangle \leq 2^{k-\ell + \min\{k-\ell',0\}} \norm{\bw}_2^2, \qquad
	  k,\ell'\leq \ell, \; \bw \in \hat\cV_{k,\ell}.
\end{equation}
To see this, denote again by $w \in V_k$ the function represented by $\bw$, and consider first $\ell' \leq k \leq \ell$. Then $\tilde{\bw} := \hat{\Ten{\varXi}}_{\ell,\ell'}^\MT\, \bw$ corresponds to evaluations of $w$ on the grid of level $\ell'$, which is coarser than the one on which it is piecewise linear, and consequently $2^{\ell - k} \sum_{j' \in \hat{\cJ}_{\ell'}} \abs{\tilde{\bw}_{j'}}^2 \lesssim  \sum_{j \in \hat{\cJ}_{\ell} } \abs{\bw_j}^2$. Thus $\norm{\tilde{\bw}}_2 =\norm{\hat{\Ten{\varXi}}_{\ell,\ell'}^\MT\, \bw}_2 \lesssim 2^{\frac12(k-\ell)} \norm{\bw}_2$, and \eqref{eq:bernsteinmod} follows in this case.  If $k < \ell' \leq \ell$, $\tilde{\bw} \in \hat\cV_{k,\ell'}$ corresponds to a reinterpolation of $w$ that is still on a finer level than $k$, and thus $\norm{\tilde{\bw}}_2 \leq 2^{\frac12(\ell'-\ell)} \norm{\bw}_2$. Using \eqref{eq:bernstein}, we thus obtain
$\langle \hat{\Ten{L}}_{\ell',\ell'} \tilde{\bw},\tilde{\bw} \rangle
 \lesssim 2^{2(k-\ell')} \norm{\tilde{\bw}}_2^2 \lesssim 2^{2k - 2\ell' + \ell' - \ell} \norm{\bw}_2^2$, which gives \eqref{eq:bernsteinmod}.

We next show that
\begin{equation}\label{eq:projkl}
   \norm{ \hat{\Ten{P}}_{\ell,L}^\MT \hat{\Ten{P}}_{k,L} - \hat{\Ten{P}}_{k,\ell} }_{2\to 2} \lesssim 2^{\frac12 (k - \ell)}, \quad k \leq \ell.
\end{equation}
Let $s_{j\,i} := (\hat{\Ten{P}}_{\ell,L}^\MT \hat{\Ten{P}}_{k,L})_{j\,i}$, $v_{j\,i} := (\hat{\Ten{P}}_{k,\ell})_{j\,i}$, $j \in \hat{\cJ}_{\ell}$, $i \in \hat{\cJ}_{k}$.
Recalling \eqref{Eq:PartFct}, and taking into account that $\supp \vphi{\ell}{j} = [2^{-\ell} (j-1), 2^{-\ell} (j+1)]\cap [0,1]$,
\[
  s_{j\,i} = 2^{-L} \sum_{n = 2^{L-\ell} (j - 1) }^{\min\{2^{L-\ell} (j+1), 2^L\}} \vphi{\ell}{j}(2^{-L}n)\, \vphi{k}{i}(2^{-L}n),
   \qquad
  v_{j\,i} =  2^{-\frac12 \ell}\vphi{k}{i}(2^{-\ell} j).
\]
Whenever $\vphi{k}{i}$ is linear on $\supp \vphi{\ell}{j}$, one has $s_{j\,i}=v_{j\,i}$ by the symmetries in the summation in $s_{j\,i}$. This fails to hold only when $j = 2^{\ell - k} i$. In these cases, one easily verifies that $\abs{ s_{j\,i} - v_{j\,i} } \lesssim 2^{\frac32 (k-\ell)}$ when $j < 2^\ell$ and $\abs{ s_{j\,i} - v_{j\,i} } \lesssim 2^{\frac12 (k-\ell)}$ for $i=2^k$, $j = 2^\ell$, with $L$-independent constants. Using interpolation to bound $\norm{ \hat{\Ten{P}}_{\ell,L}^\MT \hat{\Ten{P}}_{k,L} - \hat{\Ten{P}}_{k,\ell} }_{2\to 2}$ by the corresponding row- and column-sum norms, where the number of nonzero entries in each row and column is uniformly bounded, we obtain \eqref{eq:projkl}.

Note that for any $\bw \in \hat\cV_{k}$ there exists a unique $\Vec{z}\in \R^{\hat{\cJ}_{k}}$ such that $\bw = \hat{\Ten{P}}_{k,L} \Vec{z}$, where $\norm{\bw}_2 \sim \norm{\Vec{z}}_2$ with constants independent of $k$, $L$.
As a consequence, using this with \eqref{eq:projkl}, we obtain
$
   \norm{ \hat{\Ten{P}}_{\ell,L}^\MT \bw - \hat{\Ten{P}}_{k,\ell} \Vec{z} }_2 \lesssim  2^{\frac12 (k - \ell)} \norm{\bw}_2
$ for such $\bw$ and $\Vec{z}$.
Since
\begin{multline*}
  \langle \hat{\Ten{L}}_{\ell,\ell} \hat{\Ten{ P}}_{\ell,L}^\MT\bw, \hat{\Ten{ P}}_{\ell,L}^\MT \bw\rangle
    = \langle  \hat{\Ten{L}}_{\ell,\ell} \hat{\Ten{P}}_{k,\ell} \Vec{z}, \hat{\Ten{P}}_{k,\ell} \Vec{z} \rangle
       + 2 \langle \hat{\Ten{L}}_{\ell,\ell} (\hat{\Ten{ P}}_{\ell,L}^\MT\bw - \hat{\Ten{P}}_{k,\ell} \Vec{z}) , \hat{\Ten{P}}_{k,\ell} \Vec{z}\rangle  \\
       +  \langle \hat{\Ten{L}}_{\ell,\ell} (\hat{\Ten{ P}}_{\ell,L}^\MT\bw - \hat{\Ten{P}}_{k,\ell} \Vec{z}) , (\hat{\Ten{ P}}_{\ell,L}^\MT\bw - \hat{\Ten{P}}_{k,\ell} \Vec{z})  \rangle,
\end{multline*}
using \eqref{eq:bernstein} for $\hat{\Ten{P}}_{k,\ell} \Vec{z} \in \hat\cV_{k,\ell}$, $\norm{\hat{\Ten{L}}_{\ell,\ell} }_{2\to 2}\lesssim 1$, and the Cauchy--Schwarz inequality for the middle term on the right, we obtain
\[
   \langle \hat{\Ten{L}}_{\ell,\ell} \hat{\Ten{ P}}_{\ell,L}^\MT\bw, \hat{\Ten{ P}}_{\ell,L}^\MT \bw\rangle
     \lesssim ( 2^{2 (k-\ell)}  + 2^{\frac32 (k-\ell)} + 2^{k-\ell} ) \norm{\bw}_2^2
   \lesssim 2^{k-\ell}  \norm{\bw}_2^2,
\]
and similarly, using \eqref{eq:bernsteinmod} in the same manner,
\[
\langle  \hat{\Ten{\varXi}}_{\ell,\ell'} \hat{\Ten{L}}_{\ell',\ell'} \hat{\Ten{\varXi}}_{\ell,\ell'}^\MT\, \hat{\Ten{ P}}_{\ell,L}^\MT\bw, \hat{\Ten{ P}}_{\ell,L}^\MT \bw\rangle
  \lesssim 2^{k - \ell} \norm{\bw}_2^2,
\]
for any $\bw \in \hat\cV_k$, $k\leq \ell$.

Consequently, with $0\leq k\leq \ell$, $0\leq k'\leq \ell'$, $\ell\leq \ell'$, for all $\bw_k \in \hat{\cV}_k$ and $\bw_{k'} \in \hat{\cV}_{k'}$,
\begin{equation}\label{offdiagestimate}
	\begin{aligned}
  \Abs[1]{
		\IProd[1]{
			\hat{\Ten{L}}_{\ell,\ell'}
			\hat{\Ten{ P}}_{\ell',L}^\MT \, \bw_{k'}
		}{
			\hat{\Ten{ P}}_{\ell,L}^\MT \, \bw_k
		}
	}	&=
			2^{-\frac12 \abs{\ell'-\ell}}
			\;
			\Abs[1]{
				\IProd[1]{
					\hat{\Ten{L}}_{\ell',\ell'}
					\hat{\Ten{ P}}_{\ell',L}^\MT \, \bw_{k'}
				}{
					{\hat{\Ten{\varXi}}}_{\ell,\ell'}^\MT \hat{\Ten{ P}}_{\ell,L}^\MT \bw_k
				}
			}
			\\
			&\leq
			2^{-\frac12 \abs{\ell'-\ell}}
			\;
			\begin{multlined}[t]
				\IProd[1]{
					\hat{\Ten{L}}_{\ell',\ell'} \hat{\Ten{ P}}_{\ell',L}^\MT\bw_{k'}
				}{
					\hat{\Ten{ P}}_{\ell',L}^\MT\bw_{k'}
				}^{\frac12}
				\\
				\times
				\IProd{
					\hat{\Ten{L}}_{\ell',\ell'} {\hat{\Ten{\varXi}}}_{\ell,\ell'}^\MT \, \hat{\Ten{ P}}_{\ell,L}^\MT \bw_k
				}{
					{\hat{\Ten{\varXi}}}_{\ell,\ell'}^\MT \hat{\Ten{ P}}_{\ell,L}^\MT \bw_k
				}^{\frac12}
			\end{multlined}
			\\
			&\leq
			2^{-\frac12 \abs{\ell'-\ell}}
			\;
			2^{\frac12 (k'-\ell')}
			\norm{\bw_{k'}}_2
			\;
			2^{\frac12 (k-\ell)}
			\norm{\bw_{k}}_2
			\, .
	\end{aligned}
\end{equation}
By~\eqref{Akronecker}, since $\norm{\hat{\Ten{E}}_{\ell,\ell'}}_{2\to 2} \leq 1$, this implies \eqref{combinedoffdiagest}.
\end{proof}

\begin{proof}[Proof of Theorem \ref{Th:Prec}]
Theorem \ref{thm:bpx} implies in particular that $\langle \Ten{C}_{2,L} \bv,\bv\rangle \sim \langle \Ten{D}_L^{-1} \bv,\bv\rangle$ for all $\bv$, that is,
\begin{equation}\label{stdbpxinv}
  \langle \Ten{D}_L^{-1} \bv,\bv\rangle \sim \sum_{\ell=0}^L \bignorm{2^{-\ell} \Ten{ P}_{\ell,L}^\MT \Ten{v} }_2^2.
\end{equation}
We use this in the following proof of the lower bound in \eqref{bpxtobeshown}, which is inspired by arguments using frame theory from \cite{MR2425155}.
Let $
 \bar\cV_L = \bigtimes_{\ell=0}^L \R^{\cJ_\ell}
$. We consider the mappings
$\Ten{F}\colon \cV_L \to \bar\cV_L$ and $\Ten{F}^\MT \colon \bar\cV_L \to \cV_L$ given by
\[
 \Ten{F}\colon \bv \mapsto \bigl(   2^{-\ell} \Ten{ P}_{\ell,L}^\MT \bv \bigr)_{\ell=0,\ldots,L},
 \qquad
  \Ten{F}^\MT\colon (\bv_\ell)_{\ell=0,\ldots,L} \mapsto \sum_{\ell=0}^L 2^{-\ell} \Ten{ P}_{\ell,L} \bv_\ell.
\]
For any $\bw = (\bw_\ell)_{\ell=0,\ldots,L} \in\ran \Ten{F}$, where $\bw = \Ten{F}\bv$ for $\bv \in \cV_L$, we obtain
\[
 \norm{\Ten{F}^\MT \bw}_{\Ten{D}_L} = \sup_{\Vec{z}\neq 0} \frac{\langle \Ten{F}^\MT \bw, \Vec{z}\rangle}{\norm{\Vec{z}}_{\Ten{D}_L^{-1}}} =
  \sup_{\Vec{z}\neq 0} \frac{\langle\Ten{F} \bv , \Ten{F}\Vec{z}\rangle}{\sqrt{\langle \Ten{D}_L^{-1}\Vec{z},\Vec{z}\rangle}} \sim \norm{\Ten{F} \bv }_2 = \norm{\bw}_2
\]
by \eqref{stdbpxinv}.
Now let $\Ten{G} \colon \cV_L \to \bar\cV_L, \, \Vec{v} \mapsto \bigl( \Ten{ P}_{\ell,L}^\MT \Vec{v}\bigr)_{\ell=0,\ldots,L}$. Then $\ran \Ten{G} \subset \ran \Ten{F}$, and thus
\[
   \langle \Ten{C}_{L} \Ten{D}_{L} \Ten{C}_{L}\Vec{v},\Vec{v} \rangle
 = \norm{ \Ten{F}^\MT \Ten{G} \Vec{v}}_{\Ten{D}_L}^2
      \sim \norm{\Ten{G} \Vec{v}}_2^2 \gtrsim \norm{\Vec{v}}_2^2,
\]
which shows the lower bound in \eqref{bpxtobeshown}.

Arguing along similar lines to obtain the upper bound in \eqref{bpxtobeshown} would lead to a constant depending linearly on $L$, and we thus now turn to a different approach using Lemma \ref{Lm:strengthenedCS}.
Let
$\Ten{R}_\ell = \Ten{ P}_{\ell,L} (\Ten{ P}_{\ell,L}^\MT \, \Ten{ P}_{\ell,L})^{-1}\Ten{ P}_{\ell,L}^\MT$
be the discrete orthogonal projector onto $\cV_\ell$.
For any $\Vec{w} \in \cV_L$, setting
$\Vec{w}_0 = \Ten{R}_0 \QQ \Vec{w}$
and
$\Vec{w}_\ell = (\Ten{R}_\ell - \Ten{R}_{\ell-1}) \, \Vec{w}$
for $\ell = 1,\ldots,L$,
we obtain the decomposition
\begin{equation}\label{wdecomp}
  \Vec{w} =\sum_{\ell=0}^L \Vec{w}_\ell
  \quad\text{with}\quad
  \norm{\Vec{w}}_2^2 =  \sum_{\ell=0}^L \norm{\Vec{w}_\ell}_2^2
  \, ,
\end{equation}
which yields
\[
  \langle \Ten{C}_L \Ten{D}_L \Ten{C}_L \bw,\bw\rangle
    = \sum_{\ell,\ell'=0}^L \Bigl\langle {\Ten{L}}_{\ell,\ell'} \Ten{ P}_{\ell',L}^\MT \sum_{k'=0}^{\ell'} \bw_{k'},\; \Ten{ P}_{\ell,L}^\MT \sum_{k=0}^\ell \bw_k \Bigr\rangle .
\]
For $n = 0,1,\ldots,L$, by Lemma \ref{Lm:strengthenedCS},
\begin{multline*}
	\sum_{\ell=0}^{L-n}
	\IProd[2]{
		{\Ten{L}}_{\ell,\ell+n}
		\Ten{ P}_{\ell+n,L}^\MT
		\sum_{k'=0}^{\ell+n} \bw_{k'}
	}{
		\Ten{ P}_{\ell,L}^\MT
		\sum_{k=0}^\ell \bw_k
	}
	\\
	\lesssim
	2^{-\frac12 n}
	\sum_{\ell=0}^{L-n}
	\sum_{k'=0}^{\ell+n}
	\sum_{k=0}^\ell
	2^{\frac12 (k'-\ell-n)}
	\,
	2^{\frac12 (k-\ell)}
	\,
	\norm{\bw_k}_2
	\,
	\norm{\bw_{k'}}_2
	\\
	\leq
	2^{-\frac12 n }
	\sum_{\ell=0}^{L-n}
	\CuBr[4]{
		\,
		\sum_{k'=0}^{\ell+n}
		2^{\frac12 (k'-\ell-n) }
		\,
		\norm{\bw_{k'}}_2^2
		+
		\sum_{k=0}^\ell
		2^{\frac12 (k-\ell) }
		\,
		\norm{\bw_k}_2^2
	}
	\, .
\end{multline*}
We thus arrive at
\[
  \langle \Ten{C}_L \Ten{D}_L \Ten{C}_L \bw,\bw\rangle
   \lesssim  \sum_{n=0}^L 2^{-\frac12 n} \sum_{\ell=0}^{L} \norm{\bw_\ell}_2^2
    \lesssim \norm{\bw}_2^2,
\]
completing the proof of the upper bound in \eqref{bpxtobeshown} and hence of Theorem \ref{Th:Prec}.
\end{proof}

\begin{remark}
Although we have used some simplifications due to the tensor structure in our particular setting, the proof of Theorem \ref{thm:bpxsymm} carries over to more general hierarchies of finite element spaces, provided that one can establish a corresponding strengthened Cauchy--Schwarz inequality as in \eqref{combinedoffdiagest}, see, e.g., \cite{MR1213412,MR853662,MR1189535}.
\end{remark}

\section{Rank-Reduced Decomposition}\label{App:RR}

The following proof of Lemma~\ref{Lm:DecMP} relies on properties of
the strong Kronecker product inherited from the matrix and Kronecker products:
linearity, associativity, and distributivity.
In particular, products of cores can be transformed
into products of smaller cores
by eliminating linear dependence from
the decomposition, as the following example illustrates.

	For any scalar coefficients $\alpha$, $\beta$ and
	blocks \emph{or subcores} $V_{11}$, $V_{12}$, $V_{21}$, $V_{22}$, $W_{11}$, $W_{12}$
	of suitable rank and mode size, we have
	\begin{subequations}
	\begin{gather}
		\begin{bmatrix}
			V_{11}	& V_{12}	\\
			V_{21}	& V_{22}	\\
		\end{bmatrix}
		\RP
		\begin{bmatrix}
			\alpha W_{11}	& \alpha W_{12}	\\
			\beta W_{11}	& \beta W_{12}	\\
		\end{bmatrix}
		=
		\begin{bmatrix}
			V_{11}	&V_{12}	\\
			V_{21}	&V_{22}	\\
		\end{bmatrix}
		\RP
		\left(
		\begin{bmatrix}
			\alpha
			\\
			\beta
			\\
		\end{bmatrix}
		\RP
		\begin{bmatrix}
			W_{11}
			&
			W_{12}
			\\
		\end{bmatrix}
		\right)
		\label{Eq:CoreTransformation1}
		\\
		=
		\begin{bmatrix}
			V_{11}	& V_{12}	\\
			V_{21}	& V_{22}	\\
		\end{bmatrix}
		\RP
		\begin{bmatrix}
			\alpha	\\
			\beta	\\
		\end{bmatrix}
		\RP
		\begin{bmatrix}
			W_{11}	& W_{12}	\\
		\end{bmatrix}
		\label{Eq:CoreTransformation2}
		\\
		=
		\left(
		\begin{bmatrix}
			V_{11}	& V_{12}	\\
			V_{21}	& V_{22}	\\
		\end{bmatrix}
		\RP
		\begin{bmatrix}
			\alpha	\\
			\beta	\\
		\end{bmatrix}
		\right)
		\RP
		\begin{bmatrix}
			W_{11}	& W_{12}	\\
		\end{bmatrix}
		=
		\begin{bmatrix}
			\alpha V_{11} + \beta V_{12}	\\
			\alpha V_{21} + \beta V_{22}	\\
		\end{bmatrix}
		\RP
		\begin{bmatrix}
			W_{11}
			&
			W_{12}
			\\
		\end{bmatrix}
		\, .
		\label{Eq:CoreTransformation3}
	\end{gather}
	\end{subequations}
	When the partitioning shown
	in~\eqref{Eq:CoreTransformation1}--\eqref{Eq:CoreTransformation3}
	is in terms of blocks (which, by our identification convention, are subcores of rank $1 \times 1$),
	the rank of the product is
	$2 \times 2$.
	The left-hand side of~\eqref{Eq:CoreTransformation1}
	and the right-hand side of~\eqref{Eq:CoreTransformation3}
	represent this core ``in the TT format'', which has only
	only one rank parameter and
	happens to be
	nothing else than low-rank matrix factorization in these two cases.
	The ``ranks'' of the first decomposition, equal to $2$,
	are larger than the ``ranks'' of the last decomposition, equal to $1$.

	The TT representation~\eqref{Eq:CoreTransformation2}
	consists of three cores and has
	ranks $2,1$. However, all mode indices of its middle core are dummy indices
	(the mode size of the middle core is $1 \times 1$),
	so the middle core can be merged with either of the neighboring cores
	\emph{without changing the decomposition scheme}
	(by the latter we mean the set and the ordering of the
	variables separated by the TT format).

\begin{proof}[Proof of Lemma~\ref{Lm:DecMP}]
	\begin{subequations}
	Let
		$\hat{\Ten{N}}_{\ell \CQ L \CQ \alpha}
		=
		\hat{\Ten{M}}_{L \CQ \alpha}
		\,
		\hat{\Ten{P}}_{\ell,L}
		$ and $
		c_{\ell,L}
		=
		2^{ \Par{\alpha+\frac12} L - \frac12 \Par{L-\ell} }
		$.
	Applying Lemma~\ref{Lm:DecM} with the same $\ell$ as fixed here,
	we obtain
	\begin{equation}\label{Eq:DecM1}
		\hat{\Ten{M}}_{L \CQ \alpha}
		=
		2^{ \Par{\alpha+\frac12} L }
		\;
		\hat{\Core{A}}
		\RP
		\hat{\Core{U}}^{\RP \ell}
		\RP
		\hat{\Core{T}}_{0}
		\RP
		\hat{\Core{V}}^{\RP \Par{L-\ell}}
		\RP
		\hat{\Core{I}}
		\RP
		\hat{\Core{M}}_\alpha
		\, ,
	\end{equation}
	where $\hat{\Core{I}}$ is as defined in~\eqref{Eq:DefCoresZZZ4}.
	On the other hand, Lemma~\ref{Lm:DecP} gives the decomposition
	\begin{equation}\label{Eq:DecP1}
		\hat{\Ten{P}}_{\ell,L}
		=
		2^{ -\frac12 \Par{L-\ell} }
		\;
		\hat{\Core{A}}
		\RP
		\hat{\Core{U}}^{\RP \ell}
		\RP
		\hat{\Core{I}}
		\RP
		\hat{\Core{X}}^{\RP \Par{L-\ell}}
		\RP
		\hat{\Core{P}}
		\RP
		\begin{bmatrix}
			1
		\end{bmatrix}
		\, .
	\end{equation}
	Rewriting
	matrix multiplication core-wise,
	we combine the rank-two decompositions
	given by~\eqref{Eq:DecM1}--\eqref{Eq:DecP1}
	into a rank-four decomposition for the product:
	\begin{equation}\label{Eq:DecMPdiesis}
		\hat{\Ten{N}}_{\ell \CQ L \CQ \alpha}
		=
		c_{\ell,L}
		\;
		\hat{\Core{A}}_\flat
		\RP
		\hat{\Core{U}}_\sharp^{\RP \ell}
		\RP
		\hat{\Core{W}}_0
		\RP
		\hat{\Core{Y}}_\sharp^{\RP \Par{L-\ell}}
		\RP
		\Core{E}
		\RP
		\hat{\Core{M}}_\alpha
		\, ,
	\end{equation}
	where
	$\hat{\Core{A}}_\flat$
	and $\hat{\Core{W}}_\alpha$ with $\alpha=0,1$
	are as in~\eqref{Eq:DefCoresBemolle1} and~\eqref{Eq:DefCoresBemolle2}
	and
	$
		\Core{E}
		=
		\hat{\Core{I}} \MP \hat{\Core{P}}
	$,
	$
		\hat{\Core{U}}_\sharp
		=
		\hat{\Core{U}} \MP \hat{\Core{U}}
	$
	and
	$
		\hat{\Core{Y}}_\sharp
		=
		\hat{\Core{V}}	\MP \hat{\Core{X}}
	$
	are newly introduced cores.
	Direct calculation with expressions given in~\eqref{Eq:DefCoresZZZ0}, \eqref{Eq:DefCoresZZZ1}
	and~\eqref{Eq:DefCoresZZZ4} yields
	\begin{equation*}
			\Core{E}
			=
			\begin{bmatrix*}[r]
				1 &    \\
				0 &    \\
					& 1  \\
					& 0  \\
			\end{bmatrix*}
			\,
			,
		\qquad
				\hat{\Core{U}}_\sharp
				=
				\begin{bmatrix*}[l]
						I	& J^{\MT} & J^{\MT}	&				\\
							& J				&					& I_2		\\
							&					&	J				& I_1		\\
							&					&					&				\\
				\end{bmatrix*}
				\, ,
			 \quad
				\hat{\Core{Y}}_\sharp
				=
				\frac14
				\begin{bmatrix*}[r]
					\begin{pmatrix}
						3			\\
						3			\\
					\end{pmatrix}
					&
					\begin{pmatrix}
						1			\\
						1			\\
					\end{pmatrix}
					&
					\begin{pmatrix*}[r]
						-1		\\
						1			\\
					\end{pmatrix*}
					&
					\begin{pmatrix*}[r]
						-1		\\
						1			\\
					\end{pmatrix*}
					\\
					\begin{pmatrix}
						1			\\
						1			\\
					\end{pmatrix}
					&
					\begin{pmatrix}
						3			\\
						3			\\
					\end{pmatrix}
					&
					\begin{pmatrix*}[r]
						1			\\
						-1		\\
					\end{pmatrix*}
					&
					\begin{pmatrix*}[r]
						1			\\
						-1		\\
					\end{pmatrix*}
					\\
					\begin{pmatrix*}[r]
						-1		\\
						3			\\
					\end{pmatrix*}
					&
					\begin{pmatrix*}[r]
						-1		\\
						1			\\
					\end{pmatrix*}
					&
					\begin{pmatrix}
						3			\\
						1			\\
					\end{pmatrix}
					&
					\begin{pmatrix}
						1			\\
						1			\\
					\end{pmatrix}
					\\
					\begin{pmatrix}
						1			\\
						1			\\
					\end{pmatrix}
					&
					\begin{pmatrix}
						1			\\
						3			\\
					\end{pmatrix}
					&
					\begin{pmatrix*}[r]
						1			\\
						-1		\\
					\end{pmatrix*}
					&
					\begin{pmatrix*}[r]
						3			\\
						-1		\\
					\end{pmatrix*}
				\\
			\end{bmatrix*}
	\end{equation*}
	in terms of the blocks $I$, $I_1$, $I_2$ and $J$
	defined in~\eqref{Eq:DefBlocksIJ}.

	\textbf{Sweeping from level $L$ to level $1$.}
	Let us define the following cores:
	\begin{equation}\nonumber
		\Core{C}
		=
		\begin{bmatrix}
			1 &   &   &   \\
			  & 1 &   &   \\
				&   & 1 &   \\
				& 1 &   & 0 \\
		\end{bmatrix}
		\qquad\text{and}\qquad
		\Core{G}
		=
		\begin{bmatrix}
			1 &   &   &   \\
				& 1 &   &   \\
				&   & 1 &   \\
			  &   &   & 0 \\
		\end{bmatrix}
		\, .
	\end{equation}
	First, we note that
	the second and fourth rows in each of the cores
	$\Core{E}$ and $\hat{\Core{Y}}_\sharp \RP \Core{C}$ are equal.
	This implies that
	$
		\Core{E}
		=
		\Core{C}
		\RP
		\Core{E}
	$
	and
	$
		\hat{\Core{Y}}_\sharp
		\RP
		\Core{C}
		=
		\Core{C}
		\RP
		\hat{\Core{Y}}_\sharp
		\RP
		\Core{C}
	$.
	Further, in each of the cores
	$\hat{\Core{W}}_0 \RP \Core{C}$
	and
	$\hat{\Core{U}}_\sharp$, the last row is zero, so that
	$
		\hat{\Core{W}}_0
		\RP
		\Core{C}
		=
		\Core{G}
		\RP
		\hat{\Core{W}}_0
		\RP
		\Core{C}
	$
	and
	$
		\hat{\Core{U}}_\sharp
		=
		\Core{G}
		\RP
		\hat{\Core{U}}_\sharp
	$.
	These equalities allow to sweep the cores $\Core{C}$ and
	$\Core{G}$
	through the
	last $L-\ell$ and
	first $\ell$ levels respectively:
	starting
	from~\eqref{Eq:DecMPdiesis}, we obtain
	\begin{equation}\label{Eq:DecMPdiesis1}
		\begin{aligned}
			\hat{\Ten{N}}_{\ell \CQ L \CQ \alpha}
			&=
			c_{\ell,L}
			\;
			\hat{\Core{A}}_\flat
			\RP
			\hat{\Core{U}}_\sharp^{\RP \ell}
			\RP
			\hat{\Core{W}}_0
			\RP
			\hat{\Core{Y}}_\sharp^{\RP \Par{L-\ell}}
			\RP
			\Core{C}
			\RP
			\Core{E}
			\RP
			\hat{\Core{M}}_\alpha
			\\
			&=
			c_{\ell,L}
			\;
			\hat{\Core{A}}_\flat
			\RP
			\hat{\Core{U}}_\sharp^{\RP \ell}
			\RP
			\hat{\Core{W}}_0
			\RP
			\Core{C}
			\RP
			\Par{
				\hat{\Core{Y}}_\sharp
				\RP
				\Core{C}
			}^{\RP \Par{L-\ell}}
			\RP
			\Core{E}
			\RP
			\hat{\Core{M}}_\alpha
			\\
			&=
			c_{\ell,L}
			\;
			\hat{\Core{A}}_\flat
			\RP
			\hat{\Core{U}}_\sharp^{\RP \ell}
			\RP
			\Core{G}
			\RP
			\hat{\Core{W}}_0
			\RP
			\Core{C}
			\RP
			\Par{
				\hat{\Core{Y}}_\sharp
				\RP
				\Core{C}
			}^{\RP \Par{L-\ell}}
			\RP
			\Core{E}
			\RP
			\hat{\Core{M}}_\alpha
			\\
			&=
			c_{\ell,L}
			\;
			\hat{\Core{A}}_\flat
			\RP
			\Par{
				\hat{\Core{U}}_\sharp
				\RP
				\Core{G}
			}^{\RP \ell}
			\RP
			\hat{\Core{W}}_0
			\RP
			\Core{C}
			\RP
			\Par{
				\hat{\Core{Y}}_\sharp
				\RP
				\Core{C}
			}^{\RP \Par{L-\ell}}
			\RP
			\Core{E}
			\RP
			\hat{\Core{M}}_\alpha
			\, .
		\end{aligned}
	\end{equation}

	\textbf{Sweeping from level $1$ to level $L$.}
	Further, we notice that the cores
	\begin{equation}\nonumber
		\Core{F}
		=
		\begin{bmatrix}
			1 &   &   &   \\
			  & 1 & 1 & 0 \\
		\end{bmatrix}
		\qquad\text{and}\qquad
		\Core{H}
		=
		\begin{bmatrix*}[r]
			1 &  1 &   &   \\
			  & -1 & 1 & 0 \\
		\end{bmatrix*}
	\end{equation}
	satisfy the relations
		$
		\hat{\Core{A}}_\flat
		=
		\hat{\Core{A}}
		\RP
		\Core{F}
	$,
	$
		\Core{F}
		\RP
		\hat{\Core{U}}_\sharp
		\RP
		\Core{G}
		=
		\hat{\Core{U}}
		\RP
		\Core{F}
	$,
	$
		\Core{F}
		\RP
		\hat{\Core{W}}_0
		\RP
		\Core{C}
		=
		\hat{\Core{T}}_{0}
		\RP
		\Core{H}
	$,
	$
		\Core{H}
		\RP
		\hat{\Core{Y}}_\sharp
		\RP
		\Core{C}
		=
		\hat{\Core{Y}}_{0}
		\RP
		\Core{H}
	$
	and
	$
		\Core{H}
		\RP
		\Core{E}
		=
		\hat{\Core{I}}
	$.
	These relations allow to sweep the cores $\Core{F}$ and
	$\Core{H}$
	through the
	first $\ell$ and
	last $L-\ell$ levels respectively:
	continuing~\eqref{Eq:DecMPdiesis}, we derive
	\begin{equation}\label{Eq:DecMPdiesis2}
		\begin{aligned}
			\hat{\Ten{N}}_{\ell \CQ L \CQ \alpha}
			&=
			c_{\ell,L}
			\;
			\hat{\Core{A}}
			\RP
			\Core{F}
			\RP
			\Par{
				\hat{\Core{U}}_\sharp
				\RP
				\Core{G}
			}^{\RP \ell}
			\!
			\RP
			\hat{\Core{W}}_0
			\RP
			\Core{C}
			\RP
			\Par{
				\hat{\Core{Y}}_\sharp
				\RP
				\Core{C}
			}^{\RP \Par{L-\ell}}
			\!
			\RP
			\Core{E}
			\RP
			\hat{\Core{M}}_\alpha
			\\
			&=
			c_{\ell,L}
			\;
			\hat{\Core{A}}
			\RP
			\hat{\Core{U}}^{\RP \ell}
			\RP
			\Core{F}
			\RP
			\hat{\Core{W}}_0
			\RP
			\Core{C}
			\RP
			\Par{
				\hat{\Core{Y}}_\sharp
				\RP
				\Core{C}
			}^{\RP \Par{L-\ell}}
			\RP
			\Core{E}
			\RP
			\hat{\Core{M}}_\alpha
			\\
			&=
			c_{\ell,L}
			\;
			\hat{\Core{A}}
			\RP
			\hat{\Core{U}}^{\RP \ell}
			\RP
			\hat{\Core{T}}_{0}
			\RP
			\Core{H}
			\RP
			\Par{
				\hat{\Core{Y}}_\sharp
				\RP
				\Core{C}
			}^{\RP \Par{L-\ell}}
			\RP
			\Core{E}
			\RP
			\hat{\Core{M}}_\alpha
			\\
			&=
			c_{\ell,L}
			\;
			\hat{\Core{A}}
			\RP
			\hat{\Core{U}}^{\RP \ell}
			\RP
			\hat{\Core{T}}_{0}
			\RP
			\hat{\Core{Y}}_{0}^{\RP \Par{L-\ell}}
			\RP
			\Core{H}
			\RP
			\Core{E}
			\RP
			\hat{\Core{M}}_\alpha
			\\
			&=
			c_{\ell,L}
			\;
			\hat{\Core{A}}
			\RP
			\hat{\Core{U}}^{\RP \ell}
			\RP
			\hat{\Core{T}}_{0}
			\RP
			\hat{\Core{Y}}_{0}^{\RP \Par{L-\ell}}
			\RP
			\hat{\Core{M}}_\alpha
			\, .
		\end{aligned}
	\end{equation}
	This proves the claim in the case of $\alpha=0$
	since $\hat{\Core{M}}_0 = \hat{\Core{N}}_0$ by~\eqref{Eq:DefCoresZZZ1} and~\eqref{Eq:DefCoresZZZ4}.

	\textbf{Sweeping from level $L$ to level $\ell$.}
	In the decomposition~\eqref{Eq:DecMPdiesis2},
	the ranks involved in the core products
	to the right of $\hat{\Core{T}}_{0}$
	(in particular, those bounding the ranks of unfolding matrices
	$\ell,\ldots,L-1+\alpha$)
	are all equal to two.
	To prove the claim, it remains to
	consider the case of $\alpha=1$ and obtain
	a reduced decomposition in which
	those ranks are all equal to one instead of two.
	To this end, we note that
	$
		\hat{\Core{Y}}_{0}
		\RP
		\hat{\Core{M}}_1
		=
		\hat{\Core{M}}_1
		\RP
		\hat{\Core{Y}}_{1}
		=
		\hat{\Core{M}}_1
		\RP
		\hat{\Core{Y}}_{1}
		\RP
		\hat{\Core{N}}_1
	$
	and
	$
		\hat{\Core{T}}_{0}
		\RP
		\hat{\Core{M}}_1
		=
		\hat{T}_1
	$.
	Applying these relations to~\eqref{Eq:DecMPdiesis2},
	we obtain the claim in the case of $\alpha=1$.
	\end{subequations}
\end{proof}

\end{appendix}

\end{document}